\documentclass[11pt,reqno]{amsart}
\usepackage{amsmath, amssymb, amsthm, esint, verbatim, hyperref}
\usepackage{times}
\usepackage{dsfont}

\numberwithin{equation}{section}

\hypersetup{
  pdftitle={Park City Lecture Notes on Rectifiable Reifenberg},
  pdfauthor={Aaron Naber},
  pdfsubject={},
  pdfkeywords={},
  pdfpagelayout=SinglePage,
  pdfpagemode=UseOutlines,
  colorlinks,
  bookmarksopen,
  linkcolor=[rgb]{0,0,0.7},
  urlcolor=[rgb]{0,0,0.4},
  citecolor=[rgb]{0.4,0.1,0},
}

\newtheorem{theorem}{Theorem}[section]
\newtheorem{ctheorem}[theorem]{Conjectural Theorem}
\newtheorem{proposition}[theorem]{Proposition}
\newtheorem{observation}[theorem]{Observation}
\newtheorem{prop}[theorem]{Proposition}
\newtheorem{lemma}[theorem]{Lemma}
\newtheorem{corollary}[theorem]{Corollary}
\theoremstyle{definition}
\newtheorem{definition}[theorem]{Definition}
\theoremstyle{definition}
\newtheorem{exercise}[theorem]{Exercise}
\theoremstyle{remark}
\newtheorem{remark}[theorem]{Remark}
\theoremstyle{remark}
\newtheorem{example}[theorem]{Example}
\theoremstyle{remark}
\newtheorem{note}[theorem]{Note}
\theoremstyle{remark}
\newtheorem{question}[theorem]{Question}
\theoremstyle{remark}
\newtheorem{conjecture}[theorem]{Conjecture}

\newcommand{\haus}{\mathcal{H}}

\newcommand{\cA}{\mathcal{A}}

\newcommand{\cC}{\mathcal{C}}

\newcommand{\cH}{\mathcal{H}}
\newcommand{\cK}{\mathcal{K}}
\newcommand{\cM}{\mathcal{M}}
\newcommand{\cN}{\mathcal{N}}
\newcommand{\cP}{\mathcal{P}}

\newcommand{\cS}{\mathcal{S}}

\newcommand{\Lip}{\mathrm{Lip}}

\newcommand{\Vol}{\text{Vol}}

\newcommand{\dN}{\mathds{N}}

\newcommand{\dQ}{\mathds{Q}}
\newcommand{\dR}{\mathds{R}}

\title{Lecture Notes on Rectifiable Reifenberg for Measures}
\author{Aaron Naber}\thanks{}
\date{\today}

\begin{document}
\begin{abstract}
These series of notes serve as an introduction to some of both the classical and modern techniques in Reifenberg theory.  At its heart, Reifenberg theory is about studying general sets or measures which can be, in one sense or another, approximated on all scales by well behaved spaces, typically just Euclidean space itself.  Such sets and measures turn out not to be arbitrary, and often times come with special structure inherited from what they are being approximated by.

We will begin by recalling and proving the standard Reifenberg theorem \cite{reif_orig}, which says that sets in Euclidean space which are well approximated by affine subspaces on all scales must be homoemorphic to balls.  These types of results have applications to studying the regular parts of solutions of nonlinear equations.  The proof given is designed to move cleanly over to more complicated scenarios introduced later.  

The rest of the lecture notes are designed to introduce and prove the Rectifiable Reifenberg Theorem \cite{ENV}, including an introduction to the relevant concepts.  The Rectifiable Reifenberg Theorem roughly says that if a measure $\mu$ is summably close on all scales to affine subspaces $L^k$, then $\mu=\mu^++\mu^k$ may be broken into pieces such that $\mu^k$ is $k$-rectifiable with uniform Hausdorff measure estimates, and $\mu^+$ has uniform bounds on its mass.  These types of results have applications to studying the singular parts of solutions of nonlinear equations.  The proof given is designed to give a baby introduction to ways of thinking in more modern PDE analysis, including an introduction to Neck regions and their Structure Theory.
\end{abstract}

\maketitle

\maketitle
\tableofcontents

\section{Discussion of Reifenberg Methods and Outline of Notes}

We begin these notes by listing a few types of Reifenberg results which exist, as well as the their primary applications.  We will not try and be overly precise at this point, and will return to details after.  Much more complete introductions are given to those topics discussed in these notes at the beginning of each lecture.\\

\begin{enumerate}
\item {\bf The classical Reifenberg \cite{reif_orig}}.
\begin{enumerate}
\item Background:  Given two sets $S_1$ and $S_2$ we define their Hausdorff distance $d_H(S_1,S_2)\equiv\inf\{\epsilon:S_1\subseteq B_{\epsilon}(S_2) $ and$ S_2\subseteq B_{\epsilon}(S_1)\}$, see \eqref{e:hausdorff_distance} for more.
\item Statement: Assume $S\subseteq B_2(0^n)$ is a set such that for each $B_{r}(x)\subseteq B_2$ there exists an affine subspace $L^k=L^k_{x,r}$ such that $d_H\big(S\cap B_r(x),L\cap B_r(x)\big)<\epsilon(n) r$.  Then $S\cap B_1$ is actually homeomorphic to a $k$-dimensional ball.
\item  Application: This is used to study the manifold structure of the regular sets of minimal surfaces.  \\
\end{enumerate}
\item {\bf Reifenberg Theorem for Metric Spaces by Cheeger-Colding \cite{ChCo_I}.}
\begin{enumerate}
\item Background:  Given two metric spaces $X_1$ and $X_2$ we say their Gromov Hausdorff\footnote{The definition given here is not quite the Gromov Hausdorff distance, but it is uniformly equivalent to it.} distance $d_{GH}(X_1,X_2)<\epsilon$ if there exists $\epsilon$-dense subsets $\{x_1^i\}\subseteq X_1$ and $\{x_2^i\}\subseteq X_2$ such that $|d(x_1^i,x_1^j)-d(x_2^i,x_2^j)|<\epsilon$.
\item Statement:  Assume a metric space $X$ is such that each ball $B_{r}(x)\subseteq B_2(p)$ is Gromov Hausdorff close to Euclidean space: $d_{GH}\big(X\cap B_r(x),B_r(0^n)\big)<\epsilon(n) r$.  Then $X\cap B_1(p)$ is actually homeomorphic to a $k$-dimensional ball.
\item  Application:  This is used to show the manifold structure of the regular sets of limits of spaces with lower Ricci curvature bounds.  See also \cite{DT_snowballs} for a more general study of such spaces.\\
\end{enumerate}
\item {\bf Uniform Rectifiability and Alhfor's regular Measures}
\begin{enumerate}
\item Background:  We say a measure $\mu\subseteq B_1(0^n)$ is Alhfor's regular if $c r^k\leq \mu(B_r(x))\leq C r^k$ for all $x\in \text{supp}\mu$ and $r\leq 2$.  We define the Jones $\beta$-numbers of any measure $\mu$ by \newline $\beta_k(x,r)^2\equiv \inf_{L^k} \int_{B_r(x)}d(x,L^k)^2\,d\mu[x]$, where the inf is taken over all affine subspaces, see Section \ref{s:rect_reif:Jones_beta} for a much more complete introduction.
\item Setup:  For an Alhfor's regular measure $\mu$ one sees that uniform rectifiability\footnote{For the sake of the introduction view $k$-rectifiable as being a $k$-manifold away from a set of measure zero.  Precise definitions and statements are given in  Section \ref{s:rect_reif:haus_content}.  Uniform rectifiablity roughly means that on all balls one can cover most of the support of $\mu$ by a single chart, see \cite{Semmes_UniformRect}.} is equivalent to the measures support being summably close on all scales to affine subspaces: i.e. $\int_0^2\beta_k(x,s)^2\frac{ds}{s}<\delta$, see Jones \cite{jones}, David-Semmes \cite{david-semmes}, and Toro \cite{toro:reifenberg}.
\item The introduction of these ideas were used by Jones to solve the traveling salesman problem, and more advanced refinements were used by David-Semmes to prove estimates on Calderon-Zygmund operators constructed from $\mu$.  Local refinements were used by Tolsa \cite{tolsa_char} and Tolsa-Azzam \cite{azzam-tolsa} to characterize when measures with upper and lower density bounds are rectifiable. \\ 
\end{enumerate}

\item {\bf Rectifiable Reifenberg Theorem by Edelen-Naber-Valtorta \cite{ENV}, Naber-Valtorta \cite{naber-valtorta:harmonic}.}
\begin{enumerate}
\item Statement:  If the support of a general measure $\mu$ is summably close on all scales to affine subspaces $L^k$, i.e. $\int_0^2\beta_k(x,s)^2\frac{ds}{s}<\Gamma$ where $\beta_k$ are the Jones $\beta$-numbers, then $\mu=\mu^++\mu^k$ may be broken into pieces such that $\mu^k$ is $k$-rectifiable with uniform Hausdorff measure estimates and $\mu^+$ has uniform bounds on the measure.
\item  Application:  This can be viewed as effective versions of the previous setup, as one concludes measures bounds instead of assuming them, and is used to study the rectifiable structure and volume bounds of singular sets of nonlinear equations.\\
\end{enumerate}
\item {\bf Reifenberg Theorem to Subset of Subsets.}
\begin{enumerate}
\item  Setup:  Consider a closed subset $\cC$ of the space of all subsets.  Assume for each ball that $S\cap B_r(x)$ is close to some element $C\in\cC$.  Then in many special cases, the set $S$ will inherit special properties itself from $\cC$.  See the work of Badger-Lewis \cite{BaLe_LocalReif}.
\item  Application:  Take $\cC$ to be the zero sets of harmonic polynomials.  Such subsets enjoy a frequency monotonicity, which weakly transfers to the set $S$ itself and builds a certain stratified structure on $S$, see the work of Badger-Engelstein-Toro \cite{BaEnTo_HarmZero}.  See also \cite{DPT_minimalcones} for an application of similar ideas to minimal cones.\\
\end{enumerate}
\item {\bf Canonical Rectifiable Reifenberg Theorems.}
\begin{enumerate}
\item Setup:  Most Reifenberg results rely on the same basic construction to build the Reifenberg maps.  In more complicated situations, as when the underlying space itself is twisted, this construction leads to additional errors and does not allow for rectifiable and finite measure control.  Instead, one builds Reifenberg maps canonically by letting the maps themselves solve an equation. 
\item  Application:  The main application of this is to study the singular sets of spaces with lower Ricci curvature bounds by Cheeger-Jiang-Naber \cite{ChJiNa}.  In that case, to approximate the singular set by a $k$-dimensional Euclidean space also requires approximating the underlying manifold.  These errors are worse and fundamentally not controllable using Reifenberg constructions.  One instead solves for harmonic mappings into $\dR^k$, and proves that on the (approximate) singular sets that these mappings are automatically Reifenberg and even rectifiable Reifenberg. 
\end{enumerate}
\end{enumerate}

These notes will focus primarily on $(1)$ and $(4)$ above.  The outline of these notes is as follows:\\

In Lecture 1 we will study and prove the classical Reifenberg Theorem.  Our proof of the classical Reifenberg Theorem is designed with the rectifiable Reifenberg in mind, so that many of the technical complications which appeared previously in the literature, see \cite{naber-valtorta:harmonic} for instance, may be avoided.  

In Lecture 2 we will give the necessary background needed so that we may end with a statement of the rectifiable Reifenberg Theorem.  This includes an introduction to the Jones $\beta$-number, which measures on a given ball how far away the support of a measure $\mu$ is to being contained in an affine subspace.  The rectifiable Reifenberg theorem roughly states that if a measure $\mu$ has appropriate integral control on its $\beta$-numbers, then it must be decomposable into pieces $\mu=\mu^++\mu^k$, where $\mu^k$ has $k$-rectifiable support with finite Hausdorff measure and $\mu^+$ is a uniformly finite measure.  The proof in these notes is different from \cite{ENV} and has been designed as a baby case of how one approaches singularity analysis in general.

In Lecture 3 we will introduce the notion of Neck regions and state the Neck Structure and Neck Decomposition Theorems.  Neck regions are roughly those regions for which a weak version of a Reifenberg type rigidity hold for $\mu$, and which the techniques of Lecture 1 will apply.  The Neck Decomposition Theorem will tell us how to decompose $B_1$, in a crucially effective way, into pieces which are either Neck regions or into regions which already have mass bounds.  After stating and discussing the Neck Structure and Neck Decomposition Theorems we will use them in Lecture 3 to prove the rectifiable Reifenberg Theorem itself.

In Lecture 4 we will prove the Neck Structure and Neck Decomposition Theorems, thus completing the proof of the rectifiable Reifenberg Theorem.  The proof of the Neck Structure Theorem will follow a very similar line of attack as our proof of the classical Reifenberg in Lecture 1, once some suitable technical complications are addressed.  The proof of the Neck Decomposition Theorem is an involved covering argument, the idea of which originates in the papers \cite{JiNa_L2},\cite{naber-valtorta:harmonic}.\\

It is worth taking a moment to mention that the proof structure of these notes are designed to be as widely applicable as possible.  Our proof of the classical Reifenberg is not just designed to be applied to the rectifiable Reifenberg, but in the process will build a variety of structure which is used in a lot of applications itself. 

Likewise our construction of Neck regions together with the Neck Structure and Neck Decomposition theorem is precisely what appears in many of the recent applications of this type of analysis.  Neck regions first appeared in the proof of the $n-4$ finiteness conjecture for manifolds with bounded Ricci curvature \cite{JiNa_L2}, and the proof of the energy identity conjecture for Yang Mills \cite{NaVa_EnId}.  Neck regions in those cases are quite a bit more subtle (as one cannot directly assume $\beta$-number control), and thus the neck structure theorems there take a lot more work.  However the neck decomposition theorems are almost verbatim.  In both cases the reifenberg context makes for an excellent test case.

\vspace{.5cm}

\section{Required Technical Background for these Notes}

These notes are designed to be almost entirely self-contained.  It goes without saying that the more background one has in some basic geometric measure theory the more comfortable you may feel, but strictly speaking this is not needed as we will build by hand almost all of the structure we require.  The following is meant to list some theorems and basic technical tools which will get used frequently.  The reader is free, and indeed encouraged, to skip this section for now and come back when appropriate, as many of the technical constructions will be easier to follow when there is a context.  The exercises of this section, and indeed these notes as a whole, are meant to be clear to a reader familiar with the ideas involved.  If any tricks are needed, these are basically always stated in a hint, as the goal of the exercises is to familiarize the reader with the basic technical building blocks.

\subsection{Implicit Function Theorem}

The implicit function theorem is typically stated in an ineffective manner, however since we will care about the estimates let us state for convenience the effective result (whose proof is verbatim the implicit function theorem itself):

\begin{theorem}[Implicit Function Theorem]\label{t:background:implicit_function} Let $f:B_2(0^n)\times B_2(0^m)\to \dR^m$ be a $C^1$ function and assume $f(0,0)=0$ and $|\partial_x f(x,y)|, |\partial_y f(x,y) - Id|, |\partial^2f|<\delta$.  Then there exists $g:B_1(0^n)\to \dR^m$ such that $|g|, |\partial_i g|, |\partial_i\partial_j g|<C(n)\delta$ with $f(x,g(x))=0$ for all $x\in B_1(0^n)$.
\end{theorem}

The above not only tells us that the zero set of $f$ is a graphical manifold near $(0,0)$, but also gives good estimates on the structure of that manifold.

\subsection{Elementary Measure Theory}

You will need to understand the definition of a Borel measure.  Some knowledge of the Hausdorff measure is not required, as we will review this in Section \ref{s:rect_reif:haus_content}, but it would be very helpful.

\subsection{Vitali Covering Lemma}\label{sss:background:covering}

We may not directly quote the Vitali Covering lemma, however its proof will be implicit in a lot of constructions.  Let us state the classical result:

\begin{lemma}[Vitali Covering Lemma]\label{l:vitali} Let $\{B_{r_\alpha}(x_\alpha)\}$ be any collection of balls with $r_\alpha\leq A<\infty$.  Then there exists a countable disjoint subcollection $\{B_{r_i}(x_i)\}$ such that $\bigcup B_{r_\alpha}(x_\alpha)\subseteq \bigcup B_{5r_i}(x_i)$.	
\end{lemma}

In general, the reader (and indeed any aspiring analyst) needs to get very comfortable with ways of covering sets by balls in controlled manners.  Let us give a handful of exercises which will help in this direction.

In practice one proves effective content estimates on a well behaved collection of balls by taking the collection, identifying it with a collection of balls in Euclidean space, and then estimating there.  The following exercise teaches us the minimal structure we need on these balls in Euclidean space in order to conclude content estimates: 

\begin{exercise}\label{exer:covering:1}
Let $\{B_{r_i}(x_i)\}\subseteq B_2(0^k)$ be a collection of balls in Euclidean Space $\dR^k$.  Show the following:
\begin{enumerate}
\item[(a)] If $\{B_{r_i}(x_i)\}$ are disjoint then $\sum r_i^k\leq C(k)$.
\item[(b)] If $B_1(0^k)\subseteq \bigcup B_{r_i}(x_i)$ then $C(k)\leq \sum r_i^k$.	
\end{enumerate}
\end{exercise}
\vspace{.2cm}

We will often need to build controlled covers of regions by balls of some predetermined size.  The next exercise is a gentle introduction into how one takes a covering and refines from it a 'well behaved' covering:

\begin{exercise}\label{exer:covering:2}
For $S\subseteq B_1(0^n)$ let $r_x:S\to \dR^+$ with $r_x>r_0>0$ be an assigned radius function.  Then 
\begin{enumerate}
\item[(a)] Show one can choose a maximal subcollection $\{B_{r_i}(x_i)\}\subseteq \{B_{r_x}(x)\}$ such that $\{B_{r_i/5}(x_i)\}$ are disjoint.  Maximal means if $y\in B_1$ then $B_{r_y/5}(y)\cap B_{r_i/5}(x_i)\neq \emptyset$ for some $i$.
\item[(b)] Show $S\subseteq \bigcup B_{r_i}(x_i)$.
\item[(c)] Argue as in the last Exercise to see $\#\{B_{r_i}(x_i)\}\leq N(n,r_0)$.
\end{enumerate}
\end{exercise}
\vspace{.2cm}

Our last exercise is our most technical, however constructions of this type are used almost continuously.  The idea is similar to the last exercise, but we drop our assumed lower bound on the radius and replaced it instead with lipschitz control on the radius function.  The next exercise can be used to build well behaved partitions of unity, and we will later use a very similar construction to do just that:

\begin{exercise}\label{exer:covering:3}
Let $r_x:B_1(0^n)\to \dR$ be a nonnegative radius function with $|\nabla r_x|\leq \tau^{-1}$ for some $\tau>0$.   Let $\cA_0\equiv \{x:r_x=0\}$ and let $\cA_+\subseteq \{r_x>0\}$ a maximal subset such that $\{B_{10^{-3}\tau r_x}(x)\}$ are disjoint.  
\begin{enumerate}
\item[(a)] Show $B_1\subseteq \cA_0\cup \bigcup_{\cA_+} B_{10^{-2}\tau r_x}(x)$.
\item[(b)] Show if $B_{10^{-1}\tau r_x}(x)\cap B_{10^{-1}\tau r_y}(y)\neq \emptyset$ then $10^{-1}r_y\leq r_x\leq 10 r_x$.
\item[(c)] Show for each $y\in B_1$ that $\#\{x\in \cA_+: B_{10^{-1}\tau r_x}(x)\cap B_{10^{-1}\tau r_y}(y)\neq \emptyset\}\leq C(n)$.
\end{enumerate}
\end{exercise}

\vspace{.3cm}

\subsection{Submanifolds of Euclidean Space}\label{ss:background:submanifolds}

As all submanifolds will be built explicitly, one may hobble through these notes without any real previous knowledge of submanifolds of Euclidean space, however the reader will find the learning curve shortened somewhat is time is spent on this first.  Let us very quickly recall a few definitions, and then present some exercises which are relevant to technical constructions which will appear:

\begin{definition}[Submanifolds]
Recall the following:
\begin{enumerate}
\item We call a differentiable map $f:U\subseteq \dR^k\to \dR^n$ an immersion if for each $x\in U$ we have that the linear map $d_xf:\dR^k\to \dR^n$ given by $d_xf[v]\equiv \partial_i f(x) v^i$ is injective.
\item We say a subset $S\subseteq \dR^n$ is a submanifold if for all $y\in S$ $\exists$ a neighborhood $y\in V$ and an immersion $f:U\subseteq \dR^k\to \dR^n$ such that $S\cap V = f(U)$.  We call the pair $(U,f)$ a chart.
\end{enumerate}
\end{definition}

In practice, the submanifolds of these notes will always come from one chart, as we will only be interested in local constructions.  As defined, one might also call $S$ above an embedded submanifold, which is again perfectly acceptable for the constructions of these notes.  We will be interested in tangent spaces of submanifolds:

\begin{definition}[Tangent Spaces]
Let $S\subseteq \dR^n$ be a submanifold and $f:U\subseteq \dR^k\to \dR^n$ a chart with $f(x)=y$, then we call the tangent space to be the affine subspace $T_yS\equiv \{w\in\dR^n: w=y+d_xf[v] \text{ for }v\in \dR^k\}$.	
\end{definition}

Let us give a few useful exercises on tangent spaces:

\begin{exercise}
Show $T_yS$ is independent of the chart.  Namely, if $f':U'\subseteq \dR^k\to \dR^n$ is another chart of $S$ with $f'(x')=y$, then the tangent space defined from $f'$ is the same as that defined from $f$.
\end{exercise}

If $w\in T_yS$ then we define the norm of $w$ by $|w|\equiv |w-y|$.  That is, if we view $T_yS$ as a linear subspace by moving $y$ to the origin and then we define the norm there.

\begin{exercise}
Show that if $y_i\in S\to y\in S$ with $w=y+\lim \frac{y_i-y}{|y_i-y|}$, then $w\in T_yS$ is a unit tangent vector.
\end{exercise}

Let us now consider a few simple examples:

\begin{example}
Let $f:\dR^k\to \dR^n$ be a linear isometric immersion, then the submanifold $f(\dR^k)=S=L$ is an affine subspace. 	In this case, the tangent space $T_xS=L$ is also $L$ for each $x\in S$.
\end{example}

\begin{example}
Let $L\subseteq \dR^n$ be an affine subspace and let $\hat L^\perp$ be the perpendicular subspace\footnote{As a point of notation we will use $L$ to represent affine subspaces and we will put hat's $\hat L$ to represent linear subspaces, i.e. $\hat L$ goes through the origin while $L$ need not.}.  Let $f:L\to \hat L^\perp$ be a smooth mapping, then the graphical map $g:L\to \dR^n$ given by $g(x)=x+f(x)$ is a chart and gives rise to a graphical submanifold $S=g(L)=\text{Graph}(L)$.\\
\end{example}

\subsubsection{}
{\bf Regularity of Submanifolds:\\}

Although the submanifolds of these notes will come from a single chart, it turns out that it may be much more convenient when working locally to build a new chart tailored to the local structure of the submanifold.  This is related to considering notions of regularity for submanifolds.  Locally, every submanifold looks like the last example and can be written as a graph over an affine subspace, thus let us formalize this into a notion of regularity:

\begin{definition}\label{d:graphical_regularity}
We say $S\subseteq \dR^n$ is $(\delta,r)$-graphical if for each $x\in S$ there exists an affine subspace $L_x\subseteq \dR^n$ and $f_x:L_x\to \hat L_x^\perp$ with $r^{-1}|f_x|, |\partial_i f_x|, r|\partial_i\partial_j f_x|\leq \delta$ such that $\text{Graph}(L_x)\cap B_r(x) = S\cap B_r(x)$. 
\end{definition}
\begin{remark}
The factors of $r$ are scale invariant factors.  Thus if we rescale and translate $\dR^n$ so that $B_r(x)\to B_1(0)$	 then $S$ becomes $(\delta,1)$-graphical.
\end{remark}
\begin{remark}
We can also let $\delta_x$ and $r_x$ be functions on $S$ and say $S$ is $(\delta_x,r_x)$-graphical.	
\end{remark}

Let us present some technical exercises which will build an intuition.  If the reader is not familiar with the notion of Hausdorff distance between sets then we refer them to the definition given in \eqref{e:hausdorff_distance}:

\begin{exercise}\label{exer:graphical_subspace:1}
	Let $S$ be $(\delta,r)$-graphical with $f_x:L_x\to \hat L_x^\perp$ a graphing function\footnote{As a point of notation we will use $L$ to represent affine subspaces and we will put hat's $\hat L$ to represent linear subspaces, i.e. $\hat L$ goes through the origin while $L$ need not.}. Then show that $d_H(S\cap B_r(x),L_x\cap B_r(x))<\delta r$.  
\end{exercise}

\begin{exercise}\label{exer:graphical_subspace:2}
	Let $S$ be $(\delta,r)$-graphical with $f_x:L_x\to \hat L_x^\perp$ a graphing function.  Assume $L'_x$ is an affine subspace such that $d_H(L_x\cap B_r(x),L'_x\cap B_r(x))<\delta r$.  Then show there exists a graphing function $f'_x:L'_x\to (\hat L'_x)^\perp$ with $r^{-1}|f'_x|, |\partial_i f'_x|, r|\partial_i\partial_j f'_x|\leq C(n)A\delta$ such that $\text{Graph}(L'_x)\cap B_r(x) = S\cap B_r(x)$.   
	
	Hint:  Let $f'_x$ be the composition of the projection map from $L'_x$ to $L_x$ and $f_x$.
\end{exercise}

The above exercise tells us we have some flexibility on which affine subspaces we pick.  The next exercise tells us that the tangent spaces of graphical submanifolds are well approximated:

\begin{exercise}\label{exer:graphical_subspace:2}
	Let $S$ be $(\delta,r)$-graphical with $f_x:L_x\to \hat L_x^\perp$ a graphing function.  Show for each $y\in S\cap B_r(x)$ that $d_H(T_yS\cap B_r(x), L_x\cap B_r(x))<C(n)\delta r$.
\end{exercise}

\subsubsection{}
{\bf Projections to Submanifolds:\\}

Finally, let us discuss a little the natural projection map associated to each submanifold.  Composing these will form a key technical tool in the construction of the Reifenberg maps later in these notes.

To begin, given an affine subspace $L$ let $\pi_L:\dR^n\to L$ be the projection map to $L$, and let $\hat \pi_{L}:\dR^n\to \hat L$ be the projection map to the associated linear subspace.  We wish to build projection maps to general submanifolds:

\begin{theorem}\label{t:projection_map}
Let $S\subseteq \dR^n$ be a $(\delta,r)$-graphical submanifold.  Then the closest point projection mapping $\pi_S:B_r(S)\to S\subseteq \dR^n$ defined by $\pi_S(x) = \arg\min_{y\in S}\frac{1}{2}|x-y|^2$ is well defined and satisfies
\begin{enumerate}
\item $\pi_S\cap S = Id$, 
\item $|\partial_i\pi_S(y)-\hat\pi_{L_x}|<C(n)\delta$ where $y\in B_r(x)$ with $x\in S$.
\item $r|\partial_i\partial_j \pi_S|<C(n)\delta$.
\end{enumerate}
\end{theorem}
\begin{proof}
We at least outline the proof.  There are actually several approaches to this, including more geometric ones which I personally prefer, however we will outline a proof using the implicit function theorem so we can stick with ideas more consistent with these notes. Thus let $f:L_x\to \hat L_x^\perp$ be a $\delta$-graphing function for $S$ on $B_r(x)$.  Let us consider on $B_{r}(x)$ the function $G: L_x\times \hat L^\perp_x\times L_x \to L_x $ given by \footnote{The partial derivative uses the $'$ notation in order to signify that we are taking the partial derivative in the $y'$ direction.}
\begin{align}
\langle G(y,z,y'), v'\rangle = \frac{1}{2}\partial_{v'}\big(|y-y'|^2+|z-f(y')|^2\big) = \partial_{v'} \frac{1}{2}|(y,z)-(y',f(y'))|^2\, ,	
\end{align}
so that $G(y,z,y')$ is the horizontal derivative of the square distance from $(y,z)\in B_r(x)$ to $(y',f(y'))\in S$.  In particular, if $\pi_S(y,z)=(y',f(y'))$ then we have $G(y,z,y')=0$, and thus for each $(y,z)$ there exists $y'$ such that $G(y,z,y')=0$.  Using our estimates on $f$ let us also see that:
\begin{align}
|\partial_{y} G(y,z,y')+Id|\, ,\;\;|\partial_{z} G(y,z,y')|\, ,\;\;|\partial_{y'} G(y,z,y') - Id| < C(n)\delta\, .	
\end{align}
 We can now use the implicit function theorem \ref{t:background:implicit_function} in order to find $g:B_{r/2}(x)\to L_x$ such that $G(y,z,g(y,z))=0$ with
\begin{align}
	&r^{-1}|g(y,z)-y|\,  ,\;\;|\partial_y g-Id|\, ,\;\;|\partial_z g|\leq C(n)\delta\, ,\;\; r|\partial\partial g|\leq C(n)\delta\, .
\end{align}
For each $(y,z)\in B_{r/2}(x)$ we then have, by the above estimates, that $y'=g(y,z)$ is the unique point such that $G(y,z,g(y,z))=0$, and thus we must have $\pi_S(y,z) = (g(y,z),f(g(y,z)))$.  The estimates on $g$ and $f$ therefore prove the desired estimates on $\pi_S$.\\
\end{proof}

Let us end now with an intuitive exercise:

\begin{exercise}\label{exer:class_reif:approx_subman:3}
Show using $(2)$ that for $x\in B_{r}(S)$ we have $|\pi_r(x)-x|\leq (1+C(n)\delta)d(x,S_r)$.
\end{exercise}

\newpage
\part*{Lecture 1: Classical Reifenberg}

The classical Reifenberg theorem describes sets which may be approximated on all points and scales by affine subspaces.  The basic claim is that such sets must in fact be homeomorphic to Euclidean balls, and thus are quite rigid.  To describe this in more detail let us recall the Hausdorff distance between sets:
\begin{definition}\label{d:hausdorff}
	Let $A,B\subseteq \dR^n$ be subsets, then we define their Hausdroff distance
\begin{align}\label{e:hausdorff_distance}
d_H(A,B) \equiv \inf\{r>0:A\subseteq B_r(B) \text{ and }B\subseteq B_r(A)\}\, .	
\end{align}
\end{definition}

It is worth observing that the Hausdorff distance is a complete metric on closed subsets \cite{Federer_Book}.  

\begin{exercise}
Check the following:
\begin{enumerate}
	\item Let $S\subseteq B_1(0^n)$ be any closed subset, and let $\{x_i\}\in S$ be an $\delta$-dense subset.  Then $d_H(S,\{x_i\})\leq \delta$.
	\item Let $S\subseteq B_1(0^n)\times\{0\}\subseteq B_1(0^n)\times \dR^2$ be any closed subset, and let $S_\delta \equiv S\times S^1(\delta)$, where $S^1(\delta)\subseteq \dR^2$ is the circle of radius $\delta$.  Then $d_H(S,S_\delta)\leq \delta$.
\end{enumerate}
\end{exercise}

By letting $\delta\to 0$ in the above examples one sees that Hausdorff distance certainly does not preserve either upper or lower dimensional bounds in any manner.\\

Let us now define carefully what it means for a set to satisfy the Reifenberg condition:
\begin{definition}
Let $S\subseteq B_2\subseteq \dR^n$ be a closed set:
\begin{enumerate}
\item We define the $L^\infty$ Jones $\beta$-numbers $\beta^\infty_k(x,r)\equiv r^{-1}\inf_{L^k} d_H(S\cap B_r(x),L\cap B_r(x))$, where the inf is taken over all $k$-dimensional affine subspaces $L^k$.
\item $S$ satisfies the $\delta$-Reifenberg condition if for all $x\in S$ with $B_r(x)\subseteq B_2$ we have $\beta^\infty(x,r)<\delta$.	
\end{enumerate}
\end{definition}

It turns out the Reifenberg situation can occur naturally, in one guise or another, in a variety of situations.  Reifenberg was the first to prove that this forces a very strict rigidity on $S$.  In particular, $S$ must be a topological manifold:

\begin{theorem}[Reifenberg's Theorem]\label{t:classical_reif}
	Let $S\subseteq B_2\subseteq \dR^n$ satisfy the $\delta$-Reifenberg condition.  Then for every $0<\alpha<1$ if $\delta<\delta(n,\alpha)$ then $\exists$ $\phi:S\cap B_1(0^n)\to B_1(0^k)$ which is a $C^{\alpha}$-bih\"older map.  Precisely:
	\begin{align}
		\frac{1}{2}|x-y|^{1+\alpha} < |\phi(x)-\phi(y)| < 2|x-y|^{1-\alpha}\, .
	\end{align}
\end{theorem}
\vspace{.25cm}

The first lecture in these notes will focus on describing some basic examples of the above, in particular to see that the result is sharp as stated, and to go through a careful proof.  Our proof is maybe not quite the standard one, and instead is designed so that it will easily generalize later to more complicated situations with minimal additional work.

\vspace{.5cm}

\section{Examples}

We begin in this section by discussing a handful of examples.  We will build up the complexity of these examples with the goal of seeing that the Reifenberg Theorem is sharp.

\begin{example}[Trivial Example]
	Let $L^k$ be a $k$-dimensional subspace and $S= L^k\cap B_2$.  Then clearly $S$ satisfies the $\delta$-Reifenberg condition for every $\delta>0$. $\square$
\end{example}

\begin{exercise}[Graphical] Let $f:L\to L^\perp$ be compactly supported with $f(0)=0$ and $|\nabla f|<\delta$.  If $S=\text{Graph}(f)\cap B_2 = \{(x,f(x)):x\in L\}\cap B_2$ then show that $S$ satisfies the $\delta$-Reifenberg condition.	 $\square$
\end{exercise}

The next example is our main one, and the first nontrivial example.  It will show the sharpness of the biH\"older condition.  Additionally, the proof of the properties of the example will purposefully be done so as to motivate how the general proof of Reifenberg should go.  Roughly, the proof of the general Reifenberg theorem is more or less just a reverse engineering of the following example:

\begin{example}[Snowflake]\label{ex:snowflake}

The construction is in iterative steps, and we will build a sequence of piecewise linear $S_i$ which will converge to our final example.  Let us begin with the iterative construction:\\

{\bf Construction for Iteration:}  Let $\ell_{a,b}$ be the line segment between $a,b\in \dR^2$ and let $\delta\in \dR$.  We will alter $\ell_{a,b}$ to be the union of four line segments
\begin{align}
\ell^\delta_{a,b} \equiv \ell_{a,c}\cup \ell_{c,d}\cup \ell_{d,e}\cup \ell_{e,b}\, ,  
\end{align}
which are well defined by the properties that 
\begin{align}
&\ell_{a,c},\ell_{e,b}\subseteq \ell_{a,b} \text{ with } \ell_{c,d}\cup \ell_{d,e}\cup \ell_{c,e}\text{ forming an oriented isosoceles triangle}	\, ,\notag\\
&|\ell_{a,c}|=|\ell_{c,d}|=|\ell_{d,e}|=|\ell_{e,b}| \text{ with }|\ell_{a,b}^\delta|^2 = (1+\delta^2)|\ell_{a,b}|^2\, ,
\end{align}
where $|\ell|$ is the length of the given line segment.  Let $\theta_\delta$ be the angle between $\ell_{c,d}$ and $\ell_{c,e}$, and let $d_\delta\equiv |\ell_{a,b}|^{-1}d(d,\ell_{a,b})$.  Note that both satisfy
\begin{exercise}
Show $d_\delta,\theta_\delta = O(\delta)$.
\end{exercise}
In particular this then gives
\begin{align}\label{e:snowflake:1}
	d_H(\ell_{a,b},\ell^\delta_{a,b})= O(\delta) |\ell_{a,b}|\, ,\,\,\,\,\, |\ell^\delta_{a,b}|^2 = (1+\delta^2)|\ell_{a,b}|^2\, .
\end{align}

Now our iterative construction is as follow.  Let $a^{0}_{-1}=(-2,0)$, $a^{0}_1 = (2,0)$ with $S_0 = \ell_{a^0_{-1},a^0_{1}}$ the associated interval.  Inductively, if $S_i = \bigcup \ell_{a^i_j,a^i_{j+1}}$ is piecewise linear with $4^{i}$ edges then let 
\begin{align}
	  S_{i+1} &= \bigcup\ell^{\delta}_{a^i_j,a^i_{j+1}} = \bigcup \ell_{a^{i+1}_j,a^{i+1}_{j+1}} \, 
\end{align}
be piecewise linear with $4^{i+1}$ edges.  By \eqref{e:snowflake:1} we can compute the length of any one of these edges by
\begin{align}
|\ell^i| \equiv |\ell^i_{a^i_j,a^i_{j+1}}| = |a^i_{j+1}-a^i_j| = 4\cdot 4^{-i} (1+\delta^2)^{i/2} \equiv 4\cdot 4^{-\alpha i}\, ,	
\end{align}
where $0<\alpha<1$ is given by $4^{-\alpha} = \frac{\sqrt{1+\delta^2}}{4}$.  Note that $\alpha\to 1$ as $\delta\to 0$.  In particular, using \eqref{e:snowflake:1} we can then compute
\begin{align}\label{e:snowflake:2}
	|S_i| = 4\Big(1+\delta^2\Big)^{i/2}\, ,\,\,\,\,\, d_H(S_i,S_{j})\leq O(\delta)\sum_{k=i}^j \Big(\frac{\sqrt{1+\delta^2}}{4}\Big)^{k}\leq O(\delta)\, 4^{-\alpha\,i}=O(\delta)|\ell^i|\, .
\end{align}
Thus the sequence $S_i$ is Cauchy in the Hausdorff topology and hence there exists a limit $S=\lim S_i$, see \cite{Federer_Book} and the remark after \eqref{e:hausdorff_distance}.  Note that the angle between each segment in $S_i$ is given by $\theta_\delta=O(\delta)$, from which one can conclude
\begin{exercise}
For each $x\in S_i$ show that $\beta^\infty_{S_i}(x,|\ell^i|)<O(\delta)$. 	
\end{exercise}
It is then immediate from the above exercise and the Hausdorff estimate in \eqref{e:snowflake:2} that for $i<j$:
\begin{align}
	\beta^\infty_{S_j}(x,|\ell^i|)\leq O(\delta) + \beta^\infty_{S_i}(x,|\ell^i|)\leq O(\delta)\, .
\end{align}
In particular, we get that $S$ is a $O(\delta)$-Reifenberg space.\\

Now that we have built the example, let us study the structure of this example a little.  In particular, imagine what the Reifenberg map $\Phi:[-2,2]\to S$ might look like.  Note already that we know from the volume estimate of \eqref{e:snowflake:2} that $S$ cannot be uniformly bilipschitz to an interval, and thus best case scenario is bih\"older.  In that sense we have already shown that the Reifenberg theorem must be sharp.  However, let us go to the trouble of building the bih\"older Reifenberg map $\Phi:[-2,2]\to S$.  Our construction will mimic the proof of Theorem \ref{t:classical_reif} itself, however will be technically much less involved for the example, and therefore a good place to build intuition for the general case.

Our strategy will be to build maps $\Phi_i:[-2,2]\to S_i$ with uniform bih\"older estimates, and then limit.  Let us begin by considering the projection maps $\pi_i: S_{i+1}\to S_i$.  Note that for $\delta>0$ small this map is clearly well defined and indeed bilipschitz with the estimates:
\begin{align}\label{e:snowflake:3}
	&\big||d\pi_i|(x)-1\big| = \Big|\lim_{y\in S_{i+1}\to x}\frac{|\pi_i(x)-\pi_i(y)}{|x-y|}-1\Big| \leq \sqrt{1+\delta^2}-1 = O(\delta^2)\, ,\notag\\
	&|\pi_i(x)-x| \leq d_\delta|\ell_i| = O(\delta)|\ell_i|\, .
\end{align}
Observe that the first estimate has an $\delta^2$. We will actually not use this square improvement here, however it appears again (crucially) when going from the classical Reifenberg results to the rectifiable Reifenberg results, and thus we emphasize it.  

Now we have defined a mapping from $S_{i+1}$ to $S_i$, so let us compose these mappings in order to define $\Pi_{i,j}:S_i\to S_{j}$ and $\Pi_i=\Pi_{i,0}:S_i\to S_{0}=[-2,2]$ by $\Pi_{i,j}=\pi_{j}\circ\cdots\circ \pi_{i-1}$.  Note that although each map $\Pi_i$ is bilipschitz, we see that the bilipschitz constants are becoming increasing large, so that we cannot hope to preserve that estimate.  If we can show the maps $\Pi_i$ are uniformly bih\"older then $\Phi_i = \Pi_i^{-1}$ are as well.  

So let $x,y\in S_i$ and let $j\leq i$ be the largest $j$ such that $|x-y|\leq |\ell_j|$.  If $j<i$ then we also have $\sqrt{1+\delta^2}|x-y|\geq |\ell_j|$.  Then to estimate $|\Pi_i(x)-\Pi_i(y)|$ we write $\Pi_i = \Pi_{j,0}\circ \Pi_{i,j}$.  We will estimate each factor separately, using different estimates from \eqref{e:snowflake:3}.  First let us write $x_k = \Pi_{i,k}(x)$ and $y_k=\Pi_{i,k}(y)$ where $j<k$, then using the second estimate from \eqref{e:snowflake:3} we can get
\begin{align}
&|x_{k+1}-x_k|=|\pi_{k+1}(x_k)-x_k|\leq \delta |\ell_k| = O(\delta)\, 4^{-\alpha|k-j|}|\ell_j|\, ,\notag\\
&|y_{k+1}-y_k|\leq O(\delta)\, 4^{-\alpha|k-j|}|\ell_j|\, ,
\end{align}
where recall $4^{-\alpha}\equiv \frac{\sqrt{1+\delta^2}}{4}$.
We then have
\begin{align}
&|\Pi_{i,j}(x)-x|\leq \sum_j^{i-1} |x_{k+1}-x_k|	\leq O(\delta) |\ell_j|\leq O(\delta) |x-y|\, ,\notag\\	
&|\Pi_{i,j}(y)-y|\leq O(\delta) |x-y|\, .
\end{align}
Combining and using the triangle inequality gives 
\begin{align}
\Big||\Pi_{i,j}(x)-\Pi_{i,j}(y)|	-|x-y|\Big|\leq |\Pi_{i,j}(x)-x| + |\Pi_{i,j}(y)-y|\leq O(\delta)|x-y|\, .
\end{align}
Note we have proved that if $|x-y|\approx |\ell_j|$, then $\Pi_{i,j}$ is a uniformly bilipschitz map when comparing $x$ and $y$.  To move from $S_j$ to $S_0$ we now rely on the first estimate from \eqref{e:snowflake:3} in order to estimate
\begin{align}
\Big|\Pi_i(x)-\Pi_{i}(y)\Big|&=\Big|\Pi_{j,0}(\Pi_{i,j}(x))-\Pi_{j,0}(\Pi_{i,j}(y))\Big| \notag\\
&\leq (1+\delta^2)^{j/2}|\Pi_{i,j}(x)-\Pi_{i,j}(y)| \leq (1+O(\delta))(1+\delta^2)^{j/2}|x-y|\notag\\
&\leq (1+O(\delta))(1+\delta^2)^{j/2}|\ell_j|\notag\\
& = 4(1+O(\delta))\big(\frac{1+\delta^2}{4}\big)^{j} \equiv 4(1+O(\delta))\Big(\big(\frac{\sqrt{1+\delta^2}}{4}\big)^\beta\Big)^{j}\notag\\
&\leq(1+O(\delta))\Big(4\big(\frac{\sqrt{1+\delta^2}}{4}\big)^j\Big)^{\beta}\leq (1+O(\delta))|\ell_j|^\beta\notag\\
&\leq (1+O(\delta))|x-y|^\beta\, ,
\end{align}
where $\beta<1$ was defined by $\big(\frac{\sqrt{1+\delta^2}}{4}\big)^\beta = \frac{1+\delta^2}{4}$ and thus is as close to $1$ as we wish as $\delta\to 0$.  A verbatim argument shows the opposite inequality, and thus this proves the uniform bih\"older estimate. $\qed$
\end{example}

\vspace{.5cm}

\section{Proof of Reifenberg Theorem}

We will now focus on giving a proof of the classical Reifenberg Theorem.  Our proof is designed to motivate how we will be approaching the more general and challenging cases.  We have also gone to some effort to make the general scheme one which applies in seemingly very different scenarios in geometric analysis, albeit in often much more complicated ways.  

It will be convenient in the construction to make the following notation.  For each $B_r(x)\subseteq B_2$ with $r\geq 10\,d(x,S)$ \footnote{Recall the distance function $d(y,S)\equiv \inf_{x\in S} d(y,x)$.} let us fix a choice of $k$-dimensional subspace $L_{x,r}=L_{x,r}[S]$ satisfying
\begin{align}\label{e:classical_reif:min_subspace}
L_{x,r}\in \arg\min_L d_H(S\cap B_r(x), L\cap B_r(x)) \equiv \beta^\infty_k(x,r)\, .	
\end{align}

Let us discuss some notation which will be in effect throughout these lectures:\\

{\bf Notation: } Given an affine subspace $L^k\subseteq \dR^n$ let $\hat L\subseteq \dR^n$ denote the linear subspace associated to $L$.  

{\bf Notation: } We let $\pi_{L}:\dR^n\to L\subseteq \dR^n$ and $\hat \pi_{L}:\dR^n\to \hat L\subseteq \dR^n$ denote the orthogonal projection maps. 

{\bf Notation: } In the case of the subspaces $L_{x,r}$ we will write the projection maps as $\pi_{x,r}$ and $\hat\pi_{x,r}$.\\

The following exercise is a key observation in the Reifenberg theorem:

\begin{exercise}\label{exer:subspace_close1}
Assume $S\subseteq B_2$ satisfies the $\delta$-Reifenberg condition, and let $B_r(x),B_s(y)\subseteq B_2$ with $r\geq 10 d(x,S)$ and $s\geq 10 d(y,S)$.  For $a\geq 10$ let $B_{a^{-1}s}(y)\subseteq B_{ar}(x)\subseteq B_{a^2s}(y)$ \footnote{This is saying that the balls $B_s(y)$ and $B_r(x)$ are comparable on scale $a\geq 10$.}.  Then show $d_{H}(L_{x,r}\cap B_r(x), L_{y,s}\cap B_r(x)) < C(n)\delta r$ and $||\hat\pi_{x,r}-\hat\pi_{y,s}||< C(n)\delta $, where $||\cdot||$ is the matrix norm. 
\end{exercise}

The above exercise is telling us that the 'best' subspaces $L_{x,r}$ cannot change too quickly from scale to scale, and is the basis for our ability to glue them together in a controlled fashion.\\

The outline of the proof of the Reifenberg Theorem is as follows.  In Section \ref{ss:class_reif:subman_approx} we begin by stating the Submanifold Approximation Theorem \ref{t:class_reif:approx_submanifold}, which builds a family of smooth manifolds $S_r$ which approximate the set $S$ on scale $r$.  One can take $S_1=L$ to be an affine subspace without any loss, and if we consider the sequence $S_i \equiv S_{2^{-i}}$ then a key observation will be that $S_i$ and $S_{i+1}$ are smoothly close.  Thus we can consider the projection maps $\pi_i:S_{i+1}\to S_i$ and we will build our end biH\"older mapping $\Pi:S\to L = S_1$ from $S$ to $k$-dimensional Euclidean space by simply composing the projection maps $\pi_i$.  Compare to Example \ref{ex:snowflake}.\\

The most direct route to understanding the existence of $S_r$ is to use Exercise \ref{exer:subspace_close1} to glue together the subspaces $L_{x,r}$.  One can indeed make this rigorous, but primarily because of how things will work in future sections, we take a different approach.  Instead we will construct a smooth function $\Phi_r:B_2\to \dR$ which behaves like a distance function to $S_r$.  A little more precisely, $\Phi_r$ will be a Morse Bott function whose zero level set $S_r\equiv \Phi^{-1}_r(0)$ will consist of nondegenerate critical points, and thus is a smooth manifold.  Estimates on $\Phi_r$ will then translate to estimates on $S_r$.  \\

In the end this proof strategy requires a little more work than simply gluing together the subspaces $L_{x,r}$, however comes with a key advantage.  In the proof of the Rectifiable Reifenberg Theorem \ref{t:rect_reif}, and in particular in the proof of the associated Neck Structure Theorem \ref{t:neck_structure}, a very similar construction and proof will be needed, however it will be done on a discrete set of balls.  This discreteness can cause a major technical headache, and previous arguments \cite{JiNa_L2} have used quite involved covering arguments to deal with it.  Instead, we will see our approach for the classical Reifenberg Theorem will pass over almost verbatim, with only minimal extra work.  

\vspace{.3cm}

\subsection{Submanifold Approximation Theorem}\label{ss:class_reif:subman_approx}

Now we begin by building a series of smooth manifold approximations to $S$.  

\begin{theorem}[Reifenberg Submanifold Approximation]\label{t:class_reif:approx_submanifold}
	Let $S\subseteq B_2\subseteq \dR^n$ satisfy the $\delta$-Reifenberg condition.  Then for each $0<r<1$ there exists a smooth submanifold $S_r\subseteq B_2$ which satisfies
\begin{enumerate}
\item $d_H(S_r,S) < C(n)\delta\, r$,
\item $d_H(S_r\cap B_s(x),L_{x,s}\cap B_s(x))<C(n)\delta s$ for $s\geq r$.
\item $S_r$ is a $(C(n)\delta, r)$-graphical submanifold, see Definition \ref{d:graphical_regularity}.
\item $\exists$ smooth $\pi_r: B_{r}(S)\to S_r\subseteq \dR^n$ with $\pi_r\cap S_r = Id$, $|\partial_i\pi_r(y)-\hat\pi_{y,10r}|\, , r|\partial_i\partial_j \pi_r|<C(n)\delta$. \footnote{Recall since $\pi_r$ maps $\dR^n$ to $\dR^n$ that $\partial_i\pi_r$ is a matrix, and thus our norm $|\partial_i\pi_r(y)-\hat\pi_{x,10r}|$ is the matrix norm.}  
\item $d_H(S_{r/2},S_r)<C(n)\delta r$ with $|\pi_r(x)-x|<C(n)\delta r$ for $x\in S_{r/2}$.
\item For $x\in S_{r/2}$ and a unit vector $v\in L_{x,r}$ we have $||d\pi_r[v]|-1|<C(n)\delta^2$.  \footnote{Recall $d\pi_r[v] = \partial_i\pi_r v^i\in T_{\pi_r(x)}S_r$.} 
\end{enumerate}
\end{theorem}
\begin{remark}
Note that if we take $S_1$ to be a linear subspace with $\pi_1$ the orthogonal projection map then the above holds with $r=1$.  It will be convenient  to make this choice.	
\end{remark}
\begin{remark}
Observe in $(6)$ the square gain on $\delta$ in the error.  This will not be used in this section, but in the rectifiable Reifenberg, see also \cite{jones,david-semmes,davidtoro}, it is an important gain in order to conclude mass bounds and rectifiable structure.  
\end{remark}

Let us first see that most of the conclusions of Theorem \ref{t:class_reif:approx_submanifold} follow from $(1)$ and $(3)$:

\begin{exercise}
Show $(2)$ follows from $(1)$ and the Reifenberg property of $S$.	
\end{exercise}

\begin{exercise}
Define $\pi_r:B_r(S)\to S_r$ to be the closest point projection map $\pi_r(x) = \arg\min_{y\in S_r} |x-y|$.  Show $(4)$ using $(3)$ and Theorem \ref{t:projection_map}.	
\end{exercise}

\begin{exercise}
Show the first part of $(5)$ follows from $(1)$.   Show the second part of $(5)$ follows from Exercise \ref{exer:class_reif:approx_subman:3}.
\end{exercise}

Thus we see $(1)-(5)$ follow from $(1)$ and $(3)$.  To show $(6)$ follows from $(1)-(5)$ is done in two steps.  First, observe that the tangent spaces of $S_r$ and $S_{r/2}$ must be close:

\begin{exercise}\label{exer:class_reif:approx_subman:2}
Use $(2)$ and $(3)$ to show for $y,z\in B_{2r}(x)$ with $x\in S$ that $d_H(T_yS_r\cap B_r(x),T_yS_r\cap B_r(x))<C(n)\delta r$ and 	$d_H(T_yS_r\cap B_r(x),T_yS_{r/2}\cap B_r(x))<C(n)\delta r$.

Hint: Use $(2)$ and Exercise \ref{exer:graphical_subspace:2}.
\end{exercise}

\begin{exercise}
Use $(2)$ and Exercise \ref{exer:class_reif:approx_subman:2} to show $(6)$.  

Hint:  Observe $1=|v|^2 = |d\pi_r[v]|^2+|d\pi_r[v]-v|^2$ and the Taylor expansion $\sqrt{1-x}\approx 1-\frac{1}{2}x + O(x^2)$ in order to conclude $1\geq |d\pi_r[v]|=\sqrt{1-|d\pi_r[v]-v|^2}\geq 1-C(n)\delta^2$.
\end{exercise}

Finally, we add one more complication to the mix by doing the above slightly less locally:
\begin{exercise}\label{exer:class_reif:approx_subman:1}
Use the last two exercises to show for each $x\in S$ and $y,z\in S_r\cap B_{2r}(x)$ that the unit vector $v\equiv \frac{z-y}{|z-y|}$ satisfies the estimate $\big|\hat\pi_{x,r}[v]-v\big|<C(n)\delta$.  

Hint:  Consider the curve $\gamma(t)=tz+(1-t)y$ and use $(3)$ to show $\big|\frac{d^2}{dt^2}\big(\pi_r(\gamma(t))-\gamma(t)\big)\big|\leq C(n)\delta |z-y|$.  Next use the fundamental theorem to then conclude $\big|\frac{d}{dt}\big(\pi_r(\gamma(t))-\gamma(t)\big)\big|\leq C(n)\delta |z-y|$.  Finally use that $|d\pi_r - \hat\pi_{x,r}|<C(n)\delta$ to conclude the final estimate.
\end{exercise}

\vspace{.3cm}

Let us now see how to prove the Reifenberg Theorem given the Submanifold Approximation Theorem:

\begin{proof}[Proof of Theorem \ref{t:classical_reif} given  Theorem \ref{t:class_reif:approx_submanifold}]

Consider the radii $r_i = 2^{-i}$ and the submanifolds $S_i = S_{r_i}$ from Theorem \ref{t:class_reif:approx_submanifold}.  Let $\pi_i = \pi_{r_i}\cap S_{i+1}:S_{i+1}\to S_i$ be the $\delta$-submersion from Theorem \ref{t:class_reif:approx_submanifold} restricted to $S_{i+1}$.  Our main claims are the following:\\

{\bf Claim  1:  } If $x,y\in S_{i+1}$ with $|x-y|<r_i$ then $\big| |\pi_i(x)-\pi_i(y)| - |x-y| \big|\leq C(n)\delta\, |x-y|$.

{\bf Claim  2:  } If $x,y\in S_{i+1}$ with $|x-y| \geq r_i$ then  $\big| |\pi_i(x)-\pi_i(y)| - |x-y| \big|\leq C(n)\delta\, r_i$.\\ 

To prove Claim 1 consider the straight line $x_t = tx+(1-t)y$ which connects $y$ to $x$.  Let $L_t\equiv T_{x_t}S_r$, then by Exercise \ref{exer:class_reif:approx_subman:2}, Exercise \ref{exer:class_reif:approx_subman:1} and Theorem \ref{t:class_reif:approx_submanifold}.2 we have for each $t\in [0,1]$ that
\begin{align}
\big|\frac{d}{dt}\pi_r(x_t)-(x-y)\big| &= |d\pi_r(x_t)[x-y]-(x-y)\big|\notag\\
&\leq |\big(d\pi_r(x_t)-\hat\pi_{L_t}\big)[x-y]\big|+|\hat\pi_{L_t}[x-y]-(x-y)\big|\notag\\
&\leq C(n)\delta\, |x-y|\, .	
\end{align}
By integrating we in particular have shown Claim 1, indeed we have the stronger estimate $\big| (\pi_i(x)-\pi_i(y)) - (x-y) \big|\leq C(n)\delta\, |x-y|$.  To prove Claim 2 first observe that since $S_{i+1}\subseteq B_{C\delta r_i}(S)\subseteq B_{2C\delta r_i}(S_i)$ we have by Exercise \ref{exer:class_reif:approx_subman:2} the estimate $|\pi(x)-x|, |\pi(y)-y|<C(n)\delta r_i$.  Using the triangle inequality we get $\big|(x-y)-(\pi(x)-\pi(y))\big|\leq C(n)\delta r_i$, which in particular finishes the proof of Claim 2. \footnote{Recall the analyst's convention that $C(n)$ changes from line to line, but is always a dimensional constant.} $\square$\\

With the claim in hand let us build the maps which will connect $S$ to $L$.  Let us first define the maps
\begin{align}
	&\Pi_{i,j} \equiv \pi_{j}\circ\cdots\circ \pi_{i-1}:S_i\to S_j\notag\\
	&\Pi_i\equiv \Pi_{i,0}:S_i\to S_0\equiv L\, ,
\end{align}
where recall that as in the remark following Theorem \ref{t:class_reif:approx_submanifold}  we have taken $S_0$ to be a linear subspace. We claim the following:\\

{\bf Claim 3:  }  Let  $x,y\in S_i$, then $|x-y|^{1+C(n)\delta}\leq |\Pi_{i}(x)-\Pi_{i}(y)|\leq |x-y|^{1-C(n)\delta}$.\\

To prove the claim let $d=d_i\equiv |x-y|$ and define $d_j\equiv |\Pi_{i,j}(x)-\Pi_{i,j}(y)|$.  Note that if $d_j\leq r_j$ then by Claim 1 we have that 
 \begin{align}\label{e:class_reif:proof:1}
 	&|d_{j-1} - d_j|\leq C(n)\delta\, d_j\implies d_{j-1}\leq (1+C(n)\delta)\, d_j\, ,\notag\\
 	&\implies d_{j-1}\leq r_{j-1} \implies d_k<r_k \text{ for all }k\leq j\, .
 \end{align}

On the other hand if $d_j>r_j$, then by Claim 2 we have that 
 \begin{align}\label{e:class_reif:proof:2}
 	|d_{j+1}-d_j|\leq C(n)\delta\, r_j\, .
 \end{align}
In particular, using from above that $d_k>r_k$ for all $k>j$ this then gives 
\begin{align}
&d_{j}\leq  d_i + \sum_j^{i-1}|d_{k+1}-d_{k}|\leq d_i +  C(n)\delta \sum_i^j r_k \leq d_i +C(n)\delta\, d_j\, ,\notag\\
&\implies d_j\leq (1+C(n)\delta) d\, .
\end{align}
 
Now let $j$ be the smallest integer such that $d\leq r_j$.  The above tells us that $d_j\leq (1+C(n)\delta)d$, and then using \eqref{e:class_reif:proof:1} $j-i$ times we obtain
 \begin{align}
|\Pi_{i}(x)-\Pi_{i}(y)| \equiv d_0&\leq (1+C(n)\delta)^{j-i} d  \notag\\
&\leq 	(1+C(n)\delta)^{\ln d} d = d^{1-C(n)\delta} = |x-y|^{1-C(n)\delta}\, ,
\end{align}
which provides one direction of the claim.  The other direction is the same. $\square$\\

To finish the proof of the Reifenberg Theorem \ref{t:classical_reif} we need to simply limit $\Pi_i\to \Pi:S\to L$ by combining Claim 3 with the Ascoli theorem and Theorem \ref{t:class_reif:approx_submanifold}.1.
\end{proof}

\vspace{.3cm}

\subsection{Distance Approximation Theorem and Proof of Theorem \ref{t:class_reif:approx_submanifold}}\label{ss:class_reif:dist_approx}

We have now seen how to prove the Reifenberg Theorem \ref{t:classical_reif} given the Submanifold Approximation Theorem \ref{t:approx_submanifold}.  Our focus now becomes the proof of the Submanifold Approximation Theorem itself.\\

Our basic strategy for the proof of Theorem \ref{t:class_reif:approx_submanifold} will be to build a smooth function $\Phi_r:B_{2r}(S)\to \dR$ which roughly behaves as smooth approximation to the distance function to $S_r$.  In reality, we will build $\Phi_r\geq 0$ first and then define $S_r \equiv \Phi^{-1}_r(0)$ to be the zero level set.  We will see that sufficiently strong estimates hold on $\Phi_r$ in order to conclude our estimates on $S_r$. \\

We will build $\Phi_r$ in two steps.  First, we will build a smoothly varying distribution on $B_2$ which will assign to each $y\in B_2$ a $k$-dimensional affine subspace $L_r(y)$ which acts as a Reifenberg approximation on the scale 
\begin{align}\label{e:class_reif:scale}
\overline r_y \equiv 10 d(y,S)\vee r	\, ,
\end{align}
where recall $s\vee t \equiv \max\{s,t\}$.  This assignment has a variety of useful applications in its own right.  We will then use these affine subspaces to build $\Phi_r$ directly.  We begin with the statement of the subspace selection lemma:

\begin{lemma}[Subspace Selection Lemma]\label{l:class_reif:subspace_selection}
	Let $S\subseteq B_2\subseteq \dR^n$ satisfy the $\delta$-Reifenberg condition with $0<r<1$ fixed and let $\overline r_y$ be from \eqref{e:class_reif:scale}.  Then for each $y\in B_2$ there exists a $k$-dimensional affine subspace $L_{y}$ where if $\hat \pi_y = \hat \pi_{L_y}$ and $m_y\equiv \pi_y[y]$ then:
\begin{enumerate}  
\item $L_{y}$ varies smoothly in $y$ with $\overline r_y|\nabla \hat\pi_{y}|, |\nabla_i m_y - \hat\pi_y|\leq C(n)\delta$ and $r^2_y |\nabla^2 \hat\pi_{y}|, r_y|\nabla^2 m_y|\leq C(n)\delta$.
\item We have $d_H(S\cap B_{10\overline r_y}(y), L_y\cap B_{10\overline r_y}(y))<C(n)\delta \overline r_y$.
\item We have $d_H(L_y\cap B_{10\bar r_y}(y),L_{y,10^5\bar r_y}\cap B_{10\bar r_y}(y))<C(n)\delta$.
\end{enumerate}

\end{lemma}

We will prove the above in the next section, and simply take it for granted now.  Morally, it is nothing more than an averaging procedure, though requires a little technical work to check the details.\\

Given Lemma \ref{l:class_reif:subspace_selection} let us now define our approximate distance function $\Phi_r:B_2\to \dR$ as follows:
\begin{align}\label{e:class_reif:approx_distance}
\Phi_r(y)\equiv \frac{1}{2}d(y,L_y)^2 = \frac{1}{2}|y-\pi_y[y]|^2= \frac{1}{2}|y-m_y|^2\, .
\end{align}

Let us collect together the main properties of this approximate distance function:\\

\begin{theorem}[Approximate Distance Function]\label{t:class_reif:approx_dist}
Let $\Phi_r$ be defined in \eqref{e:class_reif:approx_distance} with $\overline r_y$ from \eqref{e:class_reif:scale}.  Then for each $y\in B_2$ the following is satisfied:
	\begin{enumerate}
	\item For each $x\in S$ and $\ell\in L_{x}\cap B_r(x)$ $\exists!$ $z_\ell\in \hat L^\perp_{x}+\ell$ such that $\Phi_r(z_\ell)=0$.
	\item $\Big||\nabla \Phi_r|^2 - 4\Phi_r\Big|(y)\leq C(n)\delta \Phi_r(y)$.
	\item $|\nabla^2\Phi_r(y)-\hat\pi_{y}^\perp|<C(n)\delta^2$.
	\item $|\nabla^{(k)}\Phi|(y)\leq C(n,k)\delta^2\, r^{2-k}$ for $k\geq 3$.
	\end{enumerate}
\end{theorem}
\begin{remark}
Building a function which satisfies just $(2)$-$(4)$	 is a little easier and one does not need the Subspace Selection Lemma.  Our construction is primarily designed so that $(1)$ is also easily satisfied.  Without $(1)$ defining $S_r=\Phi^{-1}_r(0)$ may not be a reasonable definition.
\end{remark}

We will also prove the above in the next subsection, it is mostly a direct application of the definition of $\Phi_r$ combined with the properties of $L_y$.  In this section we want to use Theorem \ref{t:class_reif:approx_dist} in order to finish the proof of the Submanifold Approximation Theorem \ref{t:class_reif:approx_submanifold}:

\begin{proof}[Proof of Submanifold Approximation Theorem \ref{t:class_reif:approx_submanifold} given the Distance Approximation Theorem \ref{t:class_reif:approx_dist}]

Let us define the set $S_r\equiv \Phi_r^{-1}(0)$.  Let us make some first observations about this set:

\begin{exercise}
Use Lemma \ref{l:class_reif:subspace_selection}.2, Theorem \ref{t:class_reif:approx_dist}.1 and the definition of $\Phi_r$ to show $d_H(S_r,S)<C(n)\delta r$.
\end{exercise}

The exercise thus proves Theorem \ref{t:class_reif:approx_submanifold}.1.  Now we will prove some regularity results on $S_r$, and then use this regularity to prove Theorem \ref{t:class_reif:approx_submanifold}.3.  Both will eventually be consequences of the implicit function theorem.

  Let $x\in S$ and let us write $B_{2r}(x)$ in coordinates $(y,z)$ where $y\in L_x$ and $z\in \hat L^\perp_x$, where $L_x$ is the subspace given in Lemma \ref{l:class_reif:subspace_selection}.  Let us consider the derivative mapping $F:L_x\times \hat L_x^\perp\to L_x^\perp$ given by
\begin{align}
\langle F(y,z), w\rangle = \partial_{w} \Phi_r(y,z)\, .	
\end{align}

Using the definition of $\Phi_r$, Theorem \ref{t:class_reif:approx_dist}.2 and Theorem \ref{t:class_reif:approx_dist}.3 we have in $B_{2r}(x)$ that:
\begin{align}\label{e:class_reif:approx_subman_proof:1}
&|F(y,z) - z|<C(n)\delta\,r\, ,\notag\\
&|\partial_y F(y,z)|\, ,\;\;|\partial_z F(y,z) - Id|<C(n)\delta\, .
\end{align}

First note that the zeros of $F$ describe $S_r$:
\begin{exercise}
Show $S_r\cap B_r(x) = \{|\nabla \Phi_r|\equiv 0\} = \{(y,z)\in B_r(x): F(y,z) = 0\}$.
\end{exercise}
\noindent {\bf Hint: } Note that \eqref{e:class_reif:approx_subman_proof:1} says that $F(y,z)$ has a unique zero on each $z$-slice, use Theorem \ref{t:class_reif:approx_dist}.1 to see that the zero of $F$ must be a zero of $\Phi_r$.
\begin{remark}
This exercise is the main place we use condition $(1)$ of Theorem \ref{t:class_reif:approx_dist}.  \\
\end{remark}

Now by \eqref{e:class_reif:approx_subman_proof:1} and Theorem \ref{t:class_reif:approx_dist}.3 we may use the implicit function theorem \ref{t:background:implicit_function} in order to find a smooth function $f:B_{r}(x)\cap L_x\to \hat L_x^\perp$ such that 
\begin{align}
	&|\nabla f|\leq C(n)\delta\, ,\;\; r|\nabla^2 f|\leq C(n)\delta\, ,\notag\\
	&S_r\cap B_r(x) = \{(y,z)\in B_r(x): F(y,z)=0\} = \{(y,f(y))\}\cap B_r(x)\, .
\end{align}
Thus we have seen that $S_r$ is locally a smooth graphical submanifold and thus proved Theorem \ref{t:class_reif:approx_submanifold}.3.  We have seen in Section \ref{ss:class_reif:subman_approx} that $(1)-(6)$ of Theorem \ref{t:approx_submanifold} follow from $(1)$ and $(3)$, and therefore we have completed the proof of Theorem \ref{t:class_reif:approx_submanifold}.
	
\end{proof}

\vspace{.3cm}

\section{Proof of Distance Approximation Theorem}\label{s:class_reif:dist_approx_proof}

We now complete the proof of the Reifenberg Theorem by completing the proof of the Subspace Selection Lemma \ref{l:class_reif:subspace_selection} and the Distance Approximation Theorem \ref{t:class_reif:approx_dist}.  Let us begin some technical results, in particular we first build a useful covering of $B_2$:

\begin{lemma}\label{l:class_reif:partition}
There exists a covering $B_2\subseteq \bigcup_\alpha B_{\tilde r_\alpha}(x_\alpha)$, where $\tilde r_\alpha = \tilde r_{x_\alpha} = \frac{d(x_\alpha,S)\vee r}{100}$, and smooth nonnegative functions $\phi_\alpha$ such that
\begin{enumerate}
\item $\{ B_{\frac{1}{4}\tilde r_\alpha}(x_\alpha)\}$ are disjoint.
\item For each $y\in B_2$ we have $\#\{x_\alpha: y\in B_{4 \tilde r_\alpha}(x_\alpha)\}<C(n)$.
\item $\sum \phi_\alpha = 1$ on $B_2$ with $\text{supp}\,\phi_\alpha \subseteq B_{4\tilde  r_\alpha}(x_\alpha)$.
\item $|\partial^{(k)}\phi_\alpha|\leq C(n,k) \tilde  r_\alpha^{-k}$. \footnote{Recall $|\partial^{(k)} f|$ is the matrix norm of the $k^{th}$ derivative of $f$.}
\end{enumerate}
\end{lemma}
\begin{proof}
Let $\{x_\alpha\}\in S$ be any maximal subset so that $\{ B_{\tilde  r_\alpha/4}(x_\alpha)\}$ are disjoint.  By maximal we mean if $y\in B_2$ then $B_{\tilde  r_y/4}(y)\cap B_{\tilde  r_\alpha/4}(x_\alpha)\neq \emptyset$ for some $\alpha$.  Now let us show that $B_{2}\subseteq \bigcup_\alpha B_{\tilde  r_\alpha}(x_\alpha)$.  So for $y\in B_2$ let $\alpha$ be such that $B_{\tilde  r_y/4}(y)\cap B_{\tilde  r_\alpha/4}(x_\alpha)\neq \emptyset$.  Observing that $|\nabla \tilde  r_y|\leq \frac{1}{100}$ we have that $\tilde r_y\leq 2\tilde r_\alpha$.  In particular then gives us $y\in B_{\tilde  r_\alpha}(x_\alpha)$, and thus we have shown $B_{2}\subseteq \bigcup_\alpha B_{\tilde  r_\alpha}(x_\alpha)$.\\

The proof of $(2)$ follows from a volume estimate.  Indeed, for $y\in B_2$ consider the subset $\{x_\beta\}_1^N$ such that $y\in B_{4 \tilde  r_\beta}(x_\beta)$.  Observe as in the last paragraph that by using $|\nabla \tilde  r_y|\leq \frac{1}{100}$ we have for any such $\beta$ that $\frac{1}{2}r_y\leq r_\beta\leq 2r_y$.  In particular, $B_{\tilde  r_y/10}(x_\beta)\subseteq B_{8\tilde  r_y}(y)$ and are disjoint.  Thus we can estimate
\begin{align}
N \omega_n\,10^{-n} \tilde r_y^n = \Vol(\bigcup B_{\tilde  r_y/10}(x_\beta))	\leq \Vol(B_{8\tilde  r_y}(y)) = \omega_n  8^n \tilde r_y^n\, ,
\end{align}
which gives $N\leq 80^n$, as claimed.\\

To build the partition of unity first let $\phi':B_{4}(0^n)\to \dR$ be a fixed smooth, compactly supported nonnegative function with $\phi'\equiv 1$ on $B_{1}$.  Let us define $\phi'_\alpha(x) \equiv \phi'(\tilde  r_\alpha^{-1}(x+x_\alpha))$, and with this the partition of unity itself by
\begin{align}
\phi_\alpha(x) \equiv \frac{\phi'_\alpha(x)}{\sum_\alpha \phi'_\alpha(x)}\, .	
\end{align}

\begin{exercise}
Use $(2)$ to prove $c(n)\leq \sum \phi_\alpha(x)\leq C(n)$ for $x\in B_2$.  Use this to prove $(3)$ and $(4)$.
\end{exercise}

\end{proof}

Let us now first complete the proof of the Subspace Selection Lemma:

\begin{proof}[Proof of Subspace Selection Lemma \ref{l:class_reif:subspace_selection}]

Let $\{ B_{\tilde  r_\alpha}(x_\alpha)\}$ and $\phi_\alpha$ be the covering and partition of unity from Lemma \ref{l:class_reif:partition}.  For each $\alpha$ let 
\begin{align}
	L_{\alpha}\equiv L_{x_\alpha,10^4\tilde  r_\alpha}\, ,
\end{align}
be as in \eqref{e:classical_reif:min_subspace}.  Morally, we simply want to define $L_y\equiv \sum \phi_\alpha L_\alpha$ and check what estimates hold.  Of course, one needs a well defined way of averaging affine subspaces in order to do this.  Indeed, using the notion of nonlinear averages this is possible, but since we want these notes to be self-contained (and that is a rather technical proceedure) we will do this by hand.  However, it is helpful to keep in mind that the remainder of the proof is nothing other than some technical work in order to average nonlinear objects.

Now we wish to define subspaces  $L_y$ as in the lemma.  To define an affine subspace we need a linear subspace $\hat\pi_y$ and a point $\ell_y\in L_y$.  Let us begin by writing these out, and then we will move on to estimating them.  We define the point $\ell_y$ simply by
\begin{align}
\ell_y\equiv \sum_\alpha \phi_\alpha(y) \pi_\alpha[y]\, .	
\end{align}

The definition of $\hat\pi_y$ is a bit more involved.  Let us begin by defining the matrix valued function
\begin{align}
M_y \equiv \sum_\alpha \phi_\alpha(y)\hat \pi_\alpha\, .	
\end{align}

One sees from Exercise \ref{exer:subspace_close1} and Lemma \ref{l:class_reif:partition} that $|M_y-\hat\pi_\beta|<C(n)\delta$ for any $y\in B_{8\tilde r_\beta}(x_\beta)$, and in particular $M_y$ is close to a projection map.  If $e^1(y),\ldots,e^n(y)$ are the eigenvectors of $M_y$, in decreasing order, we then define our linear subspace
\begin{align}
	\hat\pi_y \equiv \text{span}\{e^1(y),\ldots, e^k(y)\}\, .
\end{align}
Using our estimate on $M_y$ we at least have $|\hat\pi_y-\pi_\beta|<C(n)\delta$ for any $y\in B_{8\tilde r_\beta}(x_\beta)$.  Let us state our first Claims on the regularity of $\ell_y$ and $M_y$:\\

{\bf Claim 1:  } We have the estimates $|\partial_i \ell_y - \hat\pi_y|, \tilde  r_y |\partial^2 \ell_y|\leq C(n)\delta$.

{\bf Claim 2:  }We have the estimates $|M_y-\hat\pi_\beta|\, ,\;\;\; \tilde  r_y^2 |\partial^2 M_y|\leq C(n)\delta$.\\

We prove the gradient estimate of Claim 1.  The other estimates are all the same.  We first compute
\begin{align}
\partial_i \ell_y = \sum_\alpha \partial_i\phi_\alpha(y) \pi_\alpha(y)+\sum \phi_\alpha \pi_\alpha[e_i]\, . 	
\end{align}
Now by using Exercise \ref{exer:subspace_close1} and that $\sum_\alpha \phi_\alpha = 1$ we obtain the following:
\begin{align}
&\sum_\alpha \partial_i \phi_\alpha = 0\, ,\notag\\
& |\hat\pi_\alpha - \hat\pi_\beta|,\,|\hat\pi^\perp_\alpha - \hat\pi^\perp_\beta| < C(n)\delta \tilde   r_\beta \text{ if }B_{8\tilde  r_\alpha}(x_\alpha)\cap B_{8\tilde  r_\beta}(x_\beta) \neq \emptyset\, .
\end{align}
Choosing $\beta$ so that $x_\beta\in B_{4\tilde  r_y}(y)$ we then have 
\begin{align}
|\partial_i \ell_y-\pi_\beta| = |\sum_\alpha \partial_i\phi_\alpha(y) (\pi_\alpha-\pi_\beta)|\leq C(n) \delta \tilde  r_y^{-1}\, , 
\end{align}
where we have used that our partition estimates on $\phi_\alpha$ and that $\tilde  r_y\approx \tilde  r_\alpha$ for any ball $B_{4\tilde  r_\alpha}(x_\alpha)$ which contains $y$. $\square$ \\

Estimating the subspaces $\hat\pi_y$ takes a bit more technical work, as it is not just a partition of unity argument:\\

{\bf Claim 3: }  $\tilde  r_y|\partial \hat\pi_{y}|,\, \tilde  r_y^2|\partial^2 \hat\pi_y|\leq C(n)\delta $.\\

We will focus on the gradient estimate, the hessian estimate is the same.  The key is that we need to convert the estimates on $M_y$ into estimates on $\hat\pi_y$.  The important point in this estimate is that there is a gap between the $k$ largest eigenvalues and the $n-k$ smallest eigenvalues, otherwise the estimate would not even be correct.  Let us begin by using the Rellich characterization to write
\begin{align}
\hat\pi_y \equiv \arg\sup_{\hat L^k} tr_L(M_y) = 	\arg\sup_{\hat L^k} \sum M_y(e^i,e^i)\, ,
\end{align}
where the sup is taken over all $k$-dimensional subspaces and $e_i$ are an arbitrary orthonormal basis of $\hat L^k$.  Now observe from Claim 2 that for $y\in B_{4\tilde  r_\beta}(x_\beta)$ we have 
\begin{align}
	|\hat\pi_y-\hat\pi_\beta|\, ,\;\; |M_y-\hat\pi_\beta|<C(n)\delta\, .
\end{align}

Now consider the spaces of linear maps
\begin{align}
    &V=\{v:\hat L_\beta\to \hat L_\beta^\perp\}\, ,\notag\\
	&V_{s}=\{v:\hat L_\beta\to \hat L_\beta^\perp \text{ s.t. }||v||<10^{-1}\}\, ,
\end{align}
where $V_{s}\subseteq V$ is the subset of maps with small norm.  For $v\in V_s$ we let $\hat L_v = \text{Graph}(f) = \{(\ell,v(\ell)):\ell\in \hat L_\beta\}$ be the associated linear subspace, and thus we may view $V_s$ as the open set of linear subspaces which are close to $\hat L_\beta$.  Then we define the smooth mapping $F:B_{4\tilde  r_\beta}(x_\beta)\times V_s\to V$ by 
\begin{align}
	\langle F(y,v),w\rangle  \equiv \partial_w tr_{\hat L_v}(M_y)\, .
\end{align}
Note that $F(y,\hat\pi_y) = 0$.  We wish to use the implicit function theorem \ref{t:background:implicit_function} to give estimates on $\hat\pi_y$.  Thus using Claim 2 and the eigenvalue gap we have the estimates
\begin{align}
r|\partial_{y_i} F|\, ,\;\; r^2 |\partial_{y_i}\partial_{y_j} F|\, ,\;\; |\langle\partial_{v} F, w\rangle - \langle v,w\rangle|<C(n)\delta\, .	
\end{align}

In particular, by the implicit function theorem \ref{t:background:implicit_function} there exists $\hat\pi_y:B_{2\tilde  r_\beta}(x_\beta)\to V_s$ such that $\{(y,v):F(y,v)=0\}=\{(y,\hat\pi_y)\}$ which satisfies the estimates of the claim.  $\square$\\

Having constructed $L_y$ we need only see that it satisfies the desired estimates from the Lemma.

\begin{exercise}\label{exer:class_reif:subspace_selection}
Show the following:
\begin{enumerate}
\item Using that $\pi_y[v] = \hat\pi_y[v-\ell_y]+\ell_y$ show the estimates $|\partial m_y-\hat\pi_y|, r_y|\partial^2 m_y|\leq C(n)\delta$.
\item Show $d_H(L_y\cap B_{10\tilde  r_y}(y), L_\alpha\cap B_{10\tilde  r_y}(y))<C(n)\delta\,\overline r_y$.  Use this to prove estimate $d_H(L_y\cap B_{10\tilde  r_y}(y), S\cap B_{10\tilde  r_y}(y))<C(n)\delta\,\overline r_y$.
\end{enumerate}
\end{exercise}

\end{proof}

With the Subspace Selection Lemma complete we may now prove Theorem \ref{t:class_reif:approx_dist}:

\begin{proof}[Proof of the Approximate Distance Function Theorem \ref{t:class_reif:approx_dist}]

Recall we define $\Phi_r$ explicitly by the formula 
\begin{align}
\Phi_r(y)\equiv \frac{1}{2}d(y,L_y)^2 = \frac{1}{2}|y-\pi_y[y]|^2 = \frac{1}{2}|y - m_y|^2\, ,
\end{align}
as in \eqref{e:class_reif:approx_distance}.  We begin by proving $(1)$.  Thus let $x\in S$ and $\ell\in L_x\cap B_r(x)$.  Let $\phi:\hat L_x^\perp\to \hat L_x^\perp$ be a smooth cutoff function with $\phi\equiv 1$ in $B_{r}(0)$ and $\phi\equiv 0$ outside of $B_{2r}(0)$.  Note that for each $y\in B_r(x)$ there exists a unique point in the intersection $L_y\cap \hat L_x^\perp$ since they are transverse.  This point moves smoothing since $L_y$ moves smoothly.  Consider the smooth mapping $\pi:\hat L_x^\perp\to \hat L_x^\perp$ given by
\begin{align}
\pi(y) \equiv  y-\phi(y)\cdot L_y\cap \hat L_x^\perp\, .
\end{align}
Note that $\pi=y$ outside $B_{2r}$ and from our estimates on $L_y$ we have 
\begin{align}
&|\pi(y)-y|<C(n)\delta\, r\, ,	\;\;\;\;\;|d\pi - Id|<C(n)\delta\, .
\end{align}
We see then that $\pi$ is a degree $1$ mapping which fixes the boundary of $B_{2r}(0)$, and hence there exists $y\approx 0$ for which $\pi(y)=0$.  At this point we then have
\begin{align}
y\in L_y \implies \Phi_r(y) = \frac{1}{2}|y-\pi_y(y)|^2 = 0\, ,	
\end{align}
as claimed.\\

Estimates $(2)-(4)$ are now relatively straight forward computations.  To prove Theorem \ref{t:class_reif:approx_dist}.2 we first compute the derivative of $\Phi_r$:
\begin{align}\label{e:class_reif:approx_dist_proof:1}
&\partial_i \Phi_r = \langle \hat\pi^\perp_y[e_i],y-m_y\rangle +\langle (\hat\pi_y-\partial m_y)[e_i],y-m_y \rangle \, .
\end{align}
Squaring this gives
\begin{align}
	\Big||\partial \Phi_r|^2 - |y-m_y|^2\Big| \leq C(n)\delta\, |y-m_y|^2 \leq C(n)\delta\, \Phi_r\, ,
\end{align}
as claimed.  To compute Theorem \ref{t:class_reif:approx_dist}.3 we similarly first compute the hessian
\begin{align}
&\partial_i\partial_j \Phi_r = \langle \hat\pi^\perp_y[e_i],\hat\pi^\perp_y[e_j]\rangle+\langle (\hat\pi_y-\partial m_y)[e_j],\hat\pi^\perp_y[e_i]\rangle + \langle \partial_j\hat\pi^\perp_y[e_i],y-m_y\rangle\notag\\
&\;\;\;\;\;\;\;\;\;\; + \langle (\partial_j\hat\pi_y-\partial_j\partial m_y)[e_i],y-m_y \rangle+\langle (\hat\pi_y-\partial m_y)[e_i],\hat\pi^\perp_y[e_j] \rangle\notag\\
&\;\;\;\;\;\;\;\;\;\;+\langle (\hat\pi_y-\partial m_y)[e_i],(\hat\pi_y-\partial m_y)[e_j] \rangle\, ,
\end{align}
which gives
\begin{align}
\big|\partial_i\partial_j \Phi_r - \hat\pi_y^\perp\big|\leq C(n)\delta\, ,	
\end{align}
as claimed.  Theorem \ref{t:class_reif:approx_dist}.4 is the same.

\end{proof}

\newpage

\part*{Lecture 2: Rectifiable Reifenberg for Measures}

Let us now explore the various issues that arise in attempting to use the classical Reifenberg theorem in applications (for instance singular sets of nonlinear equations).  To summarize we need to deal with the following three issues:

\begin{enumerate}
\item[(I)]  The hausdorff distance used in Reifenberg condition behaves as a pointwise $L^\infty$ bound, and in practice we will have more integral control than pointwise control. 
\item[(II)]  In applications our sets or measures can have holes and need not satisfy the Reifenberg condition!  Best we can do is force symmetry on some special regions.
\item[(III)]  BiH\"older control is simply too weak.  Lack of gradient control prevents understanding of volume or rectifiable structure.  
\end{enumerate}

To deal with $(I)$ it becomes more natural to discuss controlling measures than sets and one works with the Jones $\beta$-numbers of these sets instead of the Hausdorff distance.  Though this adds some technical complication, it is a relatively minor issue by itself.  

Dealing with $(II)$ is a more serious, and one introduces $k$-neck regions to help deal with this, see Section \ref{s:rect_reif_outline:neck_regions}.  One is able to gain back a weak version of the Reifenberg control in this case, but only on certain regions and in a discrete sense.  Dealing with these neck regions then becomes similar to the classical Reifenberg case, though dealing with the holes presents some subtle points in the construction.  One also has to then prove such neck regions exist and are even fairly common, which is the content of the the Neck Decomposition Theorem in Section \ref{s:rect_reif_outline:neck_regions}.

Dealing with $(III)$ is again a serious issue, and will require both the neck region ideas of $(II)$ and a more refined collection of hypotheses.  One issue at hand is the snowflake example of the previous section, which shows that the assumptions of the Reifenberg theorem cannot give better than biH\"older control.  One therefore needs more than just scalewise control on the Jones $\beta$-numbers, and we will require a Dini condition be satisfied.  To understand this a little better let us begin by revisiting the snowflake example:

\begin{example}[Snowflake 2]\label{ex:snowflake2}
	We are refining the snowflake construction of Example \ref{ex:snowflake}, so that much of our terminology originates there.  As before let $a^0_{-1}=(-2,0)$, $a^{0}_1 = (2,0)$ with $S_0 = \ell_{a^0_{-1},a^0_{1}}$ the associated interval, but now also choose a sequence $0<\delta_i\leq \frac{1}{8}\delta$.  Similar to the original construction, we define $S_i$ inductively in the following way.  If $S_i = \bigcup \ell_{a^i_j,a^i_{j+1}}$ is piecewise linear, then let $S_{i+1} = S^{\delta_i}_i \equiv \bigcup \ell^{\delta_i}_{a^i_j,a^i_{j+1}} \equiv \bigcup \ell_{a^{i+1}_j,a^{i+1}_{j+1}}$.  Note that $d_H(S_i,S_j)\leq 4\sum_{k=i}^j 2^{-k}O(\delta_k)\leq 2^{-i}O(\delta)$, and in particular there exists $S=\lim S_i$ and $S$ is an $O(\delta)$-Reifenberg set as in Example \ref{ex:snowflake}.
	
	As before, let us use the pythogorean theorem to compute the length of $S_i$ to be
	\begin{align}
		|S_i|^2 = 16\prod_{j\leq i}\Big(1+\delta^2_j\Big)\, .
	\end{align}
	As observed previously, for $|S_i|$ to remain uniformly bounded it is not sufficient for $\delta_j$ to remain uniformly small.  One sees from the above that $|S_i|$ remains uniformly bounded iff $\sum \delta_j^2<\infty$.  In particular, to control the volume and lipschitz structure of $S$ one requires not only that the Reifenberg constant of $S\cap B_r(x)$ tend to zero as $r$ tends to zero, but that the Reifenberg constants be square summable in the scales. $\square$
\end{example}

The above example is generalized into a Theorem in \cite{davidtoro}, where they show that a Reifenberg set $S$ for which the sum $\sum \beta^\infty_k(x,2^{-i})\approx \int_0^2 \beta^\infty_k(x,s) \frac{ds}{s}$ is uniformly bounded at each point is bilipschitz to $B_1(0^k)$. \\

We want in this lecture to work toward stating the main generalization of the Reifenberg Theorem which will interest us and solve the issues $(I),(II),(III)$.  The context is more involved now, and thus we will need to begin by describing some structure. 

\vspace{.3cm}

\section{Hausdorff, Minkowski, and Packing Content}\label{s:rect_reif:haus_content}

In this section we give a brief review of the notions of Hausdorff, Minkowski, and packing content.  For a more detailed reference, we refer the reader to \cite{mattila,Fed}. Let us begin with the notions of content:

\begin{definition}[Content]\label{d:content}
Given a set $S\subseteq \dR^n$ and $r>0$ we define the following:
\begin{enumerate}
\item The $k$-dimensional Hausdorff $r$-content of $S$ is given by \footnote{ The constant $\omega_k$ is the volume of a unit ball in $\dR^k$.}
\begin{align}
\cH^k_r(S)\equiv \inf\Big\{\sum \omega_k\, r_i^k : S\subseteq \bigcup B_{r_i}(x_i) \text{ and }r_i\leq r\Big\}\, .
\end{align}
\item The $k$-dimensional Minkowski $r$-content of $S$ is given by
\begin{align}
\cM^k_r(S)\equiv \inf\Big\{\sum \omega_k\, r^k : S\subseteq \bigcup B_{r}(x_i)\Big\} \approx  r^{k-n}\,\text{Vol} (B_r(S))\, .
\end{align}
\item The $k$-dimensional packing $r$-content of $S$ is given by
\begin{align}
\cP^k_r(S)\equiv \sup\Big\{\sum \omega_k r_i^k : \, x_i\in S\text{ with }\{B_{r_i}(x_i)\} \text{ disjoint,}\text{ and }r_i\leq r\Big\}\, .
\end{align}
\end{enumerate}
\end{definition}
\vspace{.2cm}

\begin{exercise}
Let $S^\ell=L^\ell\cap B_1$ where $L^\ell$ is an $\ell$-dimensional subspace.  Show there exists $0<c(n)<C(n)<\infty$ such that for all $0<r<1$:
\begin{align}
	&\cH^k_r(S^\ell)\, ,\;\cM^k_r(S^\ell)\, \approx r^{k-\ell}\, ,\notag\\
	&\cP^k_r(S^\ell) = \infty \text{ if }\ell>k\, ,\notag\\
	&\cP^k_r(S^\ell) \approx r^{k-\ell} \text{ if }k\geq \ell\, .
\end{align}
\end{exercise}

\begin{example}
Let $S=\dQ^n\cap B_1(0^n)$ be the rationals.  Then for all $0<r<1$ we have
\begin{align}
	&\cH^k_r(S) = 0 \text{ if $k>0$ with } \cH^k_r(S)\stackrel{r\to 0}{\longrightarrow} \infty \text{ if $k=0$ }\, .
\end{align}
In particular, $\cH^k(S)=0$ for $k>0$ and so $\dim_H(S) = 0$.  However, the Minkowski and Packing content are quite badly behaved:
\begin{align}
	&\cM^k_r(S)\approx r^{k-n}\stackrel{r\to 0}{\longrightarrow} \infty \text{ for $k<n$}\, ,\notag\\
	&\cP^k_r(S) = \infty \text{ for all }k\, .
\end{align}
Morally, this is because the closure $\overline S = \overline B_1$ is an $n$-dimensional set, and so from a packing and minkowski point of view $S$ itself is treated as an $n$-dimensional set.
\end{example}
\vspace{.2cm}

Note then that controlling the Hausdorff content amounts to finding {\it some} covering of $S$ which is well behaved, controlling the Minkowski content amounts to saying the covering $S$ by balls of radius $r$ is well behaved, and controlling the packing content amounts to saying {\it every} covering is well behaved.  In particular, bounding the Hausdorff content is less powerful than bounding the Minkowski content, which is itself less powerful than bounding the packing content.  Thus we have the relations
\begin{align}
\cH^k_r(S) \lesssim \cM^k_r(S) \lesssim \cP^k_r(S)\, ,
\end{align}
where $\lesssim$ means the inequality holds up to a dimensional constant.  One can use these notions in the classical manner to define measures and dimensions.  In particular, for completeness sake let us recall the definition of Hausdorff measure:

\begin{definition}[Hausdorff Measure]
	Given $S\subseteq \dR^n$ we define its Hausdorff measure $\cH^k(S) = \lim_{r\to 0} \cH^k_r(S)$.
\end{definition}

\vspace{.25cm}

\subsection{Rectifiability of Sets}

Let us now discuss the notion of rectifiable sets.  In essence, these are sets which are manifolds away from a set of measure zero, though this set of measure zero need not be closed.  We begin with a definition:

\begin{definition}
Let $S\subseteq \dR^n$ be a set, then we say $S$ is $k$-rectifiable if there exists a countable collection of lipschitz maps $f_i:S_i\subseteq \dR^k\to \dR^n$ such that $\cH^k(S\setminus \bigcup f_i(S_i))=0$.
\end{definition}

Sometimes the above is referred to as countably rectifiable, and one additionally assumes $\cH^k(S)<\infty$ in order to call $S$ rectifiable.  

\begin{example}
Let $S^k\subseteq \dR^n$	 be a $k$-dimensional submanifold, then $S^k$ is $k$-rectifiable.  Let $\tilde S^k\subseteq S^k$ be an arbitrary subset, then $S^k$ is also $k$-rectifiable.  Let $S\equiv \bigcup_{q\in\dQ^n} (\tilde S^k+q)$, then $S$ is also $k$-rectifiable.
\end{example}

Note that the example $S$ above is dense in $\dR^n$, so the notation of rectifability depends heavily on the ability to decompose the set. \\

The notion of a rectifiable measure is very similar:

\begin{definition}
Let $\mu$ be a measure on $B_1(0^n)$.  We say $\mu$ is $k$-rectifiable	if there exists a $k$-rectifiable set $S$ such that $\mu(B_1\setminus S)=0$ and $\mu\cap S$ is absolutely continuous with respect to the Hausdorff measure $\cH^k\cap S$.
\end{definition}

\vspace{.25cm}

\section{Jones $\beta$-numbers}\label{s:rect_reif:Jones_beta}

The Rectifiable Reifenberg Theorem \ref{t:rect_reif} we will be introducing will be for a measure $\mu$ instead of a set $S$.  As such, let us discuss the Jones $\beta$-numbers to estimate how close the support of $\mu$ is to a $k$-plane in a more $L^2$ sense, as in $(I)$.  Because of the possibility of holes as in $(II)$ we will only be concerned with how closely the support of $\mu$ is to living inside a $k$-plane, without care for how dense the support is inside $L^k$.  Precisely:

\begin{definition}[Jones $\beta$-numbers]\label{d:jones_beta}
	Given a measure $\mu$ and integer $k\in \dN$ we define the $L^2$ $\beta$-numbers 
	\begin{align}
		\beta_k(x,r;\mu)^2 = \beta_k(x,r)^2 \equiv \inf_{L^k}\, r^{-2-k}\int_{B_r(x)} d(y,L)^2\,d\mu[y]\, ,
	\end{align}
 where the infimum is taken over all $k$-planes $L^k$. 
\end{definition}

When no confusion arises we will simply write $\beta(x,r)$ and drop the dependence on the measure.  Let us state the following example, which shows in particular how control on $\beta(x,r)$ does not stop the existence of 'holes' in $\text{supp}\,\mu$:

\begin{example}\label{ex:mu_k}
Let $L^k\subseteq \dR^n$ be a fixed subspace and let $\mu_k$ be an arbitrary measure with $\text{supp}\,\mu_k\subseteq L^k\cap B_2$.  Then $\beta_k(x,r)=0$ for all $B_r(x)\subseteq B_2$. $\square$
\end{example}

Let us now give a series of examples which illustrate how $\beta_k$ behaves when the support of $\mu$ takes its support on sets of various dimensions:

\begin{example}\label{ex:mu_k_2}
As a more specific example, consider $\mu_k\equiv \alpha_0\delta_0+\alpha_k\cH^k\cap L^k\cap B_2$, where $\cH^k\cap L^k$ is the $k$-dimensional Hausdorff measure restricted to the subspace $L$, $\delta_0$ is the dirac delta measure at the origin, and $\alpha_0,\alpha_k$ are arbitrary.  Then $\beta_k(x,r)=0$ for all $B_r(x)\subseteq B_2$. $\square$
\end{example}

We see from the above example, by taking $\alpha_0,\alpha_k$ very large, that the measure $\mu_k$ does not need to have any apriori bounds, even if $\beta_k(x,r)=0$ is identically zero.  The example also illustrates how even if the support $\text{supp}\mu_k$ is $k$-rectifiable, the measure itself may not be.\\

The next example studies what happens for measures supported on higher dimensional subsets:

\begin{example}\label{ex:mu_+}
	Consider $\mu_{+} = \delta \cH^n\cap B_2$, where $\cH^n$ is the $n$-dimensional Hausdorff measure.  Then we can compute
\begin{align}
	\beta_k(x,r)^2 \approx \omega_n\,\delta^2\, r^{n-k}\, .
\end{align}
In particular, we have that $\beta_k(x,r)$ is always $\delta$-small, and indeed is decaying polynomially. $\square$
\end{example}

The above example shows that $\beta_k(x,r)$ may be uniformly small, even decaying, but that the support of $\mu_k$ need not live on a $k$-dimensional object. \\ 

The following exercises are fairly straightforward but very instructive in building an intuition for the behavior of the $\beta$-numbers:

\begin{exercise}\label{exer:beta_continuity}
Show that if $B_{s}(y)\subseteq B_r(x)\subseteq B_{as}(y)$ then $\beta_k(y,s)\leq C(n,a)\beta_k(x,r)$.	
\end{exercise}

\begin{exercise}\label{exer:beta_continuity2}
Show if $r_i =2^{-i}$ that $\int_r^1 \beta(x,s)^2\frac{ds}{s}\leq \sum_{r\leq r_i\leq 1}\beta(x,r_i)^2\leq \int_r^2 \beta(x,s)^2\frac{ds}{s}$. 
\end{exercise}

\begin{exercise}
Show if $\mu=\mu_1+\mu_2$ then  $\beta(x,r;\mu_i)^2\leq \beta(x,r;\mu)^2$.
\end{exercise}
\vspace{.3cm}

\section{Rectifiable Reifenberg Theorem for Measures}\label{s:rect_reif:rect_reif}

We are now in a position to deal with the general case and state the Rectifiable Reifenberg Theorem, which is designed to handle the issues $(I)$, $(II)$ and $(III)$.  

Let us now combine the examples of the last section in order to illustrate all the subtle issues involved in trying to use the $\beta_k$-numbers to study a completely general measure:

\begin{example}[Varying Dimensions Example]\label{ex:varying_dimensions}
	Consider the measure
\begin{align}
\mu \equiv \mu_k + \mu_+\, ,	
\end{align}
where $\mu_k$ is from Example \ref{ex:mu_k} and $\mu_+$ is from Example \ref{ex:mu_+}.  Then by considering the subspace $L^k$ from the examples we may estimate
\begin{align}
	\beta_k(x,r)^2 \leq 4\omega_n\,\delta^2\, r^{n-k}\, .
\end{align}
In particular, {\it for arbitrary $\alpha_0$, $\alpha_k$} we have that $\beta(x,r)$ is always $\delta$-small, and indeed is decaying polynomially. $\square$
\end{example}

The above example shows that even when one has extremely strong conditions on $\beta_k$, that for a general measure we cannot make a general statement about the rectifiability of $\mu$ or its total mass.  This may seem like endgame, however the example also suggests that maybe we can decompose the measure into pieces where we do have such control.  Indeed, this is exactly the case and the content of the Rectifiable Reifenberg Theorem:\\

\begin{theorem}[Rectifiable Reifenberg \cite{ENV}]\label{t:rect_reif}
Let $\mu$ be a nonnegative Borel-regular measure supported in $B_1(0^n)$.  Suppose 
\begin{equation}\label{e:rect_reif:beta_bound}
\int_{B_1} \int_0^2 \beta(x,r)^2\frac{dr}{r}\,d\mu \leq \Gamma\, .
\end{equation}
Then we can write $\mu = \mu_{k} + \mu_{+}$ into a sum of measures such that 
\begin{enumerate}
\item $\mu_{+}(B_1)\leq C(n)\Gamma$.
\item If $\cK\equiv \text{supp}\,\mu_{k}$ then $\cK$ is $k$-rectifiable with $\haus^k(\cK)<C(n)$, and indeed we have the Minkowski and packing content estimates:  
\begin{align}
	\Vol(B_r(\cK)) \leq c(n) r^{n-k}\, ,\;\;\;\;\;\cP^k_r(\cK)\leq C(n)\, .
\end{align}
\end{enumerate}
\end{theorem}
\begin{remark}
This result holds for measures in Hilbert spaces, see \cite{ENV_Hilbert}, and in particular the $c(n)$ constants above are turned into $c(k)$ constants.  The result is also generalizable to Banach Spaces, see \cite{ENV_Hilbert}, however this is more subtle. 
\end{remark}

The above result is for a general measure.  We may obtain some stronger results if we strict ourselves to measures with either upper or lower density bounds.  First let us precisely define this:

\begin{definition}
	Let $\mu$ be a nonnegative measure.  Then we define the upper and lower densities:
	\begin{align}
		&\theta^*(\mu,x)\equiv \lim\sup_{r\to 0} \frac{\mu(B_r(x))}{\omega_k r^k}\, ,\notag\\ 
		&\theta_*(\mu,x)\equiv \lim\inf_{r\to 0} \frac{\mu(B_r(x))}{\omega_k r^k} \, .
	\end{align}
\end{definition}
\vspace{.2cm}

Let us discuss some corollaries of Theorem \ref{t:rect_reif}:

\begin{corollary}\label{c:rect_reif}
Let $\mu$ be a nonnegative Borel-regular measure supported in $B_1(0^n)$ and let \eqref{e:rect_reif:beta_bound} hold.  Then we have the following:
\begin{enumerate}
\item If $\theta_*(\mu,x)\leq A$ then $\mu(B_1)\leq C(n)\Big(\Gamma+A\Big)$.
\item If $a\leq \theta^*(\mu,x)$ then $\cK\equiv \text{supp}\,\mu$ is $k$-rectifiable with $\cH^k(\cK)\leq C(n,a)\, \Gamma$.
\item If $a\leq \theta^*(\mu,x)$ and $\theta_*(\mu,x)\leq A$ then $\mu$ is $k$-rectifiable with $\mu(B_1)\leq C(n)\Big(\Gamma+A\Big)$.
\end{enumerate}
\end{corollary}
\begin{remark}
Note that we are requiring the relatively weak conditions of an upper bound on the lower density in $(1)$, and conversely a lower bound on the upper density in $(2)$.  This is directly due to the packing estimates on $\cK$.  If one only had weaker Hausdorff measure estimates on $\cK$, one would have to make the stronger assumptions of upper bounds on upper density and lower bounds on lower density. 	
\end{remark}
\begin{remark}
Tolsa \cite{tolsa:jones-rect} and 
\cite{azzam-tolsa} have proved the following related result.  If $\mu$ is a measure whose upper and lower densities are bounded almost everywhere, then $\mu$ is $k$-rectifiable iff $\int_0^2 \beta_k(x,s)\frac{ds}{s}<\infty$ for a.e. $x$. 	An effective version of the if direction is given in $(3)$, however \cite{tolsa:jones-rect} also proves the only if direction.
\end{remark}


\newpage

\part*{Lecture 3:  Outline Proof of Rectifiable Reigenberg - Neck Regions and their Structure Theory}

We now want to begin the proof of the Rectifiable Reifenberg Theorem \ref{t:rect_reif}.  This lecture will focus on introducing Neck regions and their associated Structure and Decomposition theorems.  After we discuss these we will use them to prove Theorem \ref{t:rect_reif}.  In subsequent lectures we will prove the Neck Structure and Neck Decompositions themselves.  In the process we will build quite a bit more information than is present in Theorem \ref{t:rect_reif}, which itself is quite useful in applications.\\

\section{Neck Regions and their Structure and Decomposition}\label{s:rect_reif_outline:neck_regions}

This section is dedicated to introducing the reader to the notion of a neck region, and we will be stating the basic structure theory and decomposition theorem associated to Neck Regions.  Neck regions first made their appearance in \cite{JiNa_L2} during the proof of the $n-4$ finiteness conjecture and again in \cite{NaVa_EnId} in the proof of the energy identity conjecture for Yang-Mills.  They are also used in \cite{Na_HarmBook},\cite{ChJiNa} in order to to prove the rectifiability of singular sets of harmonic maps and spaces with lower Ricci curvature, respectively.  The notion of a neck region developed in these notes is very analogous, though in some manners quite a bit easier to work with than in the last references due to the technical conditions involved.\\   

Neck Regions will be regions which we can control in a manner analogous to our control in the classical Reifenberg theorem.  There are many subtle points, including the fact that we cannot get a true Reifenberg condition to hold.  That is, when we restrict ourselves to {\it well behaved} points they may not be dense in some $k$-plane on each scale.  However, we can replace it with a weaker notion by making sure at each scale there are {\it enough} well controlled points to at least weakly span a $k$-dimensional plane.  This will be enough to get the control we desire in the end.  To make this more precise let us introduce the notions of linear independence and noncollapsing.  Recall a set of points $\{x_i\}_0^k$ is called linearly independent if no point lives in the span of any of the others.  More effectively:

\begin{definition}[$\epsilon$-Linear Independence]\label{d:linear_independence}
We call a set $\{x_i\}_0^k\in B_r$ $\epsilon$-linearly independent if $x_{i+1}\not\in B_{\epsilon r}\big(x_0+\text{span}\{x_1-x_0,\ldots,x_i-x_0\}\big)$ for each $i$.  We say a set $S\subseteq B_r$ is a $(k,\epsilon)$-	linearly independent set if $\exists$ a $\epsilon$-linearly independent set of points $\{x_i\}_0^k\in S$.
\end{definition}

The notion of $\epsilon$-independence can be viewed as a very weak version of the Reifenberg condition.  That is, a set of points may not {\it densely} span an affine space, but they at least effectively span such a space.  In practice what we will need to be independent are noncollapsing points, which is to say points with lower mass bounds.   Precisely:

\begin{definition}[Noncollapsing]\label{d:noncollapsing}
	Let $\mu$ be a measure, then we say a ball $B_{r}(x)$ is $(k,\epsilon,\nu)$-noncollapsed if there exists a $2\epsilon$-linearly independent $\{x_i\}_0^k\in B_r(x)$ such that we have the lower mass bounds $\mu(B_{\epsilon r}(x))>\nu (\epsilon r)^k$.
\end{definition}
\begin{remark}
The condition that $B_r(x)$ be $(k,\epsilon,\nu)$-noncollapsed guarantees not only that there are balls with definite mass	 in $B_r(x)$, but that there are $k+1$ such balls which effectively span a $k$-dimensional affine subspace.
\end{remark}
\begin{remark}
The condition implies that if $y_i\in B_{\epsilon r}(x_i)$ then $\{y_i\}_0^k$ are $\epsilon$-linearly independent. 	\\
\end{remark}

Let us now give our formal definition of Neck Regions:\\

\begin{definition}[Neck Regions]\label{d:neck_region}
Let $\mu$ be a measure on $B_r$ with $\cC\subseteq B_r$ a closed subset and $r_x:\cC\to \dR^+$ a radius function such that the closed balls $\{\overline B_{\tau^2 r_x}(x)\}$ are disjoint \footnote{From this point forward we take $\tau=\tau_n \equiv 10^{-10n}$.  The main property necessary for this constant is that on scale $\tau_n$ we need that the $k$-content of an $\ell$-plane is small when $\ell<k$, more precisely we need $\Vol(B_{\tau_n}(L^{k-1})\cap B_2) < 10^{-2}\tau^k_n$.}.  We call $\cN = B_r\setminus \overline B_{r_x}(\cC)$ a $(k,\delta,\epsilon,\nu )$-neck region if \footnote{If $A\subseteq \dR^n$ is a closed subset and $r_A:A\to \dR^+$ is a nonnegative function then we define the closed variable radius tube $\overline B_{r_A}(A)\equiv \bigcup_{a\in A} \overline B_{r_A(a)}(a)$.}
\begin{enumerate}
\item[(n1)] For $x\in \cC$ and $r_x\leq s\leq r$ $\exists$ affine $L^k$ such that $L\cap B_s\subseteq B_{\tau s}(\cC)$ and $\cC\cap B_s\subseteq B_{\delta s}(L)$.
\item[(n2)] For each $x\in\cC$ with $\tau^{-1}r_x\leq r\leq 1$ we have that $B_r(x)$ is $(k,\epsilon,\nu)$-noncollapsed. 
\item[(n3)] $\int_{r_x}^{2r} \beta_k(x,s)^2\frac{ds}{s}<\delta$ for each $x\in \cC$.
\end{enumerate}
\end{definition}
\begin{remark}
We call $\cC_0\equiv \{x\in \cC: r_x=0\}$ and $\cC_+\equiv \{x\in \cC: r_x>0\}$.	
\end{remark}
\begin{remark}
One should imagine $\delta<<\epsilon<<\tau=\tau(n)\equiv 10^{-10n}$.  
\end{remark}

\begin{remark}
Note that condition $(n3)$ implies, along with the scale continuity of Exercise \ref{exer:beta_continuity}, for every $r_x<s<100 r$ that $\beta_k(x,s)<C(n)\delta$.	
\end{remark}

It is probably helpful to begin with a handful of random remarks discussing the conditions of a Neck Region and where they come into play.  Some of these remarks may not be completely sensible until the reader begins the process of going through the details of the constructions, however it seems nonetheless helpful to have these remarks to put everything in the right framework:
\begin{enumerate}
\item[(nr1)] $\tau$ represents the scale parameter.  In the iterative construction of Neck Regions we will drop by a factor of $\tau$ in each stage.  The condition that $\{\overline B_{\tau^2 r_x}(x)\}$ be disjoint guarantees we do not overcover any region by more than a controlled amount.
\item[(nr2)] One imagines $\cC$ as roughly living on the best approximating subspaces at each stage, hence the $(n1)$ condition $\cC\cap B_s\subseteq B_{\delta s}(L)$.  Because $\{\overline B_{\tau^2 r_x}(x)\}$ are disjoint one cannot ask for better than the weak converse $L\cap B_s\subseteq B_{\tau s}(\cC)$.
\item[(nr3)] The purpose of $\epsilon$ is to control the linear independence in the noncollapsing in $(n2)$.  One wants $\epsilon<<\tau$ so that when a ball {\it fails} to be $(k,\epsilon,\nu)$-noncollapsed, then when recovering the number of balls which have large mass is appropriately small.  This will be important in controlling the number balls which are not Neck regions in the Decomposition Theorem.
\item[(nr4)] The $(n3)$ condition $\int_{r_x}^{2r} \beta_k(x,s)^2\frac{ds}{s}<\delta$ tells us that the error from approximating $\mu$ by an affine subspaces is summably small at each point.  As we will see in Section \ref{ss:proof_neck_structure:best_subspaces} using $\beta_k$ to control $\mu$ is only meaningful on noncollapsed balls, otherwise $\beta_k$ may be small for trivial reasons.  We want $\delta<<\epsilon$ so that our ability to approximate $\mu$ by an affine subspace is much better than how spaced out the noncollapsed balls are.
\item[(nr5)] A nice exercise is to see that the disjoint property of $\{\overline B_{\tau^2 r_x}(x)\}$ trivially forces the lipschitz estimate $|\Lip\, r_x|\leq \tau^{-2}$.  It is possible in the construction to force the much nicer condition $|\Lip\, r_x|\leq \delta$, however this requires a nontrivial argument which complicates the proof of the Neck Decomposition Theorem rather significantly, see \cite{JiNa_L2}.  
\item[(nr6)] As in the classical Reifenberg we will prove a form of submanifold approximation theorem for Neck regions.  Under the condition $|\Lip\, r_x|\leq \delta$ a gluing construction for the submanifold approximation theorem is possible, however under the condition $|\Lip\, r_x|\leq \tau^{-2}$ it really is not.  This is our primary reason for introducing the approximate distance function construction in Section \ref{ss:class_reif:dist_approx}, which is a more flexible argument.\\
\end{enumerate}


We will discuss examples in a moment, but let us first introduce some basic notation which will be in force throughout.\\

{\bf Notation:  }  Given a Neck region $\cN = B_1\setminus \overline B_{r_x}(\cC)$ we extend the radius function $r_x:\cC\to \dR$ to $r_x:B_1\to \dR$ by the regularity scale formula:
\begin{align}\label{e:regularity_scale}
r_x \equiv \sup\{0< s<1: r_y\geq \tau^{-2}s \;\forall\; y\in \cC\cap B_s(x)\}\, .
\end{align}

Let us observe that $r_x\geq d(x,\cC)$.  Below are a handful of basic properties about our regularity scale:
\begin{exercise}\label{exer:regularity_scale}
Show the following:
\begin{enumerate}
\item $r_x$ is lipschitz with $|\text{Lip}\, r_x|\leq \tau^{-2}$.  \footnote{Recall $|\text{Lip} f|(x) \equiv \lim\sup_{y\to x} \frac{|f(x)-f(y)|}{|x-y|}$.  The lipschitz constant is controlled by how large of radii are disjoint.}
\item If $x\in \cC$ then $r_x$ agrees with the radius function on $\cC$.
\end{enumerate}
\end{exercise}

The above exercise tells us that $r_x$ is a fair extension of the radius function on $\cC$.  Let us begin with an easy example:

\begin{example}[Trivial Neck Region]\label{ex:trivial_neck}
Let $\mu_k = \nu  \cH^k\cap L\cap B_1$ be a multiple of the Hausdorff meausure on a linear subspace $L^k\subseteq \dR^n$.  Let $r_x:L\cap B_1\to \dR^+$ any nonnegative function with $|\nabla r_x|\leq \tau^{-2}$ and let $\cC\subseteq L^k\cap B_1$ be any closed subset such that $\{B_{\tau^2 r_x}(x)\}_{\cC}$ are disjoint and maximal.  By maximal we mean that if $y\in L\cap B_1$, then $B_{\tau^2 r_y}(y)\cap B_{\tau^2 r_x}(x)\neq \emptyset$ for some $x\in \cC$.  Then observe that $\cN\equiv B_1\setminus \overline B_{r_x}(\cC)$ is a $(k,\delta,\epsilon,\nu )$-neck region for $\mu$. 
\end{example}

We can extend this example slightly:

\begin{example}[Neck Region Example]\label{ex:trivial_neck2}
Let $\mu = \mu_k+\mu_+$, where $\mu_k = \nu  \cH^k\cap L\cap B_1 + \mu^{-}_k$ is a multiple of the Hausdorff meausure on a subspace $L^k\subseteq \dR^n$, plus an arbitrary measure $\mu'_k$ supported on $L^k$, and $\mu_+ = \delta\,\cH^n\cap B_1$ is a small multiple of the standard Hausdorff measure.  Let $r_x:L\cap B_1\to \dR^+$ any nonnegative function with $|\nabla r_x|\leq \tau^{-2}$ and let $\cC\subseteq L^k\cap B_1$ be any closed subset such that $\{B_{\tau^2 r_x}(x)\}_{\cC}$ are disjoint and maximal.  Then observe that $\cN\equiv B_1\setminus \overline B_{r_x}(\cC)$ is a $(k,\delta,\epsilon,\nu )$-neck region for $\mu$. 
\end{example}

Let us make some observations about the above examples:
\begin{itemize}
\item The center points $\cC$ all lie inside a well behaved submanifold, indeed just a linear subspace.
\item  The measure $\mu$ is uniformly bounded, and indeed small, on the neck region.
\end{itemize}

These two properties are not random, and indeed we wish to show they actually always hold on Neck Regions.  The following is our main structure theorem for Neck Regions: 

\begin{theorem}[Neck Structure Theorem]\label{t:neck_structure}
	Let $\mu$ be a measure on $B_1$ with $\cN = B_1\setminus \overline B_{r_x}(\cC)$ a $(k,\delta,\epsilon,\nu )$-neck region.  Then there exists a submanifold $T\subseteq B_1$ such that
\begin{enumerate}
\item $\cC\subseteq T$.
\item $T$ is $1+C(n,\epsilon,\nu )\delta$-bilipschitz to $B_1(0^k)$.
\item $\mu(\cN)\leq \mu(B_{r_x}(T))<C(n,\epsilon)\,\delta$. 
\end{enumerate}
In particular, for $\delta<\delta(n,\epsilon,\nu)$ we have $\cC_0\subseteq T$ is $k$-rectifiable with $\cH^k(\cC_0)+\sum_{\cC_+} r_x^k \leq C(n)$.
\end{theorem}
\begin{remark}
Note that if $\mu_\cC = H^k\cap \cC_0+ \omega_k\sum_{\cC_+}r_x^k$ is the packing measure of $\cC$, then condition $(2)$ gives us that for each $x\in \cC$ and $r_x\leq r\leq 2$ the Ahlfors regularity condition $A(n)^{-1}r^k \leq \mu_\cC(B_r(x))\leq A(n) r^k$.	See Exercises \ref{exer:covering:1}.
\end{remark}
\begin{remark}
Note if $B_{s_i}(y_i)$ is any disjoint collection of balls with $y_i\in T$, then $(2)$ gives that $\sum s_i^k \leq C(n)$.  See Exercise \ref{exer:covering:1}.
\end{remark}
\begin{remark}
The Neck Regions and Neck Structure Theorem mimic very closely the constructions and results from singularity analysis of nonlinear equations, see \cite{naber-valtorta:harmonic},\cite{ChJiNa}.  However, in those contexts one replaces the assumed $\beta$-bounds in $(n3)$ with bounds on the appropriate monotone quantity.  This makes the proof of the Neck Structure Theorem much more subtle in those contexts.  One proves so-called $L^2$-subspace approximation theorems in order to (sharply) turn the monotone quantity drop into a $\beta$-number estimate, however to do this one must already know the volume bound on $\mu$.  This leads to a loop where one must prove the volume estimates, rectifiability, and $\beta$-number estimates all simultaneously.  In the context of these notes we get to avoid this subtlety.	
\end{remark}

Now that we have introduced Neck regions and discussed some their structure theory, the next reasonable question is whether or not any exist.  More than that, for this to be worth the time and effort we should somehow hope that enough neck regions exist that we may potentially use them toward our greater goal of proving the Rectifiable Reifenberg Theorem \ref{t:rect_reif}.  The next result tells us how under the correct $\beta$-number bounds we may decompose $\text{supp}\mu\subseteq B_1$ into pieces with special behavior.  The crucial piece of this decomposition will turn out to be neck regions:

\begin{theorem}[Neck Decomposition]\label{t:neck_decomposition}
	Let $\mu$ be a Borel measure on $B_1$ and assume for each $x\in B_1$ that $\int_0^{2} \beta_k(x,s)^2\frac{ds}{s}\leq \Gamma$.   
Then for each $\nu,\epsilon ,\delta>0$ with $\delta<\delta(n,\epsilon,\nu )$ $\exists$ a covering $B_1 \subseteq \cS^{-}\cup \cS^k\cup \cS^{+}\,  \text{ with }$
\begin{align}
&\cS^{+} = \bigcup_a \big(\cN_a\cap B_{r_a}\big) \cup \bigcup_b B_{r_b}(x_b)\, \text{ and }\;\; \cS^k = \bigcup_a \cC_{0,a}\, ,
\end{align}
and such that
\begin{enumerate}
\item 	$\cN_a = B_{2r_a}(x_a)\setminus \overline B_{r_{a,x}}(\cC_a)$ is a $(k,\delta,\epsilon,\nu )$-neck region.  In particular, $\mu(\cN_a)\leq C(n)\delta\, r_a^k$ and $\cC_{0,a}$ is $k$-rectifiable by Theorem \ref{t:neck_structure}.
\item $B_{r_b}(x_b)$ satisfies the measure constraint $\mu(B_{2r_b})<C(n)\nu \, r_b^k$ .
\item We have the content estimates $\sum r_a^k + \sum r_b^k < C(n,\delta,\epsilon,\nu ,\Gamma)$ and packing estimate $\cP^k(\cS^-\cup\cS^k)<C(n,\delta,\epsilon,\nu ,\Gamma)$ \footnote{Due to time constraints we will not carefully prove the packing estimate in these notes and instead focus on the Hausdorff measure estimates, however they follow from the precise arguments of these notes, no new ideas are necessary, just a little messy technical work.},
\item We have the Hausdorff measure estimate $\cH^{k}(\cS^{-})=0$.
\end{enumerate}
\end{theorem}
\begin{remark}\label{r:neck_decomposition:rectifiable}
It follows immediately that $\cS^k$ is $k$-rectifiable.	
\end{remark}
\begin{remark}[Effective Estimates]\label{r:eff_estimates}
It follows from $(1)$, $(2)$ and $(3)$ that $\cH^{k}(\cS^{k})\leq C(n,\epsilon,\Gamma)$ and $\mu(S^{+}) < C(n,\delta,\epsilon,\Gamma)$.  Repeating $(4)$ gives $\cH^{k}(\cS^{-})=0$.
\end{remark}
\begin{remark}
If the measure $\mu$ has a lower density bound then one can weaken the assumption to more of a Carlson estimate: $r^{-k}\int_{B_r(x)}\int_0^{2r} \beta_k(x,s)^2\frac{ds}{s}\,d\mu\leq \Gamma$ for all $B_r(x)\subseteq B_1$.  Indeed, most interesting examples satisfy this but not the pointwise condition.  The lower density bound is important as one can go from the pointwise condition to the carleson condition through an iterative process, but to control estimates one needs to turn small mass into small hausdorff measure at each stage.
\end{remark}

Let us present an example of this: 

\begin{example}

Let $L$ and $L'$ be perpendicular $k$-dimensional subspaces with $\mu = \nu \cH^k\cap L\cap B_1 + \nu \cH^k\cap L'\cap B_1$.  For each $x\in L\cup L'\setminus \{0\}$ consider the radius $d_x\equiv 10^{-2}|x|$.  Note then that we may write the trivial neck region $\cN_x = B_{2 d_x}(x)\setminus L$ with $\cC_{x}=\cC_{0,x}= L\cap B_{2d_x}(x)$.  Now let $\{B_{r_a}(x_a)\}$ with $x_a\in L\cup L' \cap B_1$ and $r_a\equiv 10^{-2}|x_a|$ be a maximal collection of balls such that $\{B_{\tau^2 r_a}(x_a)\}$ are disjoint, so that in particular $(L\cup L'\cap B_1)\setminus\{0\}\subseteq \bigcup_a \cC_{0,a}$.  Similarly let us now choose $\{B_{r_b}(x_b)\}$ to be a maximal collection with $x_b\in B_1\setminus \bigcup_a B_{r_a}(x_a)$, $r_b\equiv 10^{-3}|x_b|$ and such that $\{B_{\tau^2 r_b}(x_b)\}$ are disjoint.  If we then define
\begin{align}
&\cS^+ \equiv \bigcup_a \big(\cN_a\cap B_{r_a}(x_a)\big) \cup \bigcup_b B_{r_b}(x_b)\, ,\notag\\
&\cS^k \equiv \bigcup_a \cC_{0,a}\, ,\notag\\
&\cS^- \equiv \{0\}\, ,
\end{align}
then we see that $B_1\subseteq \cS^-\cup \cS^k\cup \cS^+$ is a Neck Decomposition covering.  In particular one sees from this example that it is fully possible to have a countable number of pieces in the decomposition. $\square$
\end{example}

The rest of this section will be dedicated to seeing that the Neck Decomposition Theorem can be used to conclude the Rectifiable Reifenberg of Theorem \ref{t:rect_reif}. 

\vspace{.3cm}

\section{Proof of the Rectifiable Reifenberg Theorem \ref{t:rect_reif} given the Neck Decomposition Theorem \ref{t:neck_decomposition}}

First note that if we prove the result for $\Gamma=1$, then the general result has been proven by simply rescaling $\mu$.  Thus, we will focus on this case.  Now observe that using \eqref{e:rect_reif:beta_bound} we can define the set
\begin{align}
U_\Lambda \equiv \{x\in B_1: \int_0^2 \beta_k(\mu;x,r)^2\frac{dr}{r} < \Lambda\}\, ,
\end{align}
so that we have
\begin{align}\label{e:mu_Lambda}
\mu(B_1\setminus U_\delta)\leq \Lambda^{-1}\, . 	
\end{align}
Thus if we eventually choose $\Lambda=1$ say, then the above tells us we need only worry about estimating $\mu$ on $U_\Lambda$ \footnote{We consider general $\Lambda$ here as it will be used also in the proof of Corollary \ref{c:rect_reif}.}.  Now let $\mu_\Lambda\equiv \mu\cap U_\Lambda$.  Using that $\mu_\Lambda\leq \mu$, and hence $\beta_k(\mu_\Lambda;x,r)\leq \beta_k(\mu;x,r)$, we then have by the definition of $U_\Lambda$ and Exercise \ref{exer:subspace_close1} that for every $x\in B_1$ 
\begin{align}
\int_0^2 \beta_k(\mu_\Lambda;x,r)^2\frac{dr}{r} < C(n)\Lambda	\, .
\end{align}

Let us now apply the Neck Decomposition Theorem \ref{t:neck_decomposition} to $\mu_\Lambda$ in order to write $B_1 \subseteq \cS^{-}_\Lambda\cup \cS^k_\Lambda\cup \cS^{+}_\Lambda\,  \text{ with }$
\begin{align}
&\cS^{+}_\Lambda = \bigcup_a \big(\cN_a\cap B_{r_a}\big) \cup \bigcup_b B_{r_b}(x_b)\, ,\notag\\ 
&\cS^k_\Lambda = \bigcup_a \cC_{0,a}\, .
\end{align}

We now write
\begin{align}
\mu^k_\Lambda \equiv \mu_\Lambda \cap (\cS^K\cup \cS^-)\, ,\;\;\;\;\; \mu^+_\Lambda = \mu_\Lambda \cap \cS^+\, .
\end{align}

Using the content estimate Theorem \ref{t:neck_decomposition}.4 and the volume condition Theorem \ref{t:neck_decomposition}.2 we have that
\begin{align}
\mu^+_\Lambda(B_1) \leq \sum_a \mu_\Lambda(\cN_a) + \sum_b \mu_\Lambda(B_{r_b}(x_b)) \leq \nu \big(\sum_a r_a^k + \sum_b r_b^k\big)\leq C(n,\delta,\nu ,\Lambda)\, .
\end{align}

On the other hand, using Theorem \ref{t:neck_decomposition}.4, Theorem \ref{t:neck_decomposition}.5, and Remark \ref{r:neck_decomposition:rectifiable} then we see that if $\cK\equiv \text{supp}\mu^k_\Lambda = \cS^k\cup\cS^-$ then $\cK$ is $k$-rectifiable with the packing estimates $\cP^k(\cK)<C(n,\delta,\nu ,\Lambda)$.  In particular, $\cH^k(\cK)<C(n,\delta,\nu ,\Lambda)$.  Now fix $\nu >0$ with $\delta<\delta(n,\nu ,\Lambda)$ with $\Lambda=1$, and define
\begin{align}
	\mu_+ = \mu\cap U_\Lambda + \mu^+_\Lambda\, ,\;\;\;\;\; \mu_k = \mu^k_\Lambda\, ,
\end{align}
then we see we have completed the proof of Theorem \ref{t:rect_reif}.\\

\subsection{Proof of Corollary \ref{c:rect_reif}}

To prove Corollary \ref{c:rect_reif}.1 let us consider the decomposition $\mu = \mu_+ + \mu_k$ from Theorem \ref{t:rect_reif} and let $\cK\equiv \text{supp}\mu_k$.  For each $x\in \cK$ use the density assumption to fix $s_x>0$ such that
\begin{align}
\mu(B_{s_x}(x))\leq 2 A\, \omega_k\, s_x^k	\, .
\end{align}
Note then that $\cK\subseteq \bigcup_{x\in \cK} B_{s_x}(x)$ is a so called Besicovitch covering, and thus by the Besicovitch covering theorem we have
\begin{align}
\cK\subseteq \bigcup_{i=1}^N\bigcup_{a_i} B_{s^i_{a_i}}(x^i_{a_i})\, ,
\end{align}
where $N\leq N(n)$ and each collection $\{B_{s^i_{a_i}}(x^i_{a_i})\}\subseteq \{B_{s_x}(x)\}$ are of disjoint balls.  Note then that the packing estimate $\cP^k(\cK)\leq C(n)$ tells us for each $i$ that $\sum_{a_i}s_{a_i}^k\leq C(n)$, and thus we can estimate
\begin{align}
\mu(\cK) \leq \sum_{i=1}^N \sum_{a_i}\mu(B_{s^i_{a_i}}(x^i_{a_i}))\leq 2A\,\omega_k\, \sum_{i=1}^N \sum_{a_i}	 s_{a_i}^k\leq C(n) A\, ,
\end{align}
as claimed.\\

Let us now prove Corollary \ref{c:rect_reif}.2.  Let us first prove the Hausdorff measure estimate on $\text{supp}\mu$.  For this consider again the decomposition $\mu = \mu_+ + \mu_k$.  The estimates for $\mu_k$ are already sufficient, so we focus on $\mu_+$.  For each $x\in \text{supp}\mu_+\setminus \text{supp}\mu_k$ consider a radius $s_x\leq r\cdot d(x,\text{supp}\mu_k)$ such that\footnote{Note that since $\cP(\text{supp}\mu_k)<C(n)$ we have that $\text{supp}\mu_k$ is compact, thus this is a reasonable constraint on the radius.}
\begin{align}
\mu_+(B_{s_x}(x)) = \mu(B_{s_x}(x))\geq \frac{1}{2} a\, \omega_k\, s_x^k	\, .
\end{align}
As before we then have a Besicovitch covering $\text{supp}\mu_+\subseteq \bigcup_{x\in \text{supp}\mu_+} B_{s_x}(x)$ so that we can find a covering
\begin{align}
\text{supp}\mu_+\setminus \text{supp}\mu_k\subseteq \bigcup_{i=1}^N\bigcup_{a_i} B_{s^i_{a_i}}(x^i_{a_i})\, ,
\end{align}
where $N\leq N(n)$ and each collection $\{B_{s^i_{a_i}}(x^i_{a_i})\}\subseteq \{B_{s_x}(x)\}$ are of disjoint balls.  We can then estimate
\begin{align}
\omega_k\, \sum_{i=1}^N \sum_{a_i}	 s_{a_i}^k \leq 2a^{-1}\sum_{i=1}^N \sum_{a_i}\mu_+(B_{s^i_{a_i}}(x^i_{a_i}))\leq 2a^{-1}N(n)\mu_+(B_1)\leq 2a^{-1}C(n,\Gamma)\, .	
\end{align}
In particular, this gives the content estimate $\cH^k_r(\text{supp}\mu_+\setminus \text{supp}\mu_k)\leq 2a^{-1}C(n,\Gamma)$.  Since $r<1$ was arbitrary this gives the Hausdorff measure estimate, as claimed.  \\

To finish Corollary \ref{c:rect_reif}.2 we need to show that $\text{supp}\mu$ is $k$-rectifiable.  Recall from \eqref{e:mu_Lambda} the definition of $U_\Lambda$ and $\mu_\Lambda$, and as in the proof of Theorem \ref{t:rect_reif} from the last subsection we have
\begin{align}
	&\mu(B_1\setminus U_\Lambda)<\Lambda^{-1}\, ,\notag\\
	&\mu_\Lambda = \mu^+_\Lambda+\mu^k_\Lambda\, .
\end{align}
It is thus enough to show $\mu_\Lambda$ is $k$-rectifiable for each $\Lambda<\infty$.  Now since for each $x\in U_\Lambda$ we have $\int_0^{2r}\beta_k(x,s)^2\frac{ds}{s}<\Lambda$, note then that for each $x\in U_\Lambda$ we must also have
\begin{align}
\lim_{r\to 0} \int_0^{2r}\beta_k(x,s)^2\frac{ds}{s}	= 0\, .
\end{align}
Fix $\eta>0$ very small, then by the above we can find $r>0$ such that if $$U_{\eta,r}\equiv\{x\in U_\Lambda: \int_0^{2r}\beta_k(x,s)^2\frac{ds}{s}\leq \eta\}\, , $$ then 
\begin{align}
	\mu_\Lambda(B_1\setminus U_{\eta,r})<\eta\, .
\end{align}
Let us now use the Hausdorff measure estimate on $U_{\eta,r}\subseteq U_\Lambda\subseteq \text{supp}\mu$ previously proved in order to find a covering
\begin{align}
U_{\eta,r}\subseteq \bigcup_i B_{r_i}(x_i) \text{ with }r_i\leq r \text{ and }\sum r^k_i\leq C(n)\, .	
\end{align}
Let $\mu_{i}\equiv \mu_\Lambda\cap U_{\eta,r}\cap B_{r_i}(x_i)$.  Then we can apply Theorem \ref{t:rect_reif} in order to write $\mu_{i}=\mu^+_{i}+\mu^-_{i}$ such that
\begin{align}
&\mu_i^+(B_{r_i}(x_i))\leq C(n)\eta\,r_i^k\, ,\notag\\
&\mu_i^k \text{ is $k$-rectifiable}\, .
\end{align}
In particular, we have that $\mu_\Lambda$ is $k$-rectifiable away from a set of measure
\begin{align}
\mu_\Lambda(B_1\setminus U_{\eta,r}) + \sum_i \mu^+_i(B_{r_i}(x_i))\leq \eta +C(n)\eta\sum r^k_i \leq C(n)\eta	\, .
\end{align}
As $\eta>0$ was arbitrary, this proves that $\mu_\Lambda$ is $k$-rectifiable.  As $\Lambda$ was arbitrary, this proves that $\mu$ is $k$-rectifiable. $\qed$

\newpage

\part*{Lecture 4:  Proof of Neck Structure Theorem and Neck Decomposition Theorem}

We focus in this lecture on the meat of the argument, and prove the neck structure theorem and the neck decomposition theorem.

\section{Proof of Neck Structure Theorem}\label{s:proof_neck_structure}

Our proof of the classical Reifenberg Theorem \ref{t:classical_reif}, given in Lecture 1, was designed precisely to pass over to the context of the Neck Structure Theorem.  Though the basic outline will remain the same, in the context of Theorem \ref{t:neck_structure} most of the results are a bit more refined and there are a handful of technical challenges and nuances beyond the classical result, which we will describe.  In particular, we begin with a few technical preliminaries in order to deal with this:

\subsection{Best Subspaces on Neck Regions}\label{ss:proof_neck_structure:best_subspaces}

Exercise \ref{exer:subspace_close1} told us in the context of the classical Reifenberg theorem that the best approximating subspaces do not change much from scale to scale.  For a general measure, even with well controlled $\beta$-numbers, this will not need to be the case.  We will study this phenomena some in this section and aim toward proving that at least on Neck regions, one can indeed control the best subspaces in a manner analogous to Exercise \ref{exer:subspace_close1}.  Let us begin with some notation and definitions.  

Recall that for each $x\in B_1$ and $0<r\leq 10$ we have fixed a choice of affine $k$-dimensional subspace $L_{x,r}=L_{x,r}[\mu]$ satisfying
\begin{align}\label{e:outline_neck_structure:min_subspace}
L_{x,r}\in \arg\min_L r^{-2-k}\int_{B_{r}(x)} d^2(y,L)d\mu[y]\, ,
\end{align}
so that 
\begin{align}\label{e:outline_neck_structure:min_subspace_error}
	\beta_k(x,r)^2 = r^{-2-k}\int_{B_{r}(x)} d^2(y,L_{x,r})d\mu[y]\, .
\end{align}
Given an affine subspace $L_{x,r}$ we denote $\pi_{x,r}:\dR^n\to L_{x,r}\subseteq \dR^n$ to be the projection map, $\hat L_{x,r}$ the associated linear subspace, and $\hat\pi_{x,r}:\dR^n\to \hat L_{x,r}\subseteq \dR^n$ the linear projection map.  The first point that is worth making is that unlike Exercise \ref{exer:subspace_close1}, if $\mu$ is arbitrary then there is no reason the subspaces need to be comparable, even if the $\beta$-numbers are small:

\begin{example}\label{ex:subspace_notclose}
Let $p,q\in B_1$ be points with $|p-q|=10^{-1}$ $\mu = \delta_{p}+\delta_q$ be the sum of dirac deltas.  Note that for every $x\in B_1$ and $r>0$ we have that $\beta_1(x,r)=0$.  If $p,q\in B_{r}(x)$ are both points in the ball then the best subspace $L_{x,r}$ is the line connecting $p$ and $q$.  If $B_{s}(y)$ contains only one of the points, say $p$, then $L_{y,s}$ can be {\it any} line which contains $p$.  In particular, $L_{x,r}$ and $L_{y,s}$ need not be at all comparable, even if $B_r(x)$ and $B_s(y)$ are comparable.  $\square$
\end{example}

Our main goal in this section is to see that in a Neck region the above does not happen, namely the subspaces $L_{x,r}$ are indeed comparable, in the spirit of Exercise \ref{exer:subspace_close1}.  The following is the main result of this section, which we will prove by its completion:

\begin{proposition}[Best Subspace Behavior on Neck Regions]\label{p:neck_best_subspaces}
Let $\cN = B_1\setminus \overline B_{r_x}(\cC)$ be a $(k,\delta,\epsilon,\nu )$-neck region with $\delta<\delta(n,\epsilon,\nu)$.  Consider subspaces $\{L_{x,r}\}$ defined as in \eqref{e:outline_neck_structure:min_subspace} and let $r\geq 10^2 r_x$.  Then for $x,y\in \cC$ the following hold:
\begin{enumerate}
	\item $d_H(L_{x,r}\cap B_{r}(x), L_{y,s}\cap B_{r}(x)) < C(n,\epsilon,\nu,a)\beta_k(x,10a\,r) \, r$ if $|x-y|\leq 10 r \text{ and } a^{-1} \leq \frac{r}{s} \leq a$.
	\item If $B_r(x)$ and $B_s(y)$ are as above then $\big||\hat\pi_{x,r}[v]|-1\big|\leq C(n,\epsilon,\nu,a)\beta_k(x,10a\, r)^2$ for each $v\in \hat L_{y,s}$ with $|v|=1$,.
\end{enumerate}
\end{proposition}

Estimates $(1)$ and $(2)$ above tell us that if $B_r(x)$ and $B_s(y)$ are comparable balls, then $L_{x,r}$ and $L_{y,s}$ are comparable affine subspaces.  As in Example \ref{ex:subspace_notclose} let us note that we need to be in a neck region for this, otherwise such results are false.  Let us also emphasize the square gain in $(2)$, as it is key to the volume and rectifiability estimates on $\mu$ later.\\

To prove Proposition \ref{p:neck_best_subspaces} we need to learn how to compare the best approximating subspaces of a measure in an accurate manner.  The following definition of a center of mass, which originates in \cite{ENV}, will be used to tell us how to control at least one point in a given ball:

\begin{definition}\label{d:generalized_COM}[Center of Mass]
We define the \emph{generalized $\mu$-center of mass} $X$ of a ball $B_r(x)$ as follows.  If $\mu(B_r(x)) < \infty$, let $X = \frac{1}{\mu(B_r(x))} \int_{B_r(x)} z d\mu(z)$ be the usual center of mass.  If $\mu(B_r(x)) = \infty$, we let $X$ be any point in the intersection
\begin{align}
X\in \overline{B_r(x)} \cap \bigcap \left\{\text{affine $V^k$} : \int_{B_r(x)} d(z, V)^2 d\mu(z) < \infty \right\}.
\end{align}
\end{definition} 

\begin{exercise}\label{exer:COM}
Show if $\mu(B_r(x))=\infty$ but $\beta_k(x,r)<\infty$ then such a point $X$ exists.	
\end{exercise}
\begin{exercise}\label{exer:COM2}
More generally, let $B_s(y)\subseteq B_r(x)$ with $\mu(B_s(y))=\infty$ but $\beta_k(x,r)<\infty$, then show the center of mass $Y$ of $B_s(y)$ lives in $L_{x,r}$.	
\end{exercise}

Let us now see how on a given ball we can at least control how far away the center of mass is from a best subspace.  The main result is the following:

\begin{lemma}[Center of Mass and Best Subspaces]\label{l:COM_control}
Suppose $B_s(y) \subset B_r(x)$ and let $L_{x,r}$ be a best subspace as in \eqref{e:outline_neck_structure:min_subspace}.  Let $Y$ be the generalized center of mass for $B_s(y)$.  Then
\begin{gather}
d(Y, L_{x,r})^2 \leq \frac{r^{k}}{\mu(B_s(y))} \beta_k(x, r)^2\,r^2\, .
\end{gather}
\end{lemma}
\begin{remark}
In this business it is key how various quantities depend on one another.  Let us emphasize that the above formula depends inverse linearly on $\mu(B_s(y))$ and quadratically on $\beta_k(x,r)^2$.  This relationship will be crucial to the mass bounds on Neck Regions later. 
\end{remark}

\begin{proof}
We can assume $\beta_k(x, r) < \infty$, otherwise there is nothing to show.  From Exercise \ref{exer:COM2} we have if $\mu(B_s(y)) =\infty$ then $Y\in L_{x,r}$. Otherwise we can calculate by Jensen's inequality
\begin{align}
d(Y, L_{x,r})^2
&\leq \frac{1}{\mu(B_s(y)) } \int_{B_s(y) } d(z, L_{x,r})^2 d\mu(z) \\
&\leq \frac{r^{k+2}}{\mu(B_s(y)) } r^{-k-2} \int_{B_r(x)} d(z, L_{x,r})^2 d\mu(z) =  \frac{r^{k+2}}{\mu(B_s(y)) } \beta_k(x,r)^2\, ,
\end{align}
as claimed.
\end{proof}

Now we need to move from our ability to control best subspaces at a single point to being able to control the whole best subspace.  The first step in this direction is to see that two $k$-dimensional affine subspaces are close iff they are close at $k+1$ independent points.  This is quite intuitive as an affine subspace is well defined by such a collection.  Before reading the next lemma recall from Definition \ref{d:linear_independence} the notion of $(k,\epsilon)$-linearly independence:

\begin{lemma}[Subspace Distance Estimates]\label{l:subspace_dist_independent_points}
Let $L_1,L_2$ be two $k$-dimensional affine subspaces and let $\{x_i\}_0^k\subseteq B_1$ be a $(k,\epsilon)$-linearly independent set.  Then we have the estimate
\begin{align}
d_H\big(L_1\cap B_1, L_2\cap B_1\big)\leq C(n,\epsilon) \sum_i\Big(d(x_i,L_1)+d(x_i,L_2)\Big)\, .	
\end{align}

\end{lemma}
\begin{proof}
Observe that we may assume $\sum_i\Big(d(x_i,L_1)+d(x_i,L_2)\Big)<10^{-3}\epsilon$, as otherwise by choosing $C(n,\epsilon)$ large the estimates trivially hold.  Let $L = x_0+\text{span}\{x_i-x_0\}$, and we will prove the result
\begin{align}
d_H\big(L_1\cap B_1, L\cap B_1\big)\leq C(n,\epsilon) \sum_id(x_i,L_1)\, ,	
\end{align}
the general case then follows from a triangle inequality.

Now with all of this given we define $\{x^1_i\}_0^k = \{\pi_{L_1}(x_i)\}$ and see that this set is $(k,10^{-1}\epsilon)$-linearly independent with $d(x^1_i,x_i)<10^{-2}\epsilon$.  Now let $\ell^1$ be any point in $L_1\cap B_1$.  Then we can uniquely write
\begin{align}
\ell_1-x^1_0 = \sum_1^k \ell^i_1(x^1_i-x^1_0)\, .	
\end{align}
An instructive exercise, which depends strongly on the $\epsilon$-linear independence of the set $\{x^1_i\}$, is to show the following:
\begin{exercise}
Show $|\ell^i_1|<C(n,\epsilon)|\ell_1-x^1_0|<C(n,\epsilon)$.	
\end{exercise}
Now if we define $\ell \equiv x_0 + \sum_1^k \ell^i_1(x_i-x_0)\in L$ then we have the estimate
\begin{align}
	d(\ell_1,\ell)< C(n,\epsilon)\sum_i \Big |x^1_i-x_i|= C(n,\epsilon) \sum_i d(x_i,L_1)\, .
\end{align}
Since $\ell_1$ was arbitrary this proves every point of $L_1\cap B_1$ lives within $C(n,\epsilon) \sum_i d(x_i,L_1)$ of a point in $L$.  The verbatim argument works with $L_1$ and $L$ switched, and thus we have proven the Hausdorff estimate and completed the Lemma.
\end{proof}

The above Lemma's will be the key point in proving Proposition \ref{p:neck_best_subspaces}.1.  In order to prove Proposition \ref{p:neck_best_subspaces}.2 we provide one more general lemma:

\begin{lemma}\label{l:linear_projection_squaregain}
	Let $L_1,L_2$ be two linear subspaces with $d\equiv d_H\big(L_1\cap B_1, L_2\cap B_1\big)$.  Then for each $v\in L_2$ with $|v|=1$ we have that
\begin{align}
	\big||\pi_1[v]|-1\big|<d^2\, .
\end{align}
\end{lemma}
\begin{proof}
	The proof is a simple application of the pythagorean theorem, however we emphasize here the important square gain on $d$ here, as this is crucial for future applications.
	
	To prove the result let $v' = \pi_1[v]$ with $w\equiv v'-v$.  Note the two estimates:
\begin{align}
	&|w| = d(v,L_1) \leq d\, ,\notag\\
	&|v|^2 = |v'|^2 + |w|^2\, .
\end{align}
With these in hand we obtain
\begin{align}
\big||v'|-1\big| \leq \big|1-\sqrt{1-|w|^2}\big|\leq |w|^2 \leq d^2\, ,  	
\end{align}
as claimed.
\end{proof}

Finally we are now in a position to prove Proposition \ref{p:neck_best_subspaces}:

\begin{proof}[Proof of Proposition \ref{p:neck_best_subspaces}]

We first prove Proposition \ref{p:neck_best_subspaces}.1.  Let us fix $B_r(x)$ with $r\geq 10^2 r_x$, and by using $(n2)$ let $\{z_i\}_0^k\in B_{r/4}(x)$ be a $(k,2\epsilon)$-linearly independent set with $r_i=\epsilon r/4$ such that $\mu(B_{r_i}(z_i))>\nu r_i^k$.  Now let $Z_i\in B_{r_i}(z_i)$ be the generalized $\mu$-center of mass as in \eqref{d:generalized_COM}.  Note then that the $\{Z_i\}$ are $(k,\epsilon)$-linearly independent.  Thus by using Lemma \ref{l:COM_control} and that $\mu(B_{r_i}(z_i))>\nu r_i^k$ by $(n2)$ we get that
\begin{align}
&d(Z_i, L_{x,r}) < C(n,\nu,\epsilon) \beta_k(x,r) r\, ,\notag\\
&d(Z_i, L_{x,10 a r}) < C(n,\epsilon,a,\nu) \beta_k(x,10a r)\, r\, ,
\end{align}
Therefore by Lemma \ref{l:subspace_dist_independent_points} we have that $d_H(L_{x,r}\cap B_{10ar}(x), L_{x,10ar}\cap B_{10ar}(x)) < C(n,\epsilon,\nu,a)\beta_k(x,10a r) \, r$.  Note that $B_s(x)\subseteq B_{10ar}(x)$, so that the same argument gives $d_H(L_{y,s}\cap B_{10a r}(x), L_{x,10ar}\cap B_{10a r}(x)) < C(n,\epsilon)\beta_k(x,10a r) \, r$.  The triangle inequality then proves Proposition \ref{p:neck_best_subspaces}.1.\\

Now to prove Proposition \ref{p:neck_best_subspaces}.2 we simply use Lemma \ref{l:linear_projection_squaregain} together with Proposition \ref{p:neck_best_subspaces}.1.\\

\end{proof}

\vspace{.3cm}


\subsection{Submanifold Approximation Theorem}

Recall that a primary goal in the proof of the Neck Structure Theorem \ref{t:neck_structure} is to build a submanifold $T$ which contains $\cC$ and is bilipschitz to the ball $B_1(0^k)$.  In the spirit of the classical Reifenberg, we will build a family $T_{r}$ of smooth submanifolds which live near $\cC$ and are scale invariantly smooth.  To state the Approximating Submanifold Theorem more precisely, let us begin by introducing some notation.  Recall from \eqref{e:regularity_scale} our extension of the radius function $r_x$ to all of $B_1$ by a regularity scale procedure.  In the construction of the submanifolds $T_{r}$ we will not want to look below scale $r$, and as such we also consider the following:  

\begin{align}\label{e:regularity_scale_r}
\overline r_x \equiv r_x\vee r \equiv \max\{r_y,r\}\, .
\end{align}
Note that $\overline r_y\geq d(y,\cC)\vee r$.

We can now state the Approximating Submanifold Theorem, which is one of the central technical results needed in the proof of Theorem \ref{t:neck_structure}:

\begin{theorem}[Approximating Submanifold Theorem]\label{t:approx_submanifold}
Let $\cN = B_1\setminus \overline B_{r_x}(\cC)$ be a $(k,\delta,\epsilon,\nu )$-neck region with  with subspaces $\{L_{x,r}\}$ defined as in \eqref{e:outline_neck_structure:min_subspace}.  For each $r>0$ there exists a smooth submanifold $T_{r}\subseteq B_2$ with $T_{r}=T_1=L_{0,2}$ on $A_{3/2,2}(0)$ and which satisfy
\begin{enumerate}
\item $\cC\subseteq B_{C(n)\delta\overline r_y}(T_r)$,  $T_r\subseteq B_{C(n)\tau\overline r_y}(\cC)$.
\item $d_H(T_r\cap B_s(x),L_{x,s}\cap B_s(x))<C(n)\Big(\int_{\bar r_y}^{\tau^{-1}s}\beta_k(x,t)\frac{dt}{r}\Big) s\leq C(n)\delta s$ for $s\geq \overline r_y$.
\item $T_r$ is a $(C(n)\beta_k(x,\tau^{-1}r_x), r_x)$-graphical submanifold, see Definition \ref{d:graphical_regularity}.
\item $\exists$ smooth $\pi_r: B_{\overline r_y}(T_r)\to T_r\subseteq \dR^n$ with $\pi_r\cap T_r = Id$ and $|\nabla^2 \pi_r|(y)<C(n)\beta(y,\tau^{-1}\overline r_y)$. 
\item $d_H(T_{r/2},T_r)<C(n)\delta r$ with $|\pi_r(x)-x|<C(n)\delta r$ for $x\in T_{r/2}$.
\item For $x\in T_{r/2}$ and a unit vector $v\in L_{x,\bar r_x}$ we have $||d\pi_r[v]|-1|<C(k)\beta_k(x, \tau^{-1}\overline r_x)^2$. 
\end{enumerate}
\end{theorem}
\begin{remark}
Note that if we take $T_1$ to be the affine subspace $L_{0,2}$ with $\pi_1$ the orthogonal projection map then the above holds with $r=1$.  It will be convenient  to make this choice.	
\end{remark}
\begin{remark}
Note the square gain in $(6)$ on the $\beta$-number estimate.
\end{remark}

\vspace{.3cm}

Let us begin with some exercises.  If the reader completed the exercises of Section \ref{ss:background:submanifolds} and Section \ref{ss:class_reif:subman_approx} then these are almost the same:

\begin{exercise}\label{exer:approx_subman:tangent_space}
Use Exercise \ref{exer:graphical_subspace:2}, Theorem \ref{t:approx_submanifold}.2 and Theorem \ref{t:approx_submanifold}.3 to see that for each $x\in T_r$ if $L\equiv T_xT_r$ is the tangent space at $x$, then $d_{GH}(L\cap B_1(x),L_{x,\overline r_x}\cap B_1(x))\leq C(n)\beta_k(x,\tau^{-1}r)$.
\end{exercise}

\begin{exercise}\label{exer:approx_subman:tangent_space2}
Use the last exercise and Theorem \ref{t:approx_submanifold}.6 and Theorem \ref{t:approx_submanifold}.4 to see that for each $x\in T_r$ if $v\in T_xT_r$ is a unit tangent vector\footnote{Recall $T_xT_r$ is an affine subspace, we say $v\in T_xT_r$ is a unit vector if $|v-x|=1$.  That is, as an element of the associated linear subspace $v$ is a unit norm vector.} at $x$, then $||d\pi_r[v]|-1|<C(n)\beta_k(x, \tau^{-1}\overline r_x)^2$.
\end{exercise}

\begin{exercise}\label{exer:approx_subman:local_bilipschitz}
Let $x,y\in T_{r}$ with $|x-y|\leq 10 r_x$.  Let $\sigma:[0,1]\to T_{r}$ be the curve connecting $x$ and $y$ defined by $\sigma(t) =\pi_{r}\big((1-t)x+ty\big)$.  Use Theorem \ref{t:approx_submanifold}.4 to show the length $|\sigma|$ of $\sigma$ satisfies $(1-C(n)\delta)|x-y|\leq |\sigma|\leq (1+C(n)\delta)|x-y|$.  

Hint:  Write the length $|\sigma|=\int_0^1|\dot\sigma|$ and argue as is outlined in Exercise \ref{exer:class_reif:approx_subman:1} to estimate $|\dot \sigma|$.
\end{exercise}

The above exercise is telling us that locally the intrinsic and extrinsic geometry of $T_r$ are the same.  The proof of Theorem \ref{t:neck_structure} in the next subsection will implicitly prove that this holds globally.

\vspace{.3cm}
\subsection{Proof of Neck Structure Theorem \ref{t:neck_structure} given the Approximating Manifold Theorem \ref{t:approx_submanifold}}

The beginning of the proof is similar to the proof of the classical Reifenberg Theorem \ref{t:classical_reif} given Theorem \ref{t:class_reif:approx_submanifold}.  However, the key estimates are different as we now need to conclude bilipschitz control over the submanifolds $T_r$, not just bih\"older, and we need to conclude a mass bound on the neck region $\mu(\cN)$.  \\

Begin by considering the radii $r_i = 2^{-i}$ and the submanifolds $T_i = T_{r_i}$ from Theorem \ref{t:approx_submanifold}.  Let $\pi_i = \pi_{r_i}\cap T_{i+1}:T_{i+1}\to T_i$ be the projection map from Theorem \ref{t:approx_submanifold} restricted to $T_{i+1}$.  Let us define the maps
\begin{align}
	&\Pi_{i,j} \equiv \pi_{i-1}\circ\cdots\circ \pi_j:T_i\to T_j\notag\\
	&\Pi_i\equiv \Pi_{i,0}:T_i\to T_0\equiv L\, .
\end{align}
Note that the $\Pi_i$ are necessarily diffeomorphisms which equal the identity in $A_{3/2,2}(0)$.  Note first the weak estimate, which follows from Exercise \ref{exer:class_reif:approx_subman:3} and Theorem \ref{t:approx_submanifold}.5:
\begin{align}\label{e:neck_str:1}
	|\Pi_{i,j}(x)-x|\leq \sum |\pi_\ell\big(\Pi_{i,\ell+1}(x)\big)-\Pi_{i,\ell+1}(x)|<C(n)\delta \sum_i^j r_\ell\leq C(n)\delta r_j\, .
\end{align}

With this in hand we turn to our first main Claim of the result:\\

{\bf Claim 1: If $x\in T_i$ and $v\in T_xT_i$ is a unit tangent vector then $||d\Pi_{i}[v]|-1|<C(n)\delta$}.  \\

Note by the chain rule that $d\Pi_{i,j}[v] = d\pi_{j}[d\Pi_{i,j-1}[v]]$.  Thus using \eqref{e:neck_str:1}, Exercise \ref{exer:beta_continuity} and Exercise \ref{exer:approx_subman:tangent_space2} we have that 
\begin{align}
||d\Pi_{i,j}[v]|-|d\Pi_{i,j-1}[v]||\leq C(n)\beta_k(x,\tau^{-1}r_j)^2|d\Pi_{i,j-1}[v]|\, ,	
\end{align}
and hence
\begin{align}
\big(1-C(n)\beta_k(x,\tau^{-2}r_j)^2\big)|d\Pi_{i,j-1}[v]|\leq |d\Pi_{i,j}[v]|\leq \big(1+C(n)\beta_k(x,\tau^{-1}r_j)^2\big)|d\Pi_{i,j-1}[v]|\, .
\end{align}
Composing the upper bounds gives us
\begin{align}
	|d\Pi_{i,j}[v]|<\prod_i^j\big(1+ C(n)\beta_k(x,\tau^{-1} r_\ell)^2\big)\, .
\end{align}

\begin{exercise}
	Use a geometric series to show if $c_i<<10^{-1}$ then $\prod_i (1+c_i)\leq 1+2\sum_i c_i$.
\end{exercise}

By this exercise we have if$C(n)\beta_k(x,\tau^{-1} r_\ell)^2<C(n)\delta<<1$ for each $r_\ell$ then we can conclude the estimate
\begin{align}
	|d\Pi_{i,j}[v]|&<1+C(n)\sum_i^j\beta_k(x,\tau^{-1} r_\ell)^2\notag\\
	&\leq 1+C(n)\int_{r_i}^{2\tau^{-1}r_j}\beta_k(x,s)^2\frac{ds}{s}\, ,\notag\\
	&\leq 1+C(n)\delta\, ,
\end{align}
where the middle estimate follows from Exercise \ref{exer:beta_continuity2} and the last estimate follows from the definition of a Neck Region.  The lower bound is the same, which finishes the proof of the Claim. $\qed$\\

Let us now see that $\Pi_i:T_i\to L$ is a bilipschitz map\footnote{One should specify whether $T_i$ is given the intrinsic or extrinsic geometry.  The next Claim is for the extrinsic geometry, however a very similar proof goes through to prove the bilipschitz estimate for the intrinsic geometry as well.}.  Precisely:\\

{\bf Claim 2:  Let $x,y\in T_i$, then $(1-C(n)\delta)|x-y|\leq |\Pi_i(x)-\Pi_i(y)|\leq (1+C(n)\delta)|x-y|$} .\\

Let us first consider the case when $|x-y|\leq 10r_i$ and let $\gamma:[0,1]\to T_i$ be the smooth curve $\gamma(t) =\pi_{\bar r}\big((1-t)x+ty\big)$ which connects $x$ and $y$.  Note that by Exercise \ref{exer:approx_subman:local_bilipschitz} we have
\begin{align}
(1-C(n)\delta)|x-y|\leq |\gamma|\leq (1+C(n)\delta)|x-y|\, .	
\end{align}
Now consider the curve $\Pi_i\circ\gamma$ connecting $\Pi_i(x)$ and $\Pi_i(y)$.  Then using Claim 1 we have the estimate
\begin{align}
|\Pi_i(x)-\Pi_i(y)|&\leq |\Pi_i\circ\gamma| = \int_0^1 |d\Pi_i[\dot\gamma]|\leq\int_0^1 |\dot\gamma| +\int_0^1\Big||d\Pi_i[\dot\gamma]|-|\dot\gamma|\Big|\notag\\
&\leq (1+C(n)\delta)|\gamma|\leq (1+C(n)\delta)|x-y|\, ,	
\end{align}
which completes half the estimate.  The other half is the same, beginning with the curve $\gamma:[0,1]\to L$ given by $\gamma(t) = (1-t)\Pi_t(x)+t\Pi_i(y)$.  Then we can similarly estimate
\begin{align}
|x-y|\leq |\Pi_i^{-1}\circ \gamma|\leq (1+C(n)\delta)|\gamma|=(1+C(n)\delta)|\Pi_i(x)-\Pi_i(y)|\, .	
\end{align}
Thus we have proven the result when $|x-y|\leq 10r_i$.  Let us now consider the general case.

For $x,y\in T_i$ let $r_j$ be such that $r_j\leq |x-y|\leq 2r_j$.  Then by the weak estimate \eqref{e:neck_str:1} we have
\begin{align}\label{e:neck_str:3}
(1-C(n)\delta)|x-y|\leq |\Pi_{i,j}(x)-\Pi_{i,j}(y)|\leq (1+C(n)\delta)|x-y|\, .
\end{align}
In particular $|\Pi_{i,j}(x)-\Pi_{i,j}(y)|\leq 10r_j$ and thus we can estimate
\begin{align}
	(1-C(n)\delta)|\Pi_{i,j}(x)-\Pi_{i,j}(y)|&\leq|\Pi_j\circ\Pi_{i,j}(x)-\Pi_j\circ\Pi_{i,j}(y)|\notag\\
	&\leq (1+C(n)\delta)|\Pi_{i,j}(x)-\Pi_{i,j}(y)|\, .
\end{align}
Using that $\Pi_i=\Pi_j\circ\Pi_{i,j}$ and combining with \eqref{e:neck_str:3} we have finished the proof of the Claim. $\qed$\\

Let us now consider the bilipschitz maps $\Phi_i\equiv \Pi_i^{-1}:L\to T_i\subseteq \dR^n$.  Note that $\Phi_i(L)=T_i$ for every $i$.  Using the standard Ascoli convergence we can now limit
\begin{align}\label{e:neck_structure_proof:bilipschitz_map}
\Phi_i\to \Phi:L\to T\equiv \Phi(L)\subseteq \dR^n\, ,	
\end{align}
so that $\Phi$ satisfies the bilipschitz estimates of Claim 2.  The following are intuitively helpful:

\begin{exercise}
	Show $T_i\to T$ in the Hausdorff topology.
\end{exercise}

\begin{exercise}\label{exer:T_affine_approx}
Use Theorem \ref{t:approx_submanifold}.2 and the Ascoli convergence to show for $x\in T$ that $d_H(T\cap B_r(x),L_{x,r}\cap B_r(x))<C(n)\delta r$ for $r\geq r_x$, where $r_x$ is as in \eqref{e:regularity_scale}.
\end{exercise}

This finishes the proof of Theorem \ref{t:neck_structure}.2, so that we are left with needing to prove the mass estimate on the Neck Region from Theorem \ref{t:neck_structure}.3.  To begin let us associate to $T$ its Hausdorff measure
\begin{align}\label{e:neck_structure_proof:T_measure}
\mu_T \equiv \cH^k\cap T\, .
\end{align}

Recall from \eqref{e:regularity_scale} the regularity scale extension of $r_x$.  An important consequence of the bilipschitz estimate of Claim 2 is the Alhfors regularity of $\mu_T$ and control on the $\beta$-numbers of $\mu$ on balls centered on $T$: \\

{\bf Claim 3:  For each $x\in T$ and $r_x\leq r<4$ we have $(1-C(n)\delta)\omega_k r^k \leq \mu_T(B_r(x))\leq (1+C(n)\delta)\omega_k r^k$. and $\int_{r_x}^4 \beta_k(x,s)^2\frac{ds}{s}<C(n)\delta$.}\\

Let us first prove the Alhfors regularity.  Indeed, recall from Section \ref{s:rect_reif:haus_content} that if $U\subseteq \dR^n$ then
\begin{align}
	\cH^k(U) = \lim_{s\to 0} \cH^k_r(U) = \lim_{r\to 0}\inf\{\omega_k\sum r_i^k: U\subseteq \bigcup B_{r_i}(x_i)\text{ and }r_i\leq s\}\, .
\end{align}
Now consider $U=B_r(x)$ and let $U'=\Phi^{-1}(B_r(x))\subseteq \dR^k$.  Note by the bilipschitz condition on $\Phi$ that if $x'=\Phi^{-1}(x)$ then 
\begin{align}
	B_{(1-C(n)\delta)r}(x')\subseteq U'\subseteq B_{(1+C(n)\delta)r}(x')\, .
\end{align}
Note by the bilipschitz estimate we similarly have that for any covering $B_r(x)\subseteq \bigcup B_{r_i}(x_i)$:
\begin{align}
	U'\subseteq \bigcup B_{r'_i}(x'_i) \text{ where }x'_i=\Phi^{-1}(x_i) \text{ and }r'_i = (1+C(n)\delta)r_i\, ,
\end{align}
and notice for this covering we have
\begin{align}
\omega_k\sum (r'_i)^k =(1+C(n)\delta)^k\omega_k\sum r_i^k\, .
\end{align}
In particular, since this held for {\it any} covering of $B_r(x)$ we get the estimate
\begin{align}
\cH^k(B_r(x))&\geq (1-C(n)\delta)\cH^k(U')\geq (1-C(n)\delta)\cH^k(B_{(1+C(n)\delta)r}(x'))\notag\\
&\geq (1-C(n)\delta)\omega_k r^k\, ,
\end{align}
which proves half the estimate.  By instead beginning with any covering of $U'$ and using $\Phi$ to construct a covering of $B_r(x)$ we obtain the reverse estimate by a verbatim argument.

Now let us focus on the $\beta$-number estimate.  Indeed, let $y\in T$ and then by the definition of $r_y$ we can find $x\in \cC\cap \overline B_{r_y}(y)$ such that $r_x\leq \tau^{-2}r_y$.  Thus by Exercise \ref{exer:beta_continuity} we can estimate
\begin{align}
\int_{r_x}^4 \beta(x,s)^2 \frac{ds}{s} \leq C(n)	 \int_{\tau^{-2}r_y}^4 \beta(y,s)^2\frac{ds}{s}\leq C(n)\delta\, ,
\end{align}
as claimed.  $\qed$\\

Let us now turn our attention to the proof of the mass bound $\mu(\cN)<C(n)\delta$.  The main technical estimate is the following, which tells us we can locally control the mass in $\cN$ by the square of the $\beta$-numbers:\\

{\bf Claim 4: Let $y\in \cN$ with $d\equiv d(y,T)$, then $\mu\big(B_{d/10}(y)\big)<B(n)^2\int_{B_{2d}(y)}\beta_k(x,10d)^2 d\mu_T$}.\\

Recall by Claim 3 that $\mu_T(B_{2d}(y))\leq C(n) d^k$, and thus we can find $z\in T\cap B_{2d}(y)$ such that $\beta_k(z,10d)^2\leq C(n) d^{-k}\int_{B_{2d}(y)}\beta_k(x,2d)^2 d\mu_T$.  Now let $L=L_{z,2\tau^{-2}d}$ and note by Exercise \ref{exer:T_affine_approx} we have that $d(T\cap B_{2d}(z),L\cap B_{2d}(z))<<\frac{1}{10}d$.  On the other hand, let us take  $Y\in B_{d/10}(y)$ to be the generalized center of mass of $B_{d/10}(y)$, as in Definition \ref{d:generalized_COM}.  In particular, by the last sentence we must have $d(Y,L)>\frac{1}{2}d$.  

However, let us now assume $\mu\big(B_{d/10}(y)\big)\geq B(n)^2\int_{B_{2d}(y)}\beta_k(x,10d)^2 d\mu_T$, and in particular by applying Exercise \ref{exer:beta_continuity} we have
\begin{align}
\mu\big(B_{d/10}(y)\big)\geq \frac{B^2}{C^2} \beta_k(z,10 d)^2 d^k\geq \frac{B(n)^2}{C(n)^2}\beta_k(y,d/10)^2\,d^k\, .	
\end{align}
If we apply Lemma \ref{l:COM_control} we can then get that $d(Y,L)<\frac{C(n)}{B(n)} d$, so that if we have chosen $B(n)>>C(n)$ we get that $d(Y,L)<\frac{1}{2}d$, which is our desired contradiction. $\qed$\\

Let us now finish the proof that $\mu(\cN)<C(n)\delta$.  To begin let us apply the usual Vitali process (see lemma \ref{l:vitali}, Exercise \ref{exer:covering:2} and Exercise \ref{exer:covering:3}) to cover $\cN$:
\begin{align}
\cN\subseteq \bigcup_{y\in \cN}B_{d_y/50}(y)\subseteq \bigcup_i B_{d_i/10}(y_i)\, ,	
\end{align}
where $d_y\equiv d(y,T)$ and $\{B_{d_i/50}(y_i)\}$ are disjoint.
\begin{exercise}
Show for $\alpha\in\dN$ that if we consider the collection of balls $\{B_{2d_i}(y_i)\}$ with $2^{-\alpha}\leq d_i\leq 2^{-\alpha+1}$, then each $z\in \dR^n$ intersects at most $C(n)$ balls from this collection.  Hint:  See Exercise \ref{exer:covering:2}.
\end{exercise}

Now using Claim 4 we can estimate

\begin{align}
\mu(\cN) &\leq \sum_i \mu(B_{d_i/10}(y_i))\leq B(n)^2\sum_i	 \int_{B_{2d_i}(y_i)}\beta_k(x,10d_i)^2 d\mu_T\, ,\notag\\
&= B(n)^2\sum_\alpha \sum_{2^{-\alpha}\leq d_i<2^{-\alpha+1}}	 \int_{B_{2d_i}(y_i)}\beta_k(x,10d_i)^2 d\mu_T\, ,\notag\\
&\leq C(n) \int_{B_1}\sum_{2^{-\alpha}\geq r_x} \beta_k(x,10^2\cdot 2^{-\alpha})^2 d\mu_T\, ,\notag\\
&\leq C(n) \int_{B_1}\int_{r_x}^2\beta_k(x,s)^2\frac{ds}{s}d\mu_T \leq C(n)\delta \mu_T(B_1)\notag\\
&\leq C(n)\delta\, ,
\end{align}
as claimed, which finishes the proof of Theorem \ref{t:neck_structure}. $\qed$

\vspace{.3cm}

\subsection{Proof of Submanifold Approximation Theorem \ref{t:approx_submanifold}}

What remains is to prove the Submanifold Approximation Theorem \ref{t:approx_submanifold}.  From this point forward the proof will mimick very closely the proof of the classical Reifenberg we presented in Lecture 1.  We will first need to prove the subspace selection lemma, which is essentially nothing more than smoothing out our local choices of best subspaces at each point and scale.  We will use this smooth collection of subspaces in order to define our approximate distance function, which will itself be a morse bott function with good estimates.  The zero level set of these approximate distance functions will be our choice of submanifolds $T_r$.  Because most of the proofs are almost verbatim to those in Lecture 1 we will mainly emphasize the slight changes (which are mostly in the form of what scale we are working on using a more local error), and leave the rest as an exercise.  

\subsubsection{Subspace Selection Lemma}
Let us begin with a statement of the Subspace Selection Lemma:

\begin{lemma}[Subspace Selection Lemma]\label{l:subspace_selection}
	Let $\cN = B_1\setminus \overline B_{r_x}(\cC)$ be a $(k,\delta,\epsilon,\nu )$-neck region with $r>0$.  Then for each $y\in B_1$ there exists a $k$-dimensional affine subspace $L_{y}$ where if $\hat \pi_y = \hat \pi_{L_y}$ is the linear projection map and $m_y\equiv \pi_y[y]$ then:
\begin{enumerate}  
\item $L_{y}$ varies smoothly in $y$ with $\overline r_y|\partial_i \hat\pi_{y}|, |\partial_i m_y - \hat\pi_y|\leq C(n,\epsilon,\nu)\beta_k(y,\tau^{-1} \overline r_y)$ and $r^2_y |\partial_i\partial_j \hat\pi_{y}|, r_y|\partial_i\partial_j m_y|\leq C(n,\epsilon,\nu)\beta_k(y,\tau^{-1}\overline r_y)$. 
\item We have $L_y\cap B_{10\overline r_y}(y)\subseteq B_{C(n)\tau \bar r_x}(\cC)$ and $\cC\cap B_{10\overline r_y}(y))\subseteq B_{C\delta \overline r_x}(L_y)$.
\item We have $d_H(L_y\cap B_{10\bar r_y}(y),L_{y,\tau^{-1}\bar r_y}\cap B_{10\bar r_y}(y))<C(n)\beta_k(y,\tau^{-1}\bar r_y)$.
\end{enumerate}
\end{lemma}

The proof of the subspace selection lemma follows almost exactly from Section \ref{ss:class_reif:dist_approx}.  The main technical distinction is that we work on scale $\tau$ instead of scales which are powers of $10$, and our estimates are now in terms of $\beta_k$ on that scale instead of $\delta$, however this changes almost nothing in the line by line.  We begin by building a partition of unity on the ball $B_2$ which will be used for our gluing process: 

\begin{lemma}\label{l:rect_reif:partition}
There exists a covering $B_2\subseteq \bigcup_\alpha B_{r_\alpha}(x_\alpha)$, where $r_\alpha \equiv \tau^2 \bar r_{x_\alpha} $, and smooth nonnegative functions $\phi_\alpha$ such that
\begin{enumerate}
\item $\{ B_{\frac{1}{4}r_\alpha}(x_\alpha)\}$ are disjoint.
\item For each $y\in B_2$ we have $\#\{x_\alpha: y\in B_{4 r_\alpha}(x_\alpha)\}<C(n)$.
\item $\sum \phi_\alpha = 1$ on $B_2$ with $\text{supp}\,\phi_\alpha \subseteq B_{4 r_\alpha}(x_\alpha)$.
\item $|\partial^{(k)}\phi_\alpha|\leq C(n,k) \tilde  r_\alpha^{-k}$.
\end{enumerate}
\end{lemma}
\begin{proof}
See Lemma \ref{l:class_reif:partition} and Exercise \ref{exer:covering:3}.
\end{proof}

In order to prove the Subspace Selection Lemma we follow the path of Lecture 1.  Using the partition from above let us define for each $\alpha$ the affine subspace
\begin{align}
L_\alpha \equiv L_{x_\alpha, 10^{-1}\tau^{-1}r_\alpha}	\, .
\end{align}
Morally, we simply want to define $L_y\equiv \sum \phi_\alpha L_\alpha$ and check what estimates hold.  To do this by hand we make the observation, as in Section \ref{ss:class_reif:dist_approx}, that to construct an affine subspace $L_y$ one requires constructing a point $\ell_y\in L_y$ and a linear subspace $\hat\pi_y$.  We define them by the formulas:
\begin{align}
	&\ell_y\equiv \sum_\alpha \phi_\alpha(y) \pi_\alpha[y]\, ,\notag\\
	&M_y\equiv \sum_\alpha \phi_\alpha(y)\hat \pi_\alpha\, ,\notag\\
	&\hat\pi_y \equiv \text{span}\{e_1(y),\ldots e_k(y)\}\, ,
\end{align}
where $e_1(y),\ldots,e_k(y)$ are the $k$-largest eigenvectors of $M_y$.  The proof of the Subspace Selection Lemma is now the same as in Section \ref{ss:class_reif:dist_approx}.

\vspace{.3cm}

\subsubsection{Distance Approximation Theorem}

The statement and proof of the Distance Approximation Theorem is almost verbatim as in Section \ref{ss:class_reif:dist_approx}.  Recall the definition 

\begin{align}\label{e:approx_distance}
\Phi_r(y)\equiv \frac{1}{2}d(y,L_y)^2 = \frac{1}{2}|y-\pi_y(y)|^2\equiv \frac{1}{2}|y-m_y|^2\, .
\end{align}

Then the main estimates on $\Phi_r$ are the following:

\begin{theorem}[Approximate Distance Function]\label{t:approx_dist}
Let $\cN = B_1\setminus \overline B_{r_x}(\cC)$ be a $(k,\delta,\nu,\epsilon )$-neck region with $\bar r_y$ defined in \eqref{e:regularity_scale} and $\Phi_r$ defined in \eqref{e:approx_distance}.  Then for each $y\in B_2$ the following is satisfied:
	\begin{enumerate}
	\item For $\ell\in L_{y}\cap B_{10^2 \overline r_y}(y)$ $\exists!$ $z_\ell\in \hat L^\perp_{y}+\ell$ such that $\Phi_r(z_\ell)=0$.
	\item $\Big||\nabla \Phi_r|^2 - 4\Phi_r\Big|(y)\leq C(n,\nu,\epsilon )\beta_k(y,\tau^{-1}\overline r_y)\, \Phi_r(y)$.
	\item $|\nabla^2\Phi_r(y)-\hat\pi_{y}^\perp|<C(n,\nu,\epsilon )\beta_k(y,\tau^{-1}\overline r_y)$.
	\item $|\nabla^{(k)}\Phi|(y)\leq C(n,k,\nu,\epsilon )\beta_k(y,\tau^{-1}\overline r_y)\, \overline r_y^{2-k}$ for $k\geq 3$.
	\end{enumerate}
\end{theorem}
\begin{proof}
	With the Subspace Selection Lemma \ref{l:subspace_selection} the proof of Theorem \ref{t:approx_dist} is almost verbatim from Section \ref{ss:class_reif:dist_approx}.

\end{proof}

\vspace{.3cm}

\subsubsection{Proof of Submanifold Approximation Theorem \ref{t:approx_submanifold}}

With the the construction of the approximate distance function and studied some of its properties in hand, the proof of the Submanifold Approximation Theorem is now verbatim that in Section \ref{ss:class_reif:subman_approx}.

\vspace{.5cm}

\section{Proof of Neck Decomposition Theorem \ref{t:neck_decomposition}}

We now focus our attention on the proof of Theorem \ref{t:neck_decomposition}.  The proof is essentially just an involved covering argument, and was first introduced in \cite{JiNa_L2},\cite{NaVa_EnId}.  The sharp content estimates of Theorem \ref{t:neck_decomposition}.3 will be a consequence of the Neck Structure Theorem from the previous section.  Let us first restate the theorem for the readers convenience:\\

{\bf Theorem. }[Theorem \ref{t:neck_decomposition} Restated]
	Let $\mu$ be a Borel measure on $B_1$ and assume for each $x\in B_1$ we have $\int_0^{2} \beta_k(x,s)^2\frac{ds}{s}\,d\mu\leq \Gamma$.
 Then for each $\nu,\epsilon ,\delta>0$ with $\epsilon<\epsilon(n)$ and $\delta<\delta(n,\epsilon,\nu )$ $\exists$ a covering $B_1 \subseteq \cS^{-}\cup \cS^k\cup \cS^{+}\,  \text{ with }$
\begin{align}
&\cS^{+} = \bigcup_a \big(\cN_a\cap B_{r_a}\big) \cup \bigcup_b B_{r_b}(x_b)\, \text{ and }\;\; \cS^k = \bigcup_a \cC_{0,a}\, ,
\end{align}
and such that
\begin{enumerate}
\item 	$\cN_a = B_{2r_a}(x_a)\setminus \overline B_{r_{a,x}}(\cC_a)$ are $(k,\delta,\epsilon,\nu )$-neck regions.  In particular, $\mu(\cN_a)\leq C(n)\delta\, r_a^k$ and $\cC_{0,a}$ are $k$-rectifiable by Theorem \ref{t:neck_structure}.
\item $B_{r_b}(x_b)$ satisfies the measure constraint $\mu(B_{2r_b})<C(n)\nu \, r_b^k$ .
\item We have the content estimates $\sum r_a^k + \sum r_b^k < C(n,\delta,\epsilon,\nu ,\Gamma)$ and packing estimate $\cP^k(\cS^-\cup\cS^k)<C(n,\delta,\epsilon,\nu ,\Gamma)$,
\item We have the Hausdorff measure estimate $\cH^{k}(\cS^{-})=0$.
\end{enumerate}

\vspace{.2cm}

We begin by discussing a variety of notation which will be convenient throughout the proof:

\begin{definition}
We define the $k$-dimensional distortion $D_k(x,r)$ of a measure $\mu$ by
\begin{align}
D_k(x,r) = \int_0^r \beta_k(x,s)^2 \frac{ds}{s}\, .	
\end{align}
\end{definition}

The following short exercises give some good intuition for the basic properties and behavior of $D_k$:
\begin{exercise}
Show the following:
\begin{enumerate}
\item $D_k(x,r)$ is monotone in $r$ and $D_k(x,r)=D_k(x,s)$ for some $s<r$ iff $\text{supp}\mu\cap B_r(x) \subseteq L^k$ for some $k$-dimensional affine subspace.
\item $C(n)^{-1}\beta_k(x,r)\leq D_k(x,2r)-D_k(x,r)\leq C(n) \beta_k(x,2r)$.
\item If $r_i=2^{-i}$ then $C(n)^{-1}\sum_{r_i\leq r} \beta_k(x,r_i)^2 \leq D_k(x,r)\leq C(n)\sum_{r_i\leq 2r} \beta_k(x,r_i)^2$.
\end{enumerate}
\end{exercise}

There are two primary pieces of information to keep track of during the proof.  The first is the distortion drop from scale to scale, the second is a lower mass bound on balls.  We formalize this with the following noncollapsing set:\footnote{Recall $\omega_n$ is the volume of the unit ball in $\dR^n$.}

\begin{definition}[Noncollapsing Set]\label{d:noncollapsing_set}
	Let $\epsilon, \nu>0$ be fixed, then we define the set of noncollapsed points:
\begin{align}
	V(x,r) \equiv \{y\in B_{r}(x):  \mu(B_{s}(y))> \nu\, s^k \text{ for } \epsilon r\leq s\leq r \}\, .\notag
\end{align}
\end{definition}
\begin{remark}
Recall the definition of noncollapsing as in Definition \ref{d:noncollapsing}.
\end{remark}

The following tells us that $V(x,r)$ must always live close to a best approximating subspace:

\begin{exercise}\label{exer:noncollapsed_subspace_approx}
If $\beta_k(x,2r)<\delta^2$ and $L=L_{x,2r}$ is a best affine subspace obtaining $\beta_k(x,2r)$, then for each $y\in V(x,r)$ we have $d(y,L)< (\epsilon+C(n,\epsilon)\nu^{-1/2}\delta^2)r$.  Hint:  Apply Lemma \ref{l:COM_control} to the center of mass $Y\in B_{\epsilon r}(y)$ and use the triangle inequality.
\end{exercise}

\vspace{.2cm}

\subsection{Proof Outline and Induction Step}

The Proof of Theorem \ref{t:neck_decomposition} will be done inductively on the size of the distortion.  

\vspace{.3cm}

\subsection{Notation and Ball Types}\label{ss:neck_decomp:notation}

In our notation for the proof during this section the subscript we use to denote a ball will always designate special structure of that ball.  All ball types will fall into the following categories:
\begin{enumerate}
\item[(a)] A ball $B_{r_a}(x_a)$ is associated with a $(k,\delta,\epsilon,\nu)$-neck region $\cN_a = B_{2r_a}(x_a)\setminus \overline B_{r_{a,x}}(\cC_a)$.
\item[(b)] A ball $B_{r_b}(x_b)$ satisfies $\mu(B_{2r_b}(x_b))<\nu r_b^k$.
\item[(c)] A ball $B_{r_c}(x_c)$ is such that $\beta_k(x_c,4r_c)<\delta^2$ and $V(x_c,r_c)$ is a $(k,2\epsilon)$-linearly independent set.
\item[(d)] A ball $B_{r_d}(x_d)$ is such that $V(x_d,r_d)$ is not a $(k,2\epsilon)$-linearly independent set.
\item[(e)] A ball $B_{r_e}(x_e)$ is such that $\beta_k(x_e,4r_e)>\delta^2$.
\item[(s)] A ball $B_{r_s}(x_s)$ is such that for each $y\in B_{r_s}(x_s)$ we have $D_k(y,2r_s)<\overline D - \delta^6$.  \footnote{In practice $\bar D=\sup_{B_R}D_k(x,R)$ will be the distortion of a potentially much larger ball.}
\item[(f)] A ball $B_{r_f}(x_f)$ is one for which we know nothing about.
\end{enumerate}


Before continuing let us discuss a little the role of each of these ball types.  The simplest two types are the $(a)$ and $(b)$ balls, as of course these are what we are wanting to construct in the theorem and there will be nothing left to do with them.  

Part of the proof will involve an induction on $D_k(x,r)$.  Therefore the $(s)$-balls will represent balls for which the distortion has strictly dropped, and therefore we will apply our inductive hypothesis to handle them.  Thus in practice we are also done on $s$-balls as well.  The $(f)$-balls will also require starting over on, however in practice when we label a ball an $(f)$-ball we will make sure it is only on a set of very small context, therefore starting over will be okay as the errors will become a geometric series, see Section \ref{ss:inductive_proof}.

The next easiest ball types to deal with are the $(d)$ and $(e)$ balls.  For a $(d)$-ball we will be able to cover all of $B_{r_d}(x_d)$ by $(b)$-balls away from a set of very small context of $(f)$-balls.  For an $(e)$-ball we will be able to entirely cover $B_{r_e}(x_e)$ by $(s)$-balls for which the distortion has strictly dropped, and thus we will be able to apply our inductive hypotheses to these new balls.

The most complicated ball type to deal with in the construction is therefore a $(c)$-ball.  The goal will be to build a neck region so that $B_{r_c}(x_c)\subseteq \cN \cup \overline B_{r_x}(\cC)$.  In order to proceed with the next step, this will have to be done in a maximal fashion so that each ball $B_{r_x}(x)$ with $x\in \cC$ is either a $(b)$,$(d)$, $(e)$ or $(s)$ ball.  Then using the Neck Structure Theorem \ref{t:neck_structure} in combination with the $(d)$ and $(e)$ coverings just discussed we can estimate the content of those balls we need to start over as being small.

\vspace{.2cm}
\subsection{Collapsing and $d$-Ball Covering}

Recall the definition of the noncollapsing set from Definition \ref{d:noncollapsing_set}.

\begin{proposition}[$d$-Ball Covering]\label{p:d_ball_covering}
Let $B_{r_d}(x_d)$ be such that $V(x_d,r_d)$ is not a $(k,2\epsilon)$-linearly independent set, then we can cover
\begin{align}
B_{r_d}(x_d)\subseteq \bigcup_b B_{r_{b}}(x_b)\cup \bigcup_f B_{r_{f}}(x_f)\, ,
\end{align}
such that $\mu(B_{r_{b}}(x_b))<\nu r_b^k$ for each $(b)$-ball and we have the estimates $\sum_b r_b^k < C(n,\epsilon) r_d^k$ and $\sum_f r_f^k < C(n)\epsilon\, r_d^k$.
\end{proposition}

\begin{proof}
Since $V(x_d,r_d)$ is not a $(k,2\epsilon)$-linearly independent set we can find a subspace $L^{k-1}$ such that $V(x_d,r_d)\subseteq B_{2\epsilon r_d}(L)$.  In particular, let $\{x_f\}\in B_{r_d}\cap L$ be $\epsilon r_d$-dense with $r_f\equiv 4\epsilon r_d$.  We see the $(f)$-balls satisfy the required property.  Now for each $x\not\in V(x_d,r_d)$ let $r_x\geq \frac{1}{2}\epsilon r_d$ be such that $\mu(B_{2r_x}(x))<\nu r_b^k$.  Let $\{B_{r_b}(x_b)\}$ be any maximal subset such that $B_{10^{-1}r_b}(x_b)$ are disjoint.  Then a simple volume estimate\footnote{See Exercise \ref{exer:covering:1} and Lemma \ref{l:class_reif:partition}.} shows that $\{B_{r_b}(x_b)\}$ satisfies the conditions of the proposition.
\end{proof}

\vspace{.2cm}

\subsection{Symmetry and $e$-Ball Covering}

\begin{proposition}[$e$-Ball Covering]\label{p:e_ball_covering}
Let $B_{r_e}(x_e)$ be such that $\beta_k(x_e,4r_e)>\delta^2$.  Then we can cover
\begin{align}
B_{r_e}(x_e)\subseteq \bigcup_b B_{r_{s}}(x_s)\, ,
\end{align}
such that $\sum r_s^k \leq C(n) r_e^k$ and $D_k(y,r_s)<D_k(y,10r_s)-C(n)\delta^4$ for each $y\in B_{r_s}(x_s)$.
\end{proposition}

\begin{proof}
Let $r_s \equiv r_e$ with $\{x_s\}\in B_{r_e}(x_e)$ a maximal subset such that $B_{10^{-1}r_s}(x_s)$ are disjoint.  In particular, $\{B_{r_s}(x_s)\}$ is a covering of $B_{r_e}(x_e)$ and $\sum r_s^k \leq C(n) r_e^k$.  Finally, let $y\in B_{2r_e}(x_e)$ and consider
\begin{align}
D_k(y,10r_s)-D_k(y,r_s) &= \int_{r_e}^{10r_e}\beta_k(y,r)^2 \frac{dr}{r}\geq \int_{6r_e}^{10r_e}\beta_k(y,r)^2 \frac{dr}{r}\notag\\
&\geq C(n)\beta_k(x_e,4r_e)^2>C(n)\delta^4\, ,\notag
\end{align}
where the last line uses Exercise \ref{exer:beta_continuity}.
\end{proof}

\vspace{.2cm}
\subsection{Neck Regions and $c$-Ball Covering}

\begin{proposition}[$c$-Ball Covering]\label{p:c_ball_covering}
Let $B_{r_c}(x_c)$ be such that $\beta_k(x_c,4r_c)\leq \delta^2$ and $V(x_c,r_c)$ is a $(k,2\epsilon)$-linearly independent set.  Then for $\delta<\delta(n,\epsilon,\nu)$ there exists a $(k,\delta,\epsilon,\nu)$-neck region 
\begin{align}
\cN = B_{2r_c}(x_c)\setminus \bigcup_x \overline B_{r_x}(\cC)\, ,
\end{align}
such that for each $x\in \cC$ with $r_x>0$ we have that $B_{r_x}(x)$ satisfies one of the following:
\begin{enumerate}
\item[(d)] $B_{r_x}(x) = B_{r_d}(x_d)$ is such that $V(x,r_x)$ is not a $(k,2\epsilon)$-linearly independent set.
\item[(e)] $B_{r_x}(x) =B_{r_e}(x_e)$ is such that $\beta_k(x,4r_x)>\delta^2$.
\item[(s)] $B_{r_x}(x) =B_{r_s}(x_s)$ is such that for $y\in B_{r_x}(x)$ we have $D_k(y,2r_x)<\sup_{B_{2r_c}(x_c)}D_k(y,4r_c) - \delta^6$.
\end{enumerate}
That is, $B_{r_x}(x)$ is either a $(d)$, $(e)$ or $(s)$-ball.
\end{proposition}

\begin{proof}
The proof is purely constructive, and will be done inductively on scales $s_\alpha = \tau^{\alpha}r_c$.  At each step of the induction we will have built a $(k,\delta,\epsilon,\nu)$-neck region $\cN_\alpha=B_{2r_c}\setminus \overline B_{r^\alpha_{x}}(\cC_\alpha)$, and for the next step of the induction we will recover those balls $\{B_{r^\alpha_x}(x^\alpha)\}$ which are not $(d)$, $(e)$ or $(s)$-balls.\\

{\bf Construction Step:}. Let us begin by describing the constructive step which will be applied each time we need to recover a ball.  Let $B_r(x)$ be a ball for which none of the conditions $(d)$, $(e)$ or $(s)$ hold.  In this case let us define the best affine subspace
\begin{align}
L_{x,4r} \equiv \arg\min_{L}\, (4r)^{-2-k}\int_{B_{4r}(x)} d(y,L)^2\,d\mu[y]\, ,
\end{align}
so that $\beta(x,4r)^2\leq \delta^2$ is obtained by $L_{x,4r}$.  Note by Exercise \ref{exer:noncollapsed_subspace_approx} that for $\delta<\delta(n,\epsilon,\nu)$ we have
\begin{align}
V(x,2r)\subseteq B_{4\epsilon\, r}(L_{x,4r})\, .
\end{align}

Now for each $y\in L_{x,4r}\cap B_{2r}(x)$ let 
\begin{align}
r_y\equiv \inf_s\{\tau r\leq s\leq r: V(y,s) \text{ is $(k,2\epsilon)$-linearly independent in }B_s(y)\}\, .	
\end{align}
Now let $\{B_{r_{i}}(x_i)\}\subseteq \{B_{r_y}(y)\}_{y\in L}$ be a maximal subcollection such that $\{B_{\tau^2 r_{i}}(x_i)\}$ are disjoint.  For future notational use we define $\cC_{x,2r}=\{x_i\}$ with $r_{x_i}=r_i\geq \tau r$.  In particular, $L_{x,4r}\cap B_{2r}(x)\subseteq B_{10^{-1}\tau r}(\cC_{x,2r})$ and $\cC_{x,2r}\subseteq L_{x,4r}$.\\

{\bf Base Step:}  Let us start with the beginning ball $B_{r_c}(x_c)$ and observe that by assumption the hypotheses of the constructive step apply.  Thus we can apply the constructive step and define $\cC_0 \equiv \cC_{x_c,2r_c}$ to be the center points with the radius function $r^0_x$.  The following exercise follows immediately from the construction:

\begin{exercise}
Show that $\cN_1\equiv B_{2r_c}(x_c)\setminus \overline B_{r^0_x}(\cC_0)$ is a $(k,\delta,\epsilon,\nu)$-neck region. \\
\end{exercise}

{\bf Inductive Step:}  Let us now assume we have constructed a $(k,\delta,\epsilon,\nu)$-neck region $\cN_\alpha=B_{2r_c}\setminus \overline B_{r^\alpha_{x}}(\cC_\alpha)$.  We wish to define from this the next step of the inductive process.  Let us first break the centerpoints $\cC_\alpha$ into two groups:
\begin{align}
&\cC_\alpha = \cC^g_\alpha\cup \cC^b_\alpha\, ,\notag\\
&\cC^g_\alpha = \{x\in \cC: B_{r^{\alpha}_{x}}(x) \text{ satisfies either condition $(d)$, $(e)$, or $(s)$}\}	\, ,\notag\\
&\cC^b_\alpha = \{x\in \cC: r^\alpha_x = \tau^\alpha r_c \text{ and } B_{r^{\alpha}_{x}}(x) \text{ does not satisfy conditions $(d)$, $(e)$, or $(s)$}\}\, .
\end{align}
We will let $\cC^g_\alpha\subseteq \cC_{\alpha+1}$ with $r^{\alpha+1}_{x}\equiv r^\alpha_{x}$ for each $x\in \cC^g_\alpha$.  It is then the balls represented by $\cC^b_\alpha$ which must be recovered.

Thus for each $x\in \cC^b_\alpha$ we can apply the Constructive Step in order to produce a best subspace $L_{x}=L_{x,4r^\alpha_x}$ and an associated collection of balls $\{B_{r^{\alpha+1}_{y}}(y)\}$, where $y\in \cC_{x,2r^\alpha_x}$ and $r^{\alpha+1}_{y} \geq \tau^{\alpha+1} r_c$.  Note that by construction we have that $\{B_{\tau^2 r^{\alpha+1}_{y}}(y)\}$ are disjoint and $\cC_{x,2r^\alpha_x}\subseteq L_x$.  If we do this for each ball from $\cC^b_\alpha$ we can choose a maximal subset:
\begin{align}
&\tilde\cC_{\alpha+1}\subseteq  \bigcup_{x\in\cC^b_\alpha}\cC_{x,2r^\alpha_x}	\, , \text{ such that }\notag\\
&\{B_{\tau^2 r^{\alpha+1}_x}(x)\}_{\tilde\cC_{\alpha+1}}\cup \{B_{\tau^2 r^{\alpha+1}_x}(x)\}_{\cC^g_{\alpha}} \text{ are disjoint}.
\end{align}

We now define $\cC_{\alpha+1}\equiv \tilde \cC_{\alpha+1}\cup \cC^g_\alpha$.  It is somewhat tedious but otherwise straightforward to check the following:
\begin{exercise}
Show $\cN_{\alpha+1}=B_{2r_c}\setminus \overline B_{r^{\alpha+1}_{x}}(\cC_{\alpha+1})$ is a $(k,\delta,\epsilon,\nu)$-neck region.
\end{exercise}

{\bf Limiting Step.}  We now want to finish the proof of Proposition \ref{p:c_ball_covering} by taking a limit $\cN=\lim\cN_\alpha$.  Let us see how this limit is carefully constructed.  Note first that $\cC_\alpha$ is a sequence of bounded closed subsets and thus after passing to a subseqence we can take a Hausdorff limits
\begin{align}
\cC = \lim_{\alpha\to  \infty} \cC_\alpha\, .	
\end{align}

Let us define the radius function $r_x:\cC\to \dR^+$ as follows.  For $x\in \cC$ if for all $\alpha$ sufficiently big we have $x\in \cC_\alpha$ with $r_x^\alpha=r_x^{\alpha+1}$, then we define $r_x\equiv \lim r^\alpha_x$.  Otherwise for $x\in \cC$ we define $r_x\equiv 0$.  Though we do not strictly need it, the following is a useful exercise for building intuition about the construction:
\begin{exercise}
Show one does not need to pass to a subsequence in order to take the Hausdorff limit $\cC = \lim\cC_\alpha$.	Show that if $x^\alpha\to x$ with $x^\alpha\in\cC^\alpha$ then $r^\alpha_x\to r_x$. 
\end{exercise}

It is clear that conditions $(n1)$, $(n2)$ and $(n3)$ all pass to the limits in this construction so that $\cN\equiv B_{2r_c}\setminus\overline B_{r_x}(\cC)$ is a $(k,\delta,\epsilon,\nu)$-neck region.  What is left is to check that if $r_x>0$ then $B_{r_x}(x)$ satisfies one of the conditions $(d)$, $(e)$, or $(s)$.  However, by construction if $r_x>0$ then there exists $\alpha$ for which $x\in \cC^\alpha$ with $r_x=r^\alpha_x$.  In the inductive step the ball $B_{r^\alpha_x}(x)$ remained in $\cC_{\alpha+1}$ only if it satisfied $(d)$, $(e)$, or $(s)$.  This finishes the proof of Proposition \ref{p:c_ball_covering}. 
\end{proof}

\vspace{.5cm}

\subsection{Inductive Proof of the Neck Decomposition Theorem}\label{ss:inductive_proof}

We now finish the proof of the Neck Decomposition Theorem.  The idea will be simply to continuously apply the covering propositions of the previous subsections.  We begin by applying each of the Propositions once in order to get a first covering:

\begin{proposition}[Induction Step 1]\label{p:inductive1}
	Let $\mu$ be a Borel measure and assume for each $x\in B_2$ that $\int_0^4 \beta_k(x,r)^2\frac{dr}{r}\leq \Gamma$ .  Then for each $\nu ,\epsilon,\delta>0$ with $\epsilon<\epsilon(n,\nu)$ and $\delta<\delta(n,\nu,\epsilon)$ $\exists$ a covering $$B_1 \subseteq \cN\cup\cC_0\cup \bigcup_b B_{r_b}(x_b)\cup \bigcup_s B_{r_s}(x_s)\cup \bigcup_f B_{r_f}(x_f)$$ 
such that
\begin{enumerate}
\item 	$\cN = B_{2}\setminus \overline B_{r_{x}}(\cC)$ is a $(k,\delta,\epsilon,\nu)$-neck region.  In particular, $\mu(\cN)\leq C(n)\delta$ and $\cC_{0}$ is $k$-rectifiable by Theorem \ref{t:neck_structure}.
\item $B_{r_b}(x_b)$ satisfy the measure constraints $\mu(B_{2r_b})<\nu \, r_b^k$ .
\item For each $y\in B_{r_s}(x_s)$ we have $D_k(y,2r_s)<\Gamma-\delta^6$.
\item We have the content estimates $\cH^k(\cC_0)+ \sum r_b^k + \sum r_s^k < C(n,\epsilon)$,
\item We have the content estimate $\sum r_f^k\leq C(n)\epsilon$.
\end{enumerate}	
\end{proposition}
\begin{remark}
The Neck Decomposition will follow by repeated applications of this Proposition.  The balls $B_{r_s}(x_s)$ have a drop in the distortion, and thus we can handle them later by an induction argument.  The balls $B_{r_f}(x_f)$ we know nothing about, however this is a set with small $k$-content.  Therefore in Inductive Step 2 we will start over on them and all errors become a geometric series which converge.
\end{remark}

\begin{proof}
	We begin with the ball $B_1$ and observe that it is either a $(c)$, $(d)$, $(e)$ or $(s)$ ball, as in Subsection \ref{ss:neck_decomp:notation}.  If $B_1$ is a $(d)$, $(e)$ or $(s)$ ball then the Proposition is immediately proved simply by applying either Proposition \ref{p:d_ball_covering} or Proposition \ref{p:e_ball_covering}.  Thus we will assume $B_1$ is a $(c)$-ball.
	
	Now we can apply Proposition \ref{p:c_ball_covering} in order to build a Neck Region $\cN = B_2\setminus \overline B_{r_x}(\cC)$ such that each ball $B_{r_x}(x)$ with $r_x>0$ is either a $(d)$, $(e)$ or $(s)$-ball.  This gives us the covering
\begin{align}
B_1\subseteq \cN\cup\cC_0\cup\bigcup_d B_{r_d}(x_d)\cup \bigcup_e B_{r_e}(x_e)\cup\bigcup_s B_{r_s}(x_s)\, . 	
\end{align}
By applying the Neck Structure Theorem \ref{t:neck_structure} we get the estimates
\begin{align}
	\cH^k(\cC_0)+\sum_d r_d^k+\sum_e r_e^k+\sum_s r_s^k\leq C(n)\, .
\end{align}
Now we wish to remove the $(d)$ and $(e)$ balls from this covering and control what is left.  If we apply Proposition \ref{p:e_ball_covering} we can cover each $e$-ball $B_{r_e}(x_e)\subseteq \bigcup_s B_{r_{es}}(x_{es})$ such that $\sum_s r^k_{es}\leq C(n) r^k_e$.  In particular the new collection of $s$ balls satisfies the estimate
\begin{align}
\sum r_{es}^k+\sum r_s^k \leq C(n)+C(n)\sum_e r_e^k\leq C(n)\, .	
\end{align}
Combining these $(s)$-balls together together gives the covering
\begin{align}
B_1\subseteq \cN\cup\cC_0\cup\bigcup_d B_{r_d}(x_d)\cup\bigcup_s B_{r_s}(x_s)\, ,
\end{align}
with the estimates
\begin{align}
	\cH^k(\cC_0)+\sum_d r_d^k+\sum_s r_s^k\leq C(n)\, .
\end{align}

We can now apply Proposition \ref{p:d_ball_covering} in order to cover each $d$-ball $B_{r_d}(x_d)\subseteq \bigcup_b B_{r_{db}}(x_{db})\cup\bigcup_f B_{r_{df}}(x_{df})$ such that $\sum_b r^k_{eb}\leq C(n,\epsilon) r^k_d$ and $\sum_s r^k_{df}\leq C(n)\epsilon\, r^k_d$.  Combining all of these together we get the covering
\begin{align}
B_1\subseteq \cN\cup\cC_0\cup\bigcup_b B_{r_b}(x_b)\cup\bigcup_s B_{r_s}(x_s)\cup\bigcup_f B_{r_f}(x_f)\, ,	
\end{align}
where $\sum_b r_b^k+\sum_s r_s^k\leq C(n,\epsilon)$ and $\sum_f r_f^k\leq C(n)\epsilon\sum r_d^k\leq C(n)\epsilon$ as claimed.
\end{proof}

\vspace{.5cm}

Our next step in our inductive proof to get rid of the $(f)$-balls in the covering.  As there is only a small content of $(f)$-balls, the trick is to count the errors which appear and see that form a geometric series.  As not everything goes away in the limit, we may be left with a $\cH^k$ null set:

\begin{proposition}[Induction Step 2.]\label{p:inductive2}
	Let $\mu$ be a Borel measure and assume for each $x\in B_2$ that $\int_0^4 \beta_k(x,r)^2\frac{dr}{r}\leq \Gamma$ .  Then for each $\nu ,\epsilon,\delta>0$ with $\epsilon<\epsilon(n,\nu)$ and $\delta<\delta(n,\nu,\epsilon)$ $\exists$ a covering $$B_1 \subseteq \bigcup_a \big(\cN_a\cap B_{r_a}\big) \cup\bigcup_a(\cC_{0,a}\cap B_{r_a}\big)\cup\cS^- \cup \bigcup_b B_{r_b}(x_b)\cup \bigcup_s B_{r_s}(x_s)$$ 
such that
\begin{enumerate}
\item 	$\cN_a = B_{2r_a}\setminus \overline B_{r_{x}}(\cC_a)$ are $(k,\delta,\epsilon,\nu)$-neck regions.  In particular, $\mu(\cN_a)\leq C(n)\delta r_a^k$ and $\cC_{0,a}$ are $k$-rectifiable by Theorem \ref{t:neck_structure}.
\item $B_{r_b}(x_b)$ satisfy the measure constraints $\mu(B_{2r_b})<\nu \, r_b^k$ .
\item For each $y\in B_{r_s}(x_s)$ we have $D_k(y,2r_s)<\Gamma-\delta^6$.
\item We have the measure estimate $\cH^k(\cS^-)=0$.
\item We have the content estimates $\cH^k(\cC_0)+\sum r_a^k + \sum r_b^k + \sum r_s^k < C(n,\epsilon)$,
\end{enumerate}	
\end{proposition}
\begin{proof}
To get from Proposition \ref{p:inductive1} to Proposition \ref{p:inductive2} we will simply continually recover the $f$-balls.  Indeed, let us begin by applying Proposition \ref{p:inductive1} to $B_1$ to get the covering
\begin{align}
B_1 \subseteq \cN\cup\cC_0\cup \bigcup_b B_{r^0_b}(x_b)\cup \bigcup_s B_{r^0_s}(x^0_s)\cup \bigcup_f B_{r^0_f}(x^0_f)\, ,	
\end{align}
where $\cH^k(\cC_0)+ \sum (r_b^0)^k + \sum (r_s^0)^k < C(n,\epsilon)$ and $\sum (r_f^0)^k\leq C(n)\epsilon$.  Let us now apply Proposition \ref{p:inductive1} to each $f$-ball $B_{r^0_f}(x^0_f)$ in order to get the new covering
\begin{align}
B_1 \subseteq \bigcup_a \big(\cN_a\cap B_{r_a}(x^1_a)\big) \cup\bigcup_a(\cC_{0,a}\cap B_{r^1_a(x^1_a)}\big)\cup \bigcup_b B_{r^1_b}(x^1_b)\cup \bigcup_s B_{r^1_s}(x^1_s)\cup \bigcup_f B_{r^1_f}(x^1_f)\, ,
\end{align}
with
\begin{align}
&\sum_a (r^1_a)^k+\sum_b (r^1_b)^k+\sum_s (r^1_s)^k\leq C(n,\epsilon)+C(n,\epsilon)\sum_f (r_f^0)^k\leq C(n,\epsilon)(1+C(n)\epsilon)\, ,\notag\\
&\sum_f (r^1_f)^k\leq \Big(C(n)\epsilon\Big)^2\, .	
\end{align}
Now we can again apply Proposition \ref{p:inductive1} to each $f$-ball $B_{r^1_f}(x^1_f)$.  Indeed, if we continue this $N$-times then we get the covering
\begin{align}
B_1 \subseteq \bigcup_a \big(\cN_a\cap B_{r^N_a}(x^N_a)\big) \cup\bigcup_a(\cC_{0,a}\cap B_{r^N_a(x^N_a)}\big)\cup \bigcup_b B_{r^N_b}(x^N_b)\cup \bigcup_s B_{r^N_s}(x^N_s)\cup \bigcup_f B_{r^N_f}(x^N_f)\, ,
\end{align}
with
\begin{align}
&\sum_a (r^N_a)^k+\sum_b (r^N_b)^k+\sum_s (r^N_s)^k\leq C(n,\epsilon)\sum_{j=0}^N\big(C(n)\epsilon\big)^j\, ,\notag\\
&\sum_f (r^N_f)^k\leq \Big(C(n)\epsilon\Big)^N\, .	
\end{align}

Now for each $N<M$ let us observe by construction that
\begin{align}
&\{B_{r^N_a}(x^N_a)\}\subseteq \{B_{r^M_a}(x^M_a)\}\, ,\;\;	\{B_{r^N_b}(x^N_b)\}\subseteq \{B_{r^M_b}(x^M_b)\}\, ,\;\; \{B_{r^N_s}(x^N_s)\}\subseteq \{B_{r^M_s}(x^M_s)\}\, ,\notag\\
&\bigcup B_{r^M_f}(x^M_f)\subseteq \bigcup B_{2r^N_f}(x^N_f)\, .
\end{align}
In particular, we can limit the $(a)$,$(b)$ and $(s)$ covers and Hausdorff limit $\cS^-\equiv \lim \{B_{r^N_f}(x^N_f)\}$ to get the covering
\begin{align}
B_1 \subseteq \bigcup_a \big(\cN_a\cap B_{r_a}(x_a)\big) \cup\bigcup_a(\cC_{0,a}\cap B_{r^N_a(x_a)}\big)\cup \bigcup_b B_{r_b}(x_b)\cup \bigcup_s B_{r_s}(x_s)\cup \cS^-\, ,	
\end{align}
with the estimates
\begin{align}
&\sum_a r_a^k+\sum_b r_b^k+\sum_s r_s^k\leq C(n,\epsilon)\sum_{j=0}^\infty\big(C(n)\epsilon\big)^j\leq  C(n,\epsilon)\, ,\notag\\
&\cH^k(\cS^-)=0\, .		
\end{align}
The last estimate follows because for each $N$ we have $\cS^-\subseteq \bigcup_f B_{2r^N_f}(x^N_f)$ with $\sum r_f^N<(C(n)\epsilon)^N\to 0$ if $\epsilon\leq \epsilon(n)$.
\end{proof}

\vspace{.5cm}

Let us now finish the proof of the Neck Decomposition:

\begin{proof}[Proof of the Neck Decomposition Theorem \ref{t:neck_decomposition}]
Our goal then is to remove the $s$-balls from Proposition \ref{p:inductive2} by iteratively applying the Proposition \ref{p:inductive2} some finite number of times.  To begin, if $\Gamma<\delta^6$ then Theorem \ref{t:neck_decomposition} follows immediately from Proposition \ref{p:inductive2}, as in this case there cannot be any $s$-balls.  Now assume we have proved Theorem \ref{t:neck_decomposition} for $\Gamma'>0$, let us now prove it holds for $\Gamma=\Gamma'+\delta^6$.  So apply Proposition \ref{p:inductive2} for $\Gamma=\Gamma'+\delta^6$ in order to get the covering
\begin{align}
	B_1 \subseteq \bigcup_a \big(\cN_a\cap B_{r_a}\big) \cup\bigcup_a(\cC_{0,a}\cap B_{r_a}\big)\cup \bigcup_b B_{r_b}(x_b)\cup\cS^- \cup \bigcup_s B_{r_s}(x_s)\, ,
\end{align}
where for each $B_{r_s}(x_s)$ we now have that $D_k(y,2r_s)<\Gamma'$.  In particular, by our inductive hypothesis we can now apply Theorem \ref{t:neck_decomposition} to these balls in order to then conclude Theorem \ref{t:neck_decomposition} for $\Gamma=\Gamma'+\delta^6$, as desired.  As the amount we increased was some definite $\delta^6$ independent of $\Gamma'$, we can for any $\Gamma>0$ simply apply this procedure $\delta^{-6}\Gamma$ times in order to then conclude Theorem \ref{t:neck_decomposition} holds for all $\Gamma$.
\end{proof}

\bibliographystyle{aomalpha}
\bibliography{Naber_ParkCity}

\end{document}